\documentclass[11pt]{amsart}
\usepackage{amsmath, amssymb, amsthm, latexsym}
\pdfoutput=1%
\usepackage[nosumlimits]{mathtools}
\usepackage{MnSymbol}
\usepackage{graphicx}
\usepackage{float}
\usepackage{subfloat}
\usepackage{subfig}
\usepackage{tikz-cd}
\usepackage{enumerate}
\usepackage{multirow}
\usepackage{array}
\usepackage{makecell}
\usepackage{tikz-cd}
\usepackage{quiver}
\usepackage{adjustbox}

\usepackage{todonotes}
\newcommand{\qi}{\mathbf{i}}
\newcommand{\qj}{\mathbf{j}}
\newcommand{\qk}{\mathbf{k}}

\newcommand{\ci}{\mathsf{i}}

\usepackage{hyperref}
\hypersetup{
colorlinks=true
,breaklinks=true
,urlcolor=magenta
,linkcolor=red
,bookmarksopen=false
,citecolor=mediumblue
,anchorcolor=mediumblue
}
\usepackage{color,svgcolor}

\usepackage{stackengine}
\usepackage{chemformula} 
\usepackage[margin=1in]{geometry}
\usepackage{algorithm}
\usepackage{algpseudocode}
\usepackage{float}
\usepackage{url}

\usepackage{subfloat}
\usepackage{subfig}
\newcommand{\tp}{{\scriptscriptstyle\mathsf{T}}}
\newcommand{\fp}{{\scriptscriptstyle\mathsf{f}}}

\newcommand\Item[1][]{%
  \ifx\relax#1\relax  \item \else \item[#1] \fi
  \abovedisplayskip=0pt\abovedisplayshortskip=0pt~\vspace*{-\baselineskip}}
\newcommand\xrowht[2][0]{\addstackgap[.5\dimexpr#2\relax]{\vphantom{#1}}}

\theoremstyle{plain}
\newtheorem{theorem}{Theorem}[section]
\newtheorem{proposition}[theorem]{Proposition}
\newtheorem{lemma}[theorem]{Lemma}
\newtheorem{corollary}[theorem]{Corollary}

\newtheorem{problem}[theorem]{Problem}

\theoremstyle{definition}
\newtheorem{definition}[theorem]{Definition}
\newtheorem{example}[theorem]{Example}

\theoremstyle{remark}
\newtheorem{remark}[theorem]{Remark}

    \newtheoremstyle{TheoremNum}
        {\topsep}{\topsep}              
        {\itshape}                      
        {}                              
        {\bfseries}                     
        {.}                             
        { }                             
        {\thmname{#1}\thmnote{ \bfseries #3}}
    \theoremstyle{TheoremNum}
    \newtheorem{thmn}{Theorem}

\let\O\undefined

\let\H\undefined
\DeclareMathOperator{\O}{O}
\DeclareMathOperator{\U}{U}

\DeclareMathOperator{\rank}{rank}

\DeclareMathOperator{\H}{H}
\DeclareMathOperator{\V}{V}

\DeclareMathOperator{\Id}{Id}
\DeclareMathOperator{\GL}{GL}
\DeclareMathOperator{\SO}{SO}

\DeclareMathOperator{\Aut}{Aut}

\DeclareMathOperator{\Spin}{Spin}
\DeclareMathOperator{\SU}{SU}
\DeclareMathOperator{\Sp}{Sp}

\DeclareMathOperator{\diag}{diag}
\DeclareMathOperator{\SE}{SE}
\DeclareMathOperator{\Rat}{Rat}

\DeclareMathOperator{\Flag}{Flag}

\DeclareMathOperator{\ISO}{ISO}

\DeclareMathOperator{\Stab}{Stab}

\DeclareMathOperator{\Conf}{Conf}

\begin{document}
\title{Rational curves on real classical groups}

\author[Z.~Li]{ZiJia Li}
\address{KLMM, Academy of Mathematics and Systems Science, Chinese Academy of Sciences, Beijing 100190, China}
\email{lizijia@amss.ac.cn}
\author[K.~Ye]{Ke Ye}
\address{KLMM, Academy of Mathematics and Systems Science, Chinese Academy of Sciences, Beijing 100190, China}
\email{keyk@amss.ac.cn}

\date{}

\begin{abstract}
This paper is concerned with rational curves on real classical groups.  Our contributions are three-fold: (i) We determine the structure of quadratic rational curves on real classical groups.  As a consequence,  we completely classify quadratic rational curves on $\U_n$,  $\O_n(\mathbb{R})$,  $\O_{n-1,1}(\mathbb{R})$ and $\O_{n-2,2}(\mathbb{R})$.  (ii) We prove a decomposition theorem for rational curves on real classical groups,  which can be regarded as a non-commutative generalization of the fundamental theorem of algebra and partial fraction decomposition.  (iii) As an application of (i) and (ii),  we generalize Kempe's Universality Theorem to rational curves on homogeneous spaces.
\end{abstract}

\maketitle

\section{Introduction}
Rational curves are ubiquitous in both pure and applied mathematics.  On the one hand,  rational curves are indispensable in modern algebraic geometry \cite{Kollar96}.  They provide essential tools in the study of the minimal model program \cite{Kollar92,Mori1988flip}, rational and unirational varieties \cite{Graber03families}, Fano varieties \cite{Kollar92fano}, etc.  In real algebraic geometry,  rational curves also take on a central role in various enumerative problems \cite{KR15,KR17,Mikhalkin17,Sottile00,Welschinger05}.  On the other hand,  there are numerous applications of rational curves in engineering practice.  It is a long established method in kinematics to parametrize and analyze motions by rational curves \cite{Roth79tk,Darboux1890}.  In computer-aided geometric design,  rational curves are imperative to an efficient modelling of 3D objects \cite{Farin97, Juttler93zwanglaufige, Winkler08}.

For a group,  decomposing its elements into the product of special ones is a classical technique to study problems associated with the group.  For example,  the fundamental theorem of finitely generated Abelian groups completely determines the structure of such a group; Levi-Mal'tsev decomposition \cite{Malcev42} reveals the structure of a general Lie group; Iwasawa decomposition \cite{Iwasawa49} plays a crucial role in understanding representation theory of a semi-simple Lie group; Bruhat decomposition \cite{Chevalley58classification} provides a cellular decomposition of a complete flag manifold.  

The subject of this paper lies at the intersection of the two aforementioned active research fields.  Namely,  we investigate the decomposition of rational curves on real algebraic groups,  and as an application we prove a generalization of the celebrated Kempe's Universality theorem \cite{Kempe75}.  In the rest of this section,  we summarize the main contributions of the paper.
\subsection*{Contribution I: classification of quadratic rational curves} Let $\mathbb{F} = \mathbb{R}$,  $\mathbb{C}$ or $\mathbb{H}$.  Given $B\in \GL_n(\mathbb{F})$ such that $B^\sigma = \pm B$ where $\sigma$ is the transpose or conjugate transpose of matrices,  we define a real algebraic group 
\[
G_B(\mathbb{F}) \coloneqq \lbrace 
X \in \mathbb{F}^{n\times n}: X B X^{\sigma} = B
\rbrace.
\]
By varying choices of $\mathbb{F}$,  $B$ and $\sigma$,  we obtain classical matrix groups extensively studied in the literature \cite{DPW83, Milnor69,  Weyl97}. 

By definition,  a rational curve on $G_B(\mathbb{F})$ is a morphism $\gamma: \mathbb{P}_{\mathbb{R}}^1 \to G_B(\mathbb{F})$ between real algebraic varieties.  If we denote by $\deg(\gamma)$ the degree of a rational curve $\gamma$,  then $\deg (\gamma)$ must be even.  Thus,  the minimal degree of a non-constant rational curve is two.  The first problem we will address in this paper is the classification of these simplest curves on $G_B(\mathbb{F})$. 
\begin{thmn}[\ref{thm:structure} \normalfont{(Structure theorem)}]
For any quadratic rational curve $\alpha$ on $G_B(\mathbb{F})$,  there exist $R\in \GL_n(\mathbb{F})$,  $a\in \mathbb{R}$ and $b\in \mathbb{R} \setminus \{0\}$ such that $B  = R \diag(B_1,\dots, B_s) R^{\sigma}$ and 
\begin{equation}\label{eq:structureGBF}
\alpha(t) = R\left( \frac{ \left( t - a \right)^2  I_{n} + b(t-a)\diag(X_1,\dots, X_s) + b^2(Y_{pq})_{p,q=1}^s}{ \left(t - a \right)^2 + b^2} \right)R^{-1},   
\end{equation}
where for each $1 \le p,  q \le s$,  $(X_p , B_p)$ is given in Table~\ref{Tab:indecomposable} and $Y_{pq}$ is given in Tables~\ref{Tab:CandidatesYij} and \ref{Tab:CandidatesYijcon'd}.
\end{thmn}
The proof of Theorem~\ref{thm:structure} heavily relies on the classification of orbits of the adjoint representation of $G_B(\mathbb{F})$ \cite{DPW83,Milnor69}.  To our surprise,  it turns out that the proof breaks down into solving Sylvester equations with structured coefficient matrices,  which are comprehensively studied in control theory and operator theory \cite{HR84, TSH01}.  

It is obviously not true that any curve parametrized as in \eqref{eq:structureGBF} lies on $G_B(\mathbb{F})$.  However,  using the structure determined by Theorem~\ref{thm:structure},  we obtain a complete classification of quadratic rational curves on $\U_n$ (cf.~Theorem~\ref{thm:classification U_n}),  $\O_n(\mathbb{R})$ (cf. ~Theorem~\ref{thm:classification O_n}),  $\O_{n-1,1}(\mathbb{R})$ (cf. ~Theorem~\ref{thm:classification O_n1}) and $\O_{n-2,2}(\mathbb{R})$ (cf.~Theorem~\ref{thm:classification O_n2}),  which are arguably the most important matrix groups for applications in physics and kinematics \cite{Cornwell97,  Dirac36, Dorst2016construction,  Kalkan2022study,  Schottenloher08}.  Our results are analogues of the classification of low degree planar algebraic curves intensively studied in the past three centuries \cite{KW05, Newton03,  Noether82}.  
\subsection*{Contribution II: decomposition of rational curves}
Given rational curves $\gamma_1$ and $\gamma_2$ on $G_B(\mathbb{F})$,  we have 
\[
\deg (\gamma_1 \gamma_2) \le \deg (\gamma_1) + \deg (\gamma_2).
\]
This observation together with the fundamental theorem of algebra and its various generalizations \cite{EN44, Gentili21, Jou50} motivates us to consider the decomposition problem of rational curves on $G_B(\mathbb{F})$.
\begin{thmn}[\ref{thm:KempeGBF} \normalfont{(Decomposition theorem)}]
If $\gamma(t)$ is a degree $d$ rational curve on $G_B(\mathbb{F})$ with poles of multiplicities $s_1,\dots, s_l$, then $\gamma (t) = \beta_1(t) \cdots \beta_l(t)$ for some  rational curves $\beta_1(t),\dots, \beta_l(t)$ of degrees $2s_1,\dots, 2s_l$ respectively. In particular, if all the poles of $\gamma(t)$ are simple, then $\gamma(t)$ can be decomposed into a product of $d$ quadratic rational curves.
\end{thmn}
The proof of Theorem~\ref{thm:KempeGBF} proceeds by induction on $d$.  It is based on the observation that $\deg (\gamma) = \deg(\gamma^{-1})$ (cf.  Proposotion~\ref{prop:inverse}) for any rational curve $\gamma$ on $G_B(\mathbb{F})$.  We first deal with the case $\mathbb{F} = \mathbb{R}$ and then discuss cases $\mathbb{F} =\mathbb{C}$ and $\mathbb{H}$ by embedding them into $\mathbb{R}^{2\times 2}$ and $\mathbb{R}^{4 \times 4}$,  respectively.  As a consequence of Theorem~\ref{thm:KempeGBF},  we obtain the decomposition theorem for rational curves on inhomogeneous indefinite-orthogonal groups $\ISO_{p,n-p}^+(\mathbb{R})$ (cf. ~Theorem~\ref{thm:KempeSE}),  which are of great importance in the gauge theory of gravitation \cite{CDN83, Rosen68}.

On the one side,  we notice that rational curves on $G_B(\mathbb{F})$ are matrix-valued rational functions.  Assorted decompositions of matrix-valued functions are discussed in the literature.  Examples include the Birkhoff decomposition \cite{Birkhoff09},  minimal decomposition \cite{BGKR08},  unitary decomposition \cite{Gohberg88} and J-expansive decomposition \cite{Potapov55}.  The decomposition in Theorem~\ref{thm:KempeGBF} can be recognized as an analogy of these decompositions of matrix-valued functions.  Moreover,  we remark that Theorem~\ref{thm:KempeGBF} is a multiplicative and non-commutative generalization of the partial fraction decomposition of rational functions.  Indeed, a rational curve on the additive group $\mathbb{R}$ is a rational function $F(t)$.  In particular,  it can be decomposed as $F(t) = \sum_{j=1}^r p_j(t)/q^{s_j}_j(t)$,  where $s_j\ge 0$ is an integer and $q_j(t)$ is an irreducible quadratic real polynomial for each $1 \le j \le r$.  On the other side,  if we consider the group scheme $\mathcal{G}$ defined by equation $X B X^\sigma = B$,  then $G_B(\mathbb{F})$ consists of $\mathbb{R}$-points of $\mathcal{G}$ and rational curves on $G_B(\mathbb{F})$ are $R$-points of $\mathcal{G}$,  where $R$ is the ring of regular functions on $\mathbb{P}_{\mathbb{R}}^1$.  Bearing this perspective in mind,  Theorem~\ref{thm:KempeGBF} clearly shares a resemblance with renowned decomposition theorems such as Cartan–Dieudonn\'e theorem \cite{Cartan81},  Gauss decomposition theorem  \cite{Zhelobenko70} and Bruhat decomposition theorem \cite{Chevalley58classification}.  
\subsection*{Contribution III: generalized Kempe's Universality Theorem}
Since its first appearance in late 1870s,  Kempe's Universality Theorem \cite{Kempe75} stands as a cornerstone of theoretical mechanism science.  It asserts that any bounded plane algebraic curve can be faithfully reproduced by a mechanical linkage using only rotational joints.  It captivates researchers for its elegant solution and profound theoretical implications \cite{Abbott08, Demaine07, Gao01}.  Recently,  Kempe's Universality Theorem sparks renewed interest among mathematicians and computer scientists,  leading to further exploration and generalizations of the problem.  By leveraging the geometry of the configuration space,  Kempe's Universality Theorem is generalized for algebraic curves in Euclidean space of arbitrary dimension \cite{Abbott08}.  Following the topological reformulation given by Thurston,  the theorem can be generalized along the direction of moduli space of geometric objects \cite{KM96, KM02, Kourganoff16}.  By encoding 2D and 3D motions via polynomials over
non-commutative algebras,  Kempe's Universality Theorem is equivalent to the factorization problem of motion polynomials \cite{GKLRSV17, LJS18}.

According to the Erlangen program \cite{Klein1893},  geometries of a manifold are governed by their transformation groups.  This underlies our generalization of Kempe's Universality Theorem. 
\begin{thmn}[\ref{thm:Kempe G-space} (\normalfont{Generalized Kempe's Universality Theorem II})]
Let $(G,X)$ be one of the $9$ pairs listed in Theorem~\ref{thm:approx-lift}.  For every rational curve $\gamma$ on $X$ with $\gamma(0) = x_0$,  there exist rational curves $\alpha_1,\dots,  \alpha_s$ on $G$ with $\alpha_1(0) = \cdots = \alpha_s(0) = I$ such that  
\begin{enumerate}[(a)]
\item Each $ \alpha_j$ only has poles at $\{c_j,\overline{c}_j\}$,  $1 \le j \le s$.  
\item If $j \ne k$ then $\{c_j,\overline{c}_j\} \ne \{c_k,\overline{c}_k\}$. 
\item $\prod_{j=1}^s \alpha(t) x_0  = \gamma(t)$. 
\end{enumerate}
Here $I$ denotes the identity matrix in $G$.
\end{thmn}
The proof of Theorem~\ref{thm:Kempe G-space} relies on the criterion \cite[Satz~3.3]{Harder67} for the triviality of a principal bundle on a smooth curve and Theorems~\ref{thm:KempeGBF} and \ref{thm:KempeSE}.  Using the approximation theorem \cite[Theorem~1.1]{BK22},  we also obtain a generalization of Kempe's Universality Theorem for continuous loops (cf. Theorem~\ref{thm:approx-lift}).

Essentially,  rational curves $\alpha_1,\dots,  \alpha_s$ and the action of $G$ on $X$ in Theorem~\ref{thm:Kempe G-space} play the role of rotational joints and the realization of linkage in the original Kempe's University Theorem and its existing generalizations \cite{Abbott08, Gao01, KM96, KM02,Kempe75, Kourganoff16},  respectively.  In fact,  if we let $(G,X) = (\SE_2(\mathbb{R}),  \mathbb{R}^2)$ (resp.  $(G,X) = (\SE_3(\mathbb{R}),  \mathbb{R}^3)$) in Theorem~\ref{thm:Kempe G-space},  then we obtain the version of Kempe's Universality Theorem for planar (resp.  space) curve proved by motion polynomials in \cite{GKLRSV17} (resp.  \cite{LJS18}).  
\subsection*{Organization of the paper}
In Section~\ref{sec:pre}, we fix notations and review some results from topology and algebraic geometry.  We investigate in Section~\ref{sec:rat} basic properties of rational curves on real algebraic varieties. Section~\ref{sec:quad} is devoted to the classification of quadratic rational curves on \texorpdfstring{$G_B(\mathbb{F})$}{GB}.  To avoid distracting the reader by lengthy calculations,  we defer the proofs of Lemmas~\ref{lem:lower triangular} and \ref{lem:Yij}  and Theorem~\ref{thm:classification O_n2} to Appendices~\ref{append:proof of lem:lower triangular}, \ref{append:proof lem:Yij} and \ref{append:proof of O_n2} respectively.  We address in Section~\ref{sec:dec} the decomposition problem for rational curves on real linear algebraic groups.  In Section~\ref{sec:gekempe},  we apply topological and rational lifting criteria,  together with results in Section~\ref{sec:dec},  to obtain two generalizations of Kempe's Universality Theorem.  This section ends with a brief discussion on examples of small dimensions,  which may of particular interest to the reader with background in geometric algebra or theoretical mechanism.
\section{Preliminaries}\label{sec:pre}
\subsection*{Linear algebraic groups}
Let $n > p$ be positive integers and let $\mathbb{F} = \mathbb{R}$,  $\mathbb{C}$ or $\mathbb{H}$.  We denote by $\GL_n(\mathbb{F})$ the group of $n\times n$ invertible matrices over $\mathbb{F}$.  Classical subgroups of $\GL_n(\mathbb{F})$ are 
\begin{align*}
\O_n(\mathbb{R}) &\coloneqq \lbrace 
X \in \GL_n(\mathbb{R}): X^\tp X = I_n
\rbrace, \\
\SO_n(\mathbb{R}) &\coloneqq \lbrace 
X \in \O_n(\mathbb{R}):\det(X) = 1
\rbrace, \\
\U_n &\coloneqq \lbrace 
X \in \GL_n(\mathbb{C}): X^\ast X = I_n
\rbrace, \\
\SU_n &\coloneqq \lbrace 
X \in \U_n: \det(X) = 1
\rbrace,  \\
\Sp_{2n}(\mathbb{R}) &\coloneqq \left\lbrace 
X \in \GL_{2n}(\mathbb{R}): X \begin{bsmallmatrix}
0 & I_n \\
- I_n & 0 
\end{bsmallmatrix} X^\tp = \begin{bsmallmatrix}
0 & I_n \\
- I_n & 0 
\end{bsmallmatrix}
\right\rbrace.
\end{align*} 
Our discussions in the sequel will also involve indefinite orthogonal groups and their inhomogeneous version.  Let $I_{p,n-p} \coloneqq \diag(
I_p, -I_{n-p})$.  We define:  
\begin{align*} 
\O_{p,n-p}(\mathbb{R}) &\coloneqq \lbrace 
X \in \GL_n(\mathbb{R}): X^\tp I_{p,n-p} X = I_{p,n-p}
\rbrace, \\
\SO^+_{p,n-p}(\mathbb{R}) &\coloneqq \text{the identity component of $\O_{p,n-p}(\mathbb{R})$}, \\
\SE_{n}(\mathbb{R}) & \coloneqq
 \mathbb{R}^n \ltimes  \O_n(\mathbb{R}) = 
\left\lbrace
\begin{bsmallmatrix}
Q & u \\
0 & 1
\end{bsmallmatrix} \in \GL_{n+1}(\mathbb{R}): Q\in \O_{n}(\mathbb{R}),  u \in \mathbb{R}^n
\right\rbrace,  \\
\ISO^+_{p,n-p}(\mathbb{R}) &\coloneqq \mathbb{R}^{n} \ltimes \SO^+_{p,n-p}(\mathbb{R}) = 
\left\lbrace
\begin{bsmallmatrix}
Q & u \\
0 & 1
\end{bsmallmatrix} \in \GL_{n+1}(\mathbb{R}): Q\in \SO^+_{p,n-p}(\mathbb{R}),  u \in \mathbb{R}^{n}
\right\rbrace.
\end{align*}
\subsection*{Topology and Geometry}
Let  $X, Y$ be topological spaces and let $p: Y \to X$,  $\gamma: \mathbb{S}^1 \to X$ be continuous maps.  A \emph{lift} of $\gamma$ is a continuous map $\beta:\mathbb{S}^1 \to Y$ such that $p \circ \beta = \gamma$. 
\begin{lemma}\cite[Lemma~55.3]{Munkres00}\label{lem:null homotopic}
Let $X$ be a topological space.  A continuous map $\gamma: \mathbb{S}^1 \to X$ is homotopic to a constant map if and only if $[\gamma] = 0 \in \pi_1(X)$.
\end{lemma}
Topological lifting criteria like \cite[Proposition~1.33]{Hatcher02} indicates that the existence of a lift of $\gamma$ is controlled by $[\gamma] \in \pi_1 (X)$.  However,  if $\gamma$ is a rational curve,  then there is no guarantee that the lift of $\gamma$,  if it exists,  is also a rational curve.  We will need the following lifting criterion for algebraic curves,  which is a consequence of \cite[Satz~3.3]{Harder67},  see also \cite{BF19, BS68, Steinberg65}.
\begin{proposition}\label{prop:rational triviality}
Let $k$ be a field (not necessarily algebraically closed) and let $C$ be a smooth affine curve over $k$.  If $G$ is a semisimple and simply connected algebraic group,  then every generically trivial principal $G$-bundle on $C$ is trivial. 
\end{proposition}

What follows is a topological criterion for the existence of a regular approximation of a continuous map between real algebraic varieties.
\begin{theorem}\cite[Theorem~1.1]{BK22}\label{thm:approx by regular maps}
Let $X$ be a real algebraic variety and let $Y$ be a homogeneous space for some linear algebraic group.  A continuous map $f: X\to Y$ can be approximated by regular maps in the compact-open topology if and only if $f$ is homotopic to a regular map.
\end{theorem}

\section{Rational curves on real algebraic varieties}\label{sec:rat}
We begin with the definition of rational curves on a real algebraic variety.
\begin{definition}[rational curve]\label{def: curve}
Let $X$ be a real quasi-affine variety. A \emph{rational curve} on $X$ is a morphism $\gamma: \mathbb{P}_{\mathbb{R}}^1 \to X$. We denote by $\Rat(X)$ the set of rational curves on $X$. Given $x_0\in X$, we also denote 
\[
\Rat(X,x_0) \coloneqq \lbrace
\gamma \in \Rat(X): \gamma \left( [0 : 1 ] \right) = x_0
\rbrace.
\]
\end{definition}
Since we have an identification $\mathbb{P}_{\mathbb{R}}^1 \simeq \mathbb{R}^1 \sqcup \{ \infty\}$,  rational curves can be characterized alternatively. 
\begin{lemma}\label{lem:equiv_def}
Let $X \subseteq \mathbb{R}^N$ be a real quasi-affine variety. The following are equivalent:
\begin{enumerate}[(i)]
\item\label{lem:equiv_def:item1} $\gamma$ is a rational curve on $X$.
\item\label{lem:equiv_def:item3}  $\gamma(t)$ is an everywhere defined $X$-valued rational function on $\mathbb{R}\sqcup \{\pm \infty\}$ such that $\gamma(+ \infty) = \gamma(-\infty) \in X$.
\item\label{lem:equiv_def:item4}
$\gamma(t) = (p_1(t)/q(t),\dots, p_N(t)/q(t)): \mathbb{R} \to X \subseteq \mathbb{R}^N$ and $\gamma(+ \infty) = \gamma(-\infty) \in X$ where $p_1,\dots, p_N,q$ are univariate real polynomials such that $q$ has no real root and $\gcd(p_1,\dots, p_N,q) = 1$.
\end{enumerate}
\end{lemma}
\begin{remark}
It is worth remarking that over an arbitrary field $k$, a rational curve on a quasi-projective variety $X$ is defined \cite[Chapter~II]{Kollar96} as a morphism from $\mathbb{P}_{k}^1$ to $X$. For $k = \mathbb{C}$, there is no non-constant rational curve on a quasi-affine variety. However, non-constant rational curves may exist on real quasi-affine varieties.  It is also noticeable that every real projective variety is isomorphic to a real affine variety \cite[Proposition~2.4.1]{AK92}.  This fact indicates that over $\mathbb{R}$, it is sufficient to consider rational curves on quasi-affine varieties.
\end{remark}
\begin{definition}[degree]
Let $\gamma (t) = (p_1(t)/q(t), \dots, p_N(t)/q(t) )$ be a rational curve on $X$ with $\gcd(p_1,\dots, p_N,q) = 1$.  The \emph{degree} of $\gamma$ is $\deg(\gamma) \coloneqq \deg (q)$. The set of rational curves of degree $d$ on $X$ is denoted by $\Rat_d(X)$. Moreover, if $x_0 \in X$ is a fixed point, we denote
\[
\Rat_d(X,x_0) \coloneqq \lbrace
\gamma \in \Rat_d(X): \gamma(\infty) = x_0 \rbrace.
\]
\end{definition}
\begin{remark}
Since $q$ has no real root, $\deg(\gamma)$ must be an even non-negative integer. 
\end{remark}

\begin{lemma}\label{lem:induced action}
Let $X$ be a real quasi-affine variety and let $G$ be a real algebraic group acting on $X$. For any $x_0 \in X$ and $g\in G$, the map 
\[
L_g: \Rat(X,x_0) \to \Rat (X,g x_0),\quad \gamma \mapsto  g \gamma
\]
is bijective. In particular, if $G$ acts on $X$ transitively, then there is a bijection between $\Rat (X)$ and $\Rat (X,x_0) \times X$.
\end{lemma}

Let $G \subseteq \GL_n(\mathbb{R})$ be a real linear algebraic group. According to Lemma~\ref{lem:induced action}, we have $\Rat (G) = \Rat (G,I_n) \times G$ where $I_n \in G$ is the identity matrix. By Lemma~\ref{lem:equiv_def}, a curve $\gamma \in \Rat (G,I_n)$ admits a unique parametrization: 
\[
\gamma(t) = (P_{ij}(t)/q(t))_{i,j=1}^n,
\]
where $q\in \mathbb{R}[t]$ and $P_{ij} \in \mathbb{R}[t]$ satisfy
\begin{itemize}
\item $\gamma(t_0) \in G$ for any $t_0\in \mathbb{R}$.
\item $\gcd(q,P_{11},\dots, P_{nn}) = 1$.
\item $q,P_{11},\dots, P_{nn}$ are monic.
\item $\deg(q)  = \deg(P_{ii}) > \deg(P_{ij})$ for $1\le i \ne j \le n$.
\item $q$ has no real roots.
\end{itemize}

\begin{lemma}
If $\gamma(t)$ is a  rational curve on $G$, then it is also a rational curve on the connected component $G_0$ of $G$.
\end{lemma}
Let $\mathbb{F}$ be $\mathbb{R},\mathbb{C}$ or $\mathbb{H}$ and let $\sigma: \mathbb{F}^{n\times n} \to \mathbb{F}^{n\times n}$ be an \emph{$\mathbb{R}$-involution} on $\mathbb{F}^{n\times n} $, i.e., $\sigma$ is an $\mathbb{R}$-linear map satisfying
\[
\sigma(I_n) = I_n,\quad \sigma(\sigma(A)) = A,\quad
\sigma(AB) = \sigma(B) \sigma(A),\quad A,B\in \mathbb{F}^{n\times n}.
\]
For each $X\in \mathbb{F}^{n\times n}$, we denote $X^\sigma \coloneqq \sigma(X)$. A typical example of an involution is the transpose (resp. conjugate transpose) of matrices in $\mathbb{R}^{n\times n}$ (resp. $\mathbb{C}^{n\times n}$ or $\mathbb{H}^{n\times n}$). Given an involution $\sigma$ on $\mathbb{F}^{n\times n}$ and $B\in \GL_n(\mathbb{F})$, we define 
\[
G_B(\mathbb{F}) \coloneqq \lbrace 
X \in \mathbb{F}^{n\times n}: X B X^{\sigma} = B
\rbrace.
\]
By definition, $G_B(\mathbb{F})$ is a real algebraic subgroup of $\GL_n(\mathbb{F})$, whose Lie algebra is 
\begin{equation}
\mathfrak{g}_B(\mathbb{F}) \coloneqq \lbrace 
Y\in \mathbb{F}^{n\times n}: B Y^\sigma + Y B   = 0
\rbrace.
\end{equation}
Familiar examples of $G_B(\mathbb{F})$ include:  
\begin{enumerate}[(a)]
\item $\mathbb{F} = \mathbb{R}$ (resp. $\mathbb{F} = \mathbb{C}$), $B = I_{p,n-p} \coloneqq \diag (I_p, -I_{n-p})$, $\sigma = \text{transpose}$: $G_B(\mathbb{F})$ is the indefinite orthogonal $\O_{p,n-p}(\mathbb{R})$ (resp. $\O_{p,n-p}(\mathbb{C}) \simeq \O_n(\mathbb{C})$) of type $(p,n-p)$. In particular, if $p = n$, then $G_B(\mathbb{F})$ is the orthogonal group $\O_n(\mathbb{R})$ (resp. $\O_{n}(\mathbb{C})$).
\item $\mathbb{F} = \mathbb{R}$ (resp. $\mathbb{F} = \mathbb{C}$), $B = \begin{bsmallmatrix}
0 & I_n \\
-I_n & 0
\end{bsmallmatrix}$, $\sigma = \text{transpose}$: $G_B(\mathbb{F}) $ is the symplectic group $\Sp_{2n}(\mathbb{R})$ (resp. $\Sp_{2n}(\mathbb{C})$).
\item $\mathbb{F} = \mathbb{C}, B = I_{p,n-p}$, $\sigma = \text{conjugate transpose}$: $G_B(\mathbb{F})$ is the indefinite unitary group $\U_{n,n-p}$ of type $(p,n-p)$.
\item $\mathbb{F} = \mathbb{H}, B = I_{p,n-p}$, $\sigma = \text{conjugate transpose}$: $G_B(\mathbb{F})$ is the quaternionic indefinite symplectic group $\Sp_{p,n-p}(\mathbb{H})$ of type $(p,n-p)$. 
\end{enumerate}
The lemma that follows is a well-known fact.  Nonetheless,  we provide a proof due to the lack of appropriate reference. 
\begin{lemma}\label{lem:involution}
Let $\sigma$ be an $\mathbb{R}$-involution on $\mathbb{F}^{n\times n}$. Then there exists $C\in \GL_n(\mathbb{F})$ such that for all $A\in \mathbb{F}^{n\times n}$ we have $A^\sigma = C A^\xi C^{-1}$ where 
\[
A^\xi = 
\begin{cases}
A^\tp,&~\text{if}~\mathbb{F} = \mathbb{R} \\
A^\tp~\text{or}~A^\ast,&~\text{if}~\mathbb{F} = \mathbb{C} \\
A^\ast,&~\text{if}~\mathbb{F} = \mathbb{H}
\end{cases}.
\]
\end{lemma}
\begin{proof}
We denote by $\Aut_{\mathbb{R}}(\mathbb{F}^{n\times n})$ the automorphism group of $\mathbb{F}^{n\times n}$ as an $\mathbb{R}$-algebra. We consider the map $\varphi: \mathbb{F}^{n\times n} \to \mathbb{F}^{n\times n}$ defined by $\varphi(A) = (A^\sigma)^\xi$. Since bot $\sigma$ and $\xi$ are $\mathbb{R}$-involutions, $\varphi$ lies in $\Aut_{\mathbb{R}}(\mathbb{F}^{n\times n})$. Let $Z(\mathbb{F}^{n\times n})$ be the center of $\mathbb{F}^{n\times n}$. It is straightforward to verify that 
\[
Z(\mathbb{F}^{n\times n}) = \begin{cases}
\mathbb{R} I_n,&~\text{if}~\mathbb{F} = \mathbb{R}~\text{or}~\mathbb{F} \\
\mathbb{C} I_n,&~\text{if}~\mathbb{F} = \mathbb{C}\\
\end{cases}.
\]

If $\mathbb{F} = \mathbb{R}$ or $\mathbb{H}$, $\mathbb{F}^{n\times n}$ is a simple central algebra, thus by Skolem-Noetherm theorem \cite{Kersten90} each element in $\Aut_{\mathbb{R}}(\mathbb{F}^{n\times n})$ is an inner automorphism. Hence $A^\sigma$ can be written in the desired form for some $C\in \GL_n(\mathbb{F})$.

For $\mathbb{F} = \mathbb{C}$, we observe that $\varphi(A) \varphi(B) = \varphi(AB) = \varphi(BA) = \varphi(B) \varphi(A)$ if $A,B$ commute. Therefore, $\varphi$ preserves $Z(\mathbb{C}^{n\times n})$. Let $\psi_0: \mathbb{C} \to \mathbb{C}$ be the restriction of $\varphi$ onto $Z(\mathbb{C}^{n\times n}) \simeq \mathbb{C}$ and let $\psi: \mathbb{C}^{n\times n} \to \mathbb{C}^{n\times n}$ be the map component-wise induced by $\psi_0$. Clearly, $\psi$ is an automorphism of $\mathbb{C}^{n\times n}$ as an $\mathbb{R}$-algebra. By construction, the map $\varphi\circ \psi^{-1}$ is an automorphism of $\mathbb{C}^{n\times n}$ as a $\mathbb{C}$-algebra. Skolem-Noetherm theorem implies that $\varphi\circ \psi^{-1}$ is an inner automorphism on $\mathbb{C}^{n\times n}$ as a $\mathbb{C}$-algebra. Lastly, since $\psi_0$ is an automorphism of $\mathbb{C}$ as an $\mathbb{R}$-algebra, $\psi_0$ is either the identity map or the complex conjugation and this completes the proof.
\end{proof}
According to Lemma~\ref{lem:involution},  it is sufficient to assume that $\sigma$ is either the transpose or the conjugate transpose.  We conclude this section by an observation that is essential to our discussion in Section~\ref{sec:dec}.
\begin{proposition}[Inverse]\label{prop:inverse}
If $\gamma \in \Rat_{2d}(G_B(\mathbb{F}),I_n)$, then $\gamma(t)^{-1} \in \Rat_{2d}(G_B(\mathbb{F}),I_n)$ and it has the same poles as $\gamma(t)$.
\end{proposition}
\begin{proof}
Since $\gamma(t)$ is a curve on $G_B(\mathbb{F})$, we have $\gamma(t) B \gamma(t)^\sigma = B$. Thus $\gamma(t)^{-1} = B \gamma(t)^\sigma B^{-1}$, which is a  rational curve on $G_B(\mathbb{F})$ of degree $2d$ whose poles are the same as those of $\gamma(t)$.
\end{proof}
\begin{remark}
For a general real linear algebraic group $G \subseteq \GL_n(\mathbb{R})$, it may happen that $\deg (\gamma^{-1}) \ne \deg(\gamma)$ if $\gamma \in \Rat(G,I_n)$. For instance, we consider $G = \GL_2(\mathbb{R})$ and 
\[
\gamma(t) = \begin{bsmallmatrix}
1 & \frac{t}{t^2 + 1} \\
\frac{t}{t^2 + 1}  & 1
\end{bsmallmatrix}.
\]
Clearly, $\gamma(t)$ is a  rational curve on $\GL_2(\mathbb{R})$ since $\det(\gamma(t)) = (t^4 + t^2 + 1)/(t^2 + 1)^2$. However, a direct calculation implies 
\[
\gamma(t)^{-1} = \frac{(t^2 + 1)^2}{t^4 + t^2 + 1} \begin{bsmallmatrix}
1 & -\frac{t}{t^2 + 1} \\
-\frac{t}{t^2 + 1}  & 1
\end{bsmallmatrix} =  \begin{bsmallmatrix}
\frac{(t^2 + 1)^2}{t^4 + t^2 + 1} & -\frac{(t^2 + 1)t}{t^4 + t^2 + 1} \\[3pt]
-\frac{(t^2 + 1)t}{t^4 + t^2 + 1}   & \frac{(t^2 + 1)^2}{t^4 + t^2 + 1}
\end{bsmallmatrix}
\]
and $\deg(\gamma^{-1}) = 4 > 2 = \deg(\gamma)$. Moreover,  poles of $\gamma(t)^{-1}$ are different from those of $\gamma(t)$.
\end{remark}
\section{Quadratic  rational curves on \texorpdfstring{$G_B(\mathbb{F})$}{GB}}\label{sec:quad}
Let $\alpha$ be a quadratic  rational curve on $G_B(\mathbb{F})$.  If poles of $\alpha$ are $x \pm y \ci$ where $(x,y) \in \mathbb{R} \times (\mathbb{R}\setminus \{0\})$, then clearly $ \widetilde{\alpha}(t) \coloneqq \alpha \left( x + y t \right)$ is a quadratic  rational curve on $G_B(\mathbb{F})$ with poles $\pm \ci$, where $\ci$ is the complex unit with $\ci^2=-1$. Therefore, there is no loss of generality to assume that poles of $\alpha$  are $\pm \ci$. We write 
\[
\alpha (t) = \frac{ I_n t^2 + A_1 t + A_0}{t^2 + 1}
\]
for some $A_1, A_0\in \mathbb{F}^{n\times n}$.  Since $\alpha \in \Rat_2(G_B(\mathbb{F}), I_n)$,  we may derive 
\begin{align}
B A_1^{\sigma} + A_1 B &= 0, \label{eq:quadratic1} \\ 
B A_0^{\sigma}  + A_0 B   &= (2I_n + A_1^2) B, \label{eq:quadratic2} \\
A_1 B A_0^{\sigma}   &= A_0   A_1 B , \label{eq:quadratic3} \\
A_0 B A_0^{\sigma} &= B. \label{eq:quadratic4}
\end{align}
by comparing coefficients in the equation $(I_n t^2 + A_1 t + A_0) B (I_n t^2 + A_1 t + A_0)^{\sigma} = (t^2 + 1)^2 B$.
\begin{remark}
We notice that \eqref{eq:quadratic1} and \eqref{eq:quadratic4} are equivalent to the condition $(A_0,A_1) \in G_B(\mathbb{F}) \times \mathfrak{g}_B(\mathbb{F})$.
\end{remark}

The lemma that follows characterizes the invariance of a solution of \eqref{eq:quadratic1}-\eqref{eq:quadratic4} with respect to the action of $\GL_n(\mathbb{F}) \times Z(\mathbb{F})^{\times}$ and $G_B(\mathbb{F})$, respectively.
\begin{lemma}\label{lem:invariance}
For any $(R,c) \in \GL_n(\mathbb{F}) \times Z(\mathbb{F})^{\times}$, a triple $(A_0,A_1,B)\in \mathbb{F}^{n\times n} \times \mathbb{F}^{n\times n} \times \GL_n(\mathbb{F})$ satisfies \eqref{eq:quadratic1}-\eqref{eq:quadratic4} if and only if $(R A_0 R^{-1}, R A_1 R^{-1}, c R B R^\sigma)$ satisfies \eqref{eq:quadratic1}-\eqref{eq:quadratic4}. In particular, given $B\in \GL_n(\mathbb{F})$ and $P\in G_B(\mathbb{F})$, a pair $(A_0,A_1)\in \mathbb{F}^{n\times n} \times \mathbb{F}^{n\times n}$ satisfies \eqref{eq:quadratic1}-\eqref{eq:quadratic4} if and only if $(P A_0 P^{-1}, P A_1 P^{-1})$ satisfies \eqref{eq:quadratic1}-\eqref{eq:quadratic4}.
\end{lemma}
\begin{proof}
It can be verified by a straightforward calculation.
\end{proof}
\subsection{Structure theorem for quadratic  rational curves}
We recall that each pair $(B,X)\in \mathbb{F}^{n\times n}$ satisfying $B^\sigma = \pm B$ and $X\in \mathfrak{g}_B(\mathbb{F})$ has a block diagonal normal form under the action of $\GL_n(\mathbb{F})$. We record this fact in Lemma~\ref{lem:indecomposable} for ease of reference.
\begin{lemma}\cite[Theorem~4]{DPW83}\label{lem:indecomposable}
For any $B\in \GL_n(\mathbb{F})$ such that $B^\sigma = \varepsilon B$ where $\varepsilon = \pm 1$ and $X \in \mathfrak{g}_B(\mathbb{F})$, there exists $R \in \GL_n(\mathbb{F})$ such that
\[
R X R^{-1} = \diag(X_1,\dots,X_s),\quad R B R^\sigma = \diag(B_1,\dots,B_s),
\] 
where $(B_j, X_j)$ are normal forms in Table~\ref{Tab:indecomposable} for each $1 \le j \le s$. In Table~\ref{Tab:indecomposable}, $\kappa = \pm 1$ and we denote 
\begin{align*}
A \otimes B &\coloneqq \begin{bsmallmatrix}
Ab_{11} &  \cdots &  Ab_{1l} \\
\vdots & \ddots & \vdots \\
Ab_{l1} & \cdots  & Ab_{ll}
\end{bsmallmatrix}\in \mathbb{F}^{kl \times kl},\quad 
H_m \coloneqq \begin{bsmallmatrix}
&  & 1 \\
&\udots & \\
1&  & \\
\end{bsmallmatrix} \in \mathbb{R}^{m \times m}, \\ 
J_m(A) &\coloneqq \begin{bsmallmatrix}
A & & &  \\
I_k & & & \\
& \ddots & \ddots & \\
& & I_k & A
\end{bsmallmatrix} \in \mathbb{F}^{km \times km},\quad
F_m \coloneqq \begin{bsmallmatrix}
 & & & 1  \\
 & &-1 & \\
& \udots & & \\
(-1)^{m-1} & &  & 
\end{bsmallmatrix} \in \mathbb{R}^{m \times m}, 
\end{align*}
where $A\in \mathbb{F}^{k \times k}, B\in \mathbb{F}^{l \times l} $.
\begin{table}[!htbp]
\centering
\scalebox{0.8}{
\begin{tabular}{|c|c|c|c|c|c|c|c|}
\hline
No. & $\mathbb{F}$& $\sigma$ &  $\varepsilon$ & $X$ & $B$ & $G_B(\mathbb{F})$ & Restrictions \\ \hline

\multirow{2}{*}{$1$} & \multirow{2}{*}{$\mathbb{C}$} & \multirow{2}{*}{ $\tp$} & \multirow{2}{*}{$+$} &\xrowht[()]{10pt} $J_{2m+1}(0)$ &\xrowht[()]{10pt}  $F_{2m+1}$  &\xrowht[()]{10pt}  $\O_{2m+1}$  &  \\ \cline{5-8}  
&                  &                   &                                &\xrowht[()]{10pt}  $\diag(J_m(\lambda),-J_m(\lambda)^\tp)$   &\xrowht[()]{10pt}  $I_m \otimes H_2$   &\xrowht[()]{10pt}  $\O_{2m}$   &\xrowht[()]{10pt}  $m$ even if $\lambda = 0$   \\    \hline

\multirow{2}{*}{$2$} &\multirow{2}{*}{$\mathbb{C}$} & \multirow{2}{*}{ $\tp$} & \multirow{2}{*}{$-$} &\xrowht[()]{10pt} $J_{2m}(0)$ &\xrowht[()]{10pt}  $F_{2m}$  &\xrowht[()]{10pt}   $\Sp_{2m}$  &  \\ 
\cline{5-8}  
&              &                   &                                &\xrowht[()]{10pt} $\diag(J_m(\lambda),-J_m(\lambda)^\tp)$   &\xrowht[()]{10pt}  $I_m \otimes F_2$   &\xrowht[()]{10pt}  $\Sp_{2m}$   &\xrowht[()]{10pt}  $m$ odd if $\lambda = 0$   \\ \hline

\multirow{2}{*}{$3$} &\multirow{2}{*}{$\mathbb{C}$} & \multirow{2}{*}{ $\ast$} & \multirow{2}{*}{$+$} &\xrowht[()]{25pt}  $J_{m}(\lambda)$ &\xrowht[()]{25pt}  $\kappa \ci^{m-1}F_{m}$  &\xrowht[()]{25pt}  $\U_{p,m-p}$  &\xrowht[()]{25pt}  \makecell[c]{$\operatorname{Re}(\lambda) = 0$ \\ $2p - m =\kappa \frac{1 - (-1)^m}{2}$}   \\ 
\cline{5-8}  
&            &                   &                                &\xrowht[()]{10pt}  $\diag(J_m(\lambda),-J_m(\lambda)^\ast)$   &\xrowht[()]{10pt}  $I_m \otimes H_2$   &\xrowht[()]{10pt}  $\U_{m,m}$   &\xrowht[()]{10pt}   $\operatorname{Re}(\lambda) > 0$   \\  \hline       
                  
\multirow{4}{*}{$4$} &\multirow{4}{*}{ $\mathbb{R}$} & \multirow{4}{*}{ $\tp$} & \multirow{4}{*}{$+$} &\xrowht[()]{10pt}  $J_{2m+1}(0)$ & $\kappa (-1)^m F_{2m+1}$  &\xrowht[()]{10pt}  $\O_{p,2m+1-p}$  &  \xrowht[()]{10pt}  $2p - 2m-1 =\kappa $   \\ 
\cline{5-8}  
&        &                   &                                &\xrowht[()]{25pt}  $\diag(J_m(\lambda),-J_m(\lambda)^\tp)$   &\xrowht[()]{25pt}  $I_m \otimes H_2$   &\xrowht[()]{25pt}  $ \xrowht[()]{10pt}  \O_{m,m}$   &\xrowht[()]{25pt}  \makecell[c]{$\lambda \ge 0$ \\ $m$ even if $\lambda = 0$}   \\ 
\cline{5-8}  
&               &                   &                                &\xrowht[()]{10pt}  $J_m\left( \begin{bmatrix}
0 & b \\
-b & 0
\end{bmatrix} \right)$    &\xrowht[()]{25pt}    $\kappa F_2^{m-1} \otimes F_m$   &\xrowht[()]{25pt}   $\O_{p,2m-p}$   &\xrowht[()]{25pt}   \makecell[c]{$b > 0$ \\  $2p - 2m = \kappa (1 - (-1)^m)$}  \\ 
\cline{5-8}  
&             &                   &                                &\xrowht[()]{25pt}  $\diag\left(  J_m\left( \begin{bmatrix}
a & b \\
-b & a
\end{bmatrix} \right),-J_m\left( \begin{bmatrix}
a & b \\
-b & a
\end{bmatrix} \right)^\tp \right)$   &\xrowht[()]{25pt}  $I_{2m} \otimes H_2$   &\xrowht[()]{25pt}  $\O_{2m,2m}$   & \xrowht[()]{25pt}  $a,b > 0$   \\   \hline      

\multirow{4}{*}{$5$} &\multirow{4}{*}{ $\mathbb{R}$} & \multirow{4}{*}{ $\tp$} & \multirow{4}{*}{$-$} &\xrowht[()]{10pt}  $J_{2m}(0)$ & $\kappa F_{2m}$  &\xrowht[()]{10pt}  $\Sp_{2m}$  &     \\ 
\cline{5-8}  
&             &                   &                                &\xrowht[()]{25pt}  $\diag(J_m(\lambda),-J_m(\lambda)^\tp)$   &\xrowht[()]{25pt}  $I_m \otimes F_2$   &\xrowht[()]{25pt}  $ \xrowht[()]{10pt}  \Sp_{2m}$   &\xrowht[()]{25pt}  \makecell[c]{$\lambda \ge 0$ \\ $m$ odd if $\lambda = 0$}   \\ 
\cline{5-8}  
&              &                   &                                &\xrowht[()]{25pt}  $J_m\left( \begin{bmatrix}
0 & b \\
-b & 0
\end{bmatrix} \right)$    &\xrowht[()]{25pt}  $\kappa F_2^{m} \otimes F_m$   &\xrowht[()]{25pt}  $\Sp_{2m}$   &\xrowht[()]{25pt} $b > 0$  \\ 
\cline{5-8}  
&                &                   &                                & $\diag\left(  J_m\left( \begin{bmatrix}
a & b \\
-b & a
\end{bmatrix} \right),-J_m\left( \begin{bmatrix}
a & b \\
-b & a
\end{bmatrix} \right)^\tp \right)$   &\xrowht[()]{25pt}  $I_{2m} \otimes F_2$   &\xrowht[()]{25pt}  $\Sp_{4m}$   & \xrowht[()]{25pt}  $a,b > 0$   \\   \hline

\multirow{3}{*}{$6$} &\multirow{3}{*}{ $\mathbb{H}$} & \multirow{3}{*}{ $\ast$} & \multirow{3}{*}{$+$} &\xrowht[()]{25pt}  $ J_m\left(\begin{bmatrix}
                  0 & 0 \\
                  0 & 0
                  \end{bmatrix} \right) $ & \xrowht[()]{25pt}  $ \kappa^m F_2^{m-1} \otimes F_m$ &  \xrowht[()]{25pt}  $ \xrowht[()]{10pt}  \Sp_{p,m-p}$  &    \xrowht[()]{25pt}  $2p - m =\kappa \frac{1 - (-1)^m}{2}$  \\ 
\cline{5-8}  
&            &                   &                                &\xrowht[()]{25pt}  $ J_m\left(\begin{bmatrix}
                  0 & b \\
                  -b & 0
                  \end{bmatrix} \right) $   &\xrowht[()]{25pt}  $ \kappa F_2^{m-1} \otimes F_m$   &\xrowht[()]{25pt}  $ \xrowht[()]{10pt}  \Sp_{p,m-p}$   &\xrowht[()]{25pt}  \makecell[c]{$b > 0 $ \\ $2p - m =\kappa \frac{1 - (-1)^m}{2}$}  \\ 
\cline{5-8}  
&               &                   &                                &\xrowht[()]{25pt}  $\diag\left(  J_m\left( \begin{bmatrix}
\lambda & 0 \\
0 & \lambda^\ast
\end{bmatrix} \right),-J_m\left( \begin{bmatrix}
\lambda & 0 \\
0 & \lambda^\ast
\end{bmatrix} \right)^\ast \right)$    &\xrowht[()]{25pt}  $I_{2m} \otimes H_2$   &\xrowht[()]{25pt}  $\Sp_{m,m}$   &\xrowht[()]{25pt} \makecell[c]{$\lambda \in \mathbb{C}$ \\ $\operatorname{Re}(\lambda) > 0$, $\operatorname{Im}(\lambda) \ge 0$}  \\    \hline

\multirow{3}{*}{$7$} &\multirow{3}{*}{ $\mathbb{H}$} & \multirow{3}{*}{ $\ast$} & \multirow{3}{*}{$-$} &\xrowht[()]{25pt}  $ J_m\left(\begin{bmatrix}
                  0 & 0 \\
                  0 & 0
                  \end{bmatrix} \right)$ & \xrowht[()]{25pt}  $ \kappa^{m-1} F_2^{m} \otimes F_m$ &  \xrowht[()]{25pt}  $ \xrowht[()]{10pt}  \O^\ast_{2m}$  &     \\ 
\cline{5-8}  
&                &                   &                                &\xrowht[()]{25pt}  $ J_m\left(\begin{bmatrix}
                  0 & b \\
                  -b & 0
                  \end{bmatrix} \right) $   &\xrowht[()]{25pt}  $ \kappa F_2^{m} \otimes F_m$   &\xrowht[()]{25pt}  $ \xrowht[()]{10pt}  \O_{2m}^\ast$   &\xrowht[()]{25pt}  $b > 0$   \\ 
\cline{5-8}  
&               &                   &                                &\xrowht[()]{25pt}  $\diag\left(  J_m\left( \begin{bmatrix}
\lambda & 0 \\
0 & \lambda^\ast
\end{bmatrix} \right),-J_m\left( \begin{bmatrix}
\lambda & 0 \\
0 & \lambda^\ast
\end{bmatrix} \right)^\ast \right)$    &\xrowht[()]{25pt}  $I_{2m} \otimes F_2$   &\xrowht[()]{25pt}  $\O^{\ast}_{4m}$   &\xrowht[()]{25pt} \makecell[c]{$\lambda \in \mathbb{C}$ \\ $\operatorname{Re}(\lambda) > 0$, $\operatorname{Im}(\lambda) \ge 0$}  \\    \hline    
\end{tabular}}
\caption{Indecomposable normal forms of elements in $\mathfrak{g}_B(\mathbb{F})$}
\label{Tab:indecomposable}
\end{table}
\end{lemma}

\begin{lemma}\label{lem:block}
Let $(A_0,A_1,B) \in \mathbb{F}^{n\times n} \times \mathbb{F}^{n\times n} \times \GL_n(\mathbb{F})$ be a solution of \eqref{eq:quadratic1}--\eqref{eq:quadratic4}. Assume further that $A_1 = \diag(X_1,\dots, X_s)$ and $B = \diag(B_1,\dots, B_s)$ where $\varepsilon = \pm 1$, $B_j^\sigma = \varepsilon B_j \in \GL_{m_j}(\mathbb{F})$ and $X_j\in \mathfrak{g}_{B_j}(\mathbb{F}) \subseteq \mathbb{F}^{m_j \times m_j}$ for each $1 \le j \le s$. If we partition $A_0 \in G_B(\mathbb{F})$ accordingly as $A_0 = (Y_{ij})_{i,j=1}^s$, then for each $1 \le i,j \le s$ we have
\begin{align}
X_i Y_{ij}  + Y_{ij} X_j   &= \delta_{ij} (2X_i + X_i^3).  \label{lem:block:eq1}\\
Y_{ji} &= \left( \delta_{ij}B_i^{-1} (2I_{m_i} + X_i^2) B_i - B_i^{-1} Y_{ij} B_j \right)^\sigma , \label{lem:block:eq2}
\end{align}
Here $\delta_{ij}$ is the Kronecker delta. In particular, if $i \ne j$ and $0 \not\in \rho(X_i) + \rho(X_j)$, then $Y_{ij} = Y_{ji} = 0$, where $\rho(X)$ is the spectrum of a matrix $X\in \mathbb{F}^{m\times n}$.
\end{lemma}
\begin{proof}
By equations \eqref{eq:quadratic2} and \eqref{eq:quadratic3}, we have 
\begin{align*}
B_i Y_{ji}^\sigma   &= \delta_{ij} (2 I_{m_i} + X_i^2) B_{i} - Y_{ij} B_j, \\
X_{i} B_i Y_{ji}^\sigma &= Y_{ij} X_j B_j,
\end{align*}
from which \eqref{lem:block:eq1} and \eqref{lem:block:eq2} can be obtained easily. We observe that \eqref{lem:block:eq1} is a Sylvester equation, whose solution is unique if and only if $\rho(X_i) \cap (-\rho(X_j)) = \emptyset$. Thus for $ i \ne j$ and $\rho(X_i) \cap (-\rho(X_j)) = \emptyset$, $Y_{ij} = 0$ is the unique solution of the homogeneous Sylvester equation \eqref{lem:block:eq1}.
\end{proof}

\begin{lemma}\label{lem:lower triangular}
We have the following:
\begin{enumerate}[(a)]
\item\label{lem:lower triangular:item1} For any $\lambda\in \mathbb{C}$, a solution of $J_m(\lambda) Y + Y J_n(-\lambda) = 0$ in $\mathbb{C}^{m\times n}$ is lower triangular alternating Toeplitz. Similarly, if $Y$ is a solution of $J_m(\lambda) Y + Y J_n(-\lambda)^\tp = 0$ (resp. $J_m(\lambda)^\tp Y + Y J_n(-\lambda) = 0$ and $J_m(\lambda)^\tp Y + Y J_n(-\lambda)^\tp = 0$), then $YH_n$ (resp. $H_m Y$ and  $H_m Y H_n$) is lower triangular alternating Toeplitz.
\item\label{lem:lower triangular:item1'} For any $\lambda\in \mathbb{C}$, a solution of $J_m(\lambda) Y - Y J_n(\lambda) = 0$ in $\mathbb{C}^{m\times n}$ is lower triangular Toeplitz. Similarly, if $Y$ is a solution of $J_m(\lambda) Y - Y J_n(\lambda)^\tp = 0$ (resp. $J_m(\lambda)^\tp Y - Y J_n(\lambda) = 0$ and $J_m(\lambda)^\tp Y - Y J_n(\lambda)^\tp = 0$), then $YH_n$ (resp. $H_m Y$ and  $H_m Y H_n$) is lower triangular Toeplitz.
\item\label{lem:lower triangular:item2} For any $\lambda\in \mathbb{C}$, a solution of $J_m(\lambda) Y + Y J_m(\lambda) = 2J_m(\lambda)  + J_m(\lambda)^3$ in $\mathbb{C}^{m\times m}$ has the form $Y = I_m + J_m(\lambda)^2/2 + T$ where $T$ is a lower triangular alternating Toeplitz matrix. Similarly, if $Y$ is a solution of $J_m(\lambda)^\tp Y + Y J_m(\lambda)^\tp = 2J_m(\lambda)^\tp  + (J_m(\lambda)^\tp)^3$ in $\mathbb{C}^{m\times m}$ then $Y  = I_m + (J_m(\lambda)^\tp)^2/2 + T$ for some $T$ such that $H_m T H_m$ is a lower triangular alternating Toeplitz matrix. In particular, if $\lambda \ne 0$ then $T = 0$.
\item\label{lem:lower triangular:item3} For any $b \ge 0$, a solution of 
\[
J_m\left( \begin{bsmallmatrix}
0 & b \\
-b & 0
\end{bsmallmatrix} \right) Y + Y J_n \left(\begin{bsmallmatrix}
0 & b \\
-b & 0
\end{bsmallmatrix}\right) = 0
\] 
in $\mathbb{F}^{2m\times 2n}$ is block lower triangular alternating Toeplitz, where each block has size $2\times 2$. If $b > 0$ then $2 \times 2$ blocks are of the form $\begin{bsmallmatrix}
x & y \\
y & -x
\end{bsmallmatrix} \in \mathbb{F}^{2\times 2}$.

\item\label{lem:lower triangular:item4} For any $b \ge 0$, a solution of 
\[
J_m\left( \begin{bsmallmatrix}
0 & b \\
-b & 0
\end{bsmallmatrix} \right) Y + Y J_m \left(\begin{bsmallmatrix}
0 & b \\
-b & 0
\end{bsmallmatrix}\right) = 2 J_m \left(\begin{bsmallmatrix}
0 & b \\
-b & 0
\end{bsmallmatrix}\right)  + J_m \left(\begin{bsmallmatrix}
0 & b \\
-b & 0
\end{bsmallmatrix}\right) ^3
\] 
in $\mathbb{F}^{2m\times 2m}$ can be written as $I_{2m} + \frac{1}{2} J_m\left( \begin{bsmallmatrix}
0 & b \\
-b & 0
\end{bsmallmatrix}\right)^2 + T$ for some block lower triangular alternating Toeplitz matrix $T$, where each block has size $2\times 2$. If $b > 0$ then $2 \times 2$ blocks are of the form $\begin{bsmallmatrix}
x & y \\
y & -x
\end{bsmallmatrix}\in \mathbb{F}^{2\times 2}$.

\item\label{lem:lower triangular:item5} For any $a, b \ge 0$, $Y (I_2\otimes H_n)$ is block lower triangular Toeplitz, where each block is of the form $\begin{bsmallmatrix}
x & y \\
y & -x
\end{bsmallmatrix} \in \mathbb{C}^{2\times 2}$. Here $Y$ is a solution of 
\[
J_m\left( \begin{bsmallmatrix}
a & b \\
-b & a
\end{bsmallmatrix} \right) Y - Y J_n \left(\begin{bsmallmatrix}
a & b \\
-b & a
\end{bsmallmatrix}\right)^\tp = 0
\] 
in $\mathbb{C}^{2m\times 2n}$. Similarly, If $Y$ is a solution of 
\[
J_m^\tp \left( \begin{bsmallmatrix}
a & b \\
-b & a
\end{bsmallmatrix} \right) Y - Y J_n \left(\begin{bsmallmatrix}
a & b \\
-b & a
\end{bsmallmatrix}\right) = 0,
\] 
then $(I_2 \otimes H_m) Y$ is block lower triangular Toeplitz, where each block is of the form $\begin{bsmallmatrix}
x & y \\
y & -x
\end{bsmallmatrix} \mathbb{C}^{2\times 2}.
$
\item\label{lem:lower triangular:item6} For any $\lambda \in \mathbb{C}$ with $\operatorname{Re}(\lambda) > 0$ and $\operatorname{Im}(\lambda) \ge 0$, $Y (I_2\otimes H_n)$ is block lower triangular Toeplitz, where each block has size $2\times 2$ and $Y$ is a solution of 
\[
J_m\left( \begin{bsmallmatrix}
\lambda & 0 \\
0 & \lambda^\ast
\end{bsmallmatrix} \right) Y - Y J_n \left(\begin{bsmallmatrix}
\lambda & 0 \\
0 & \lambda^\ast
\end{bsmallmatrix}\right)^\ast  = 0
\] 
in $\mathbb{H}^{2m\times 2n}$. Similarly, if $Y$ is a solution of 
\[
J_m\left( \begin{bsmallmatrix}
\lambda & 0 \\
0 & \lambda^\ast
\end{bsmallmatrix} \right)^\ast Y - Y J_n \left(\begin{bsmallmatrix}
\lambda & 0 \\
0 & \lambda^\ast
\end{bsmallmatrix}\right)  = 0,
\] 
then $(I_2 \otimes H_m) Y$ is block lower triangular Toeplitz, where each block has size $2\times 2$.  Moreover, if $\lambda \in \mathbb{C}\setminus \mathbb{R}$ then $2\times 2$ blocks are of the form $\begin{bsmallmatrix}
x\qj & y \\
z & w\qj
\end{bsmallmatrix} \in \mathbb{H}^{2\times 2}, x, y, z, w\in \mathbb{C}
$.\footnote{Here $\qi,\qj,\qk \in \mathbb{H}$ are the units in the standard expression $a + b \qi + c \qj + d \qk \in \mathbb{H}$ of a quaternion number.  The reader should distinguish the complex unit $\ci \in \mathbb{C}$ from the quaternion unit $\qi \in \mathbb{H}$.} 
\end{enumerate}
\end{lemma}
\begin{proof}
We defer the proof to Appendix~\ref{append:proof of lem:lower triangular}.
\end{proof} 
\begin{example}
As an illustration of Lemma~\ref{lem:lower triangular},  we consider $m = 2, n =3$ and $\lambda = 0$ so that
 $J_m(\lambda) Y + Y J_n(-\lambda) = 0$ becomes
\[
\begin{bsmallmatrix}
0 & 0 & 0 \\
y_{11}  & y_{12} & y_{13}
\end{bsmallmatrix} +
\begin{bsmallmatrix}
y_{12}  & y_{13} & 0 \\
y_{22}  & y_{23} & 0 \\
\end{bsmallmatrix} = 0,
\]
where $Y = (y_{ij})_{i,j=1}^{2,3}$. Clearly we have $y_{12} = y_{13} = y_{23} = 0$ and $y_{11} + y_{22} = 0$ from which we obtain $Y = \begin{bsmallmatrix}
y_{11} & 0 & 0 \\
y_{21} & -y_{11} & 0
\end{bsmallmatrix}$ is a lower triangular alternating Toeplitz matrix.
\end{example}
\begin{lemma}\label{lem:Yij}
Let $X_1,\dots, X_s, B_1,\dots, B_s, B$ be as in Lemma~\ref{lem:block}. Let $A_0 = (Y_{ij})_{i,j=1}^s \in \mathbb{F}^{n \times n}$ be a solution of \eqref{lem:block:eq1} and \eqref{lem:block:eq2} where $Y_{ij}$ is of size $m_i \times m_j$. Given $1 \le i, j \le s$, if $0\in \rho(X_i) + \sigma_j(X_j)$ then $Y_{ij}$ has one of the forms listed in Tables~\ref{Tab:CandidatesYij} and \ref{Tab:CandidatesYijcon'd}, in which the parameter $\kappa$ and matrices $F_m, H_m, J_m(A)$ and $A\otimes B$ are the same as in Lemma~\ref{lem:indecomposable}. Moreover, given a vector $(x_1,\dots, x_p)$ we define $\hat{x} \coloneqq (x_1,-x_2,\dots, (-1)^{p-2}x_{p-1}, (-1)^{p-1}x_p)$ and 
\begin{align*}
\ch{_{$\fp$} T}(x_1,\dots, x_p) &\coloneqq 
\begin{bsmallmatrix}
x_p & 0 & \cdots & 0\\
x_{p-1} & x_p & \cdots & 0 \\
\vdots & \vdots & \ddots & 0 \\
x_1 & x_{2} & \cdots &  x_p
\end{bsmallmatrix}, \quad 
\ch{_{$\fp$} S}(x_1,\dots, x_p) \coloneqq 
\begin{bsmallmatrix}
x_p & 0 & \cdots & 0\\
x_{p-1} & -x_p & \cdots & 0 \\
\vdots & \vdots & \ddots & 0 \\
x_1 & -x_{2} & \cdots & (-1)^{p-1} x_p
\end{bsmallmatrix}, \\  
T_{\fp}(x_1,\dots, x_p) &\coloneqq 
\begin{bsmallmatrix}
0 & \cdots & 0 & x_p\\
0 & \cdots & x_p & x_{p-1} \\
\vdots & \udots & \vdots & \vdots \\
x_p & \cdots & x_{2} & x_1
\end{bsmallmatrix}, \quad 
S_{\fp}(x_1,\dots, x_p) \coloneqq 
\begin{bsmallmatrix}
0 & \cdots & 0 & x_p\\
0 & \cdots & -x_p & x_{p-1} \\
\vdots & \udots & \vdots & \vdots \\
(-1)^{p-1} x_p & \cdots & -x_{2} & x_1
\end{bsmallmatrix},\\
\ch{^{$\fp$} T}(x_1,\dots, x_p) &\coloneqq 
\begin{bsmallmatrix}
x_1 & x_{2} & \cdots &  x_p \\
\vdots & \vdots & \udots & 0 \\
x_{p-1} & x_p & \cdots & 0 \\
x_p & 0 & \cdots & 0
\end{bsmallmatrix},\quad 
\ch{^{$\fp$} S}(x_1,\dots, x_p) \coloneqq 
\begin{bsmallmatrix}
x_1 & -x_{2} & \cdots &  (-1)^{p-1}x_p\\
\vdots & \vdots  & \udots & 0 \\
x_{p-1} & -x_p & \cdots & 0 \\
x_p & 0 & \cdots & 0
\end{bsmallmatrix},  \\ 
T^{\fp}(x_1,\dots, x_p) &\coloneqq 
\begin{bsmallmatrix}
x_p & \cdots & x_{2} & x_1\\
\vdots & \udots & \vdots & \vdots \\
0 & \cdots & x_p & x_{p-1} \\
0 & \cdots & 0 & x_p
\end{bsmallmatrix},\quad 
S^{\fp}(x_1,\dots, x_p) \coloneqq 
\begin{bsmallmatrix}
(-1)^{p-1} x_p & \cdots & -x_{2} & x_1\\
\vdots & \udots & \vdots & \vdots \\
0 & \cdots & -x_p & x_{p-1} \\
0 & \cdots & 0 & x_p
\end{bsmallmatrix}.
\end{align*}

\begin{table}[!htbp] 
\scalebox{0.53}{
\begin{tabular}{|c|c|c|c|c|}
\hline
No. & $(X_i,X_j), m_i \ge m_j$ &  $(Y_{ij},Y_{ji})$ &  $Y_{ii}$ & $ (B_i,B_j)$ \\ \hline
\multirow{5}{*}{1} 
&$\left(
J_{2m+1}(0), J_{2n+1}(0)
\right)$     &\xrowht[()]{25pt}  $\left(
\begin{bsmallmatrix}
0 \\ \ch{_{$\fp$} S}(z)
\end{bsmallmatrix}, -\begin{bsmallmatrix}
\ch{_{$\fp$} S}(z) & 0
\end{bsmallmatrix} \right) $   &   $I_{2m+1} + \frac{1}{2} J_{2m+1}(0)^2$   &  $(F_{2m+1},F_{2n+1})$   \\ \cline{2-5} 

                   & \xrowht[()]{15pt}   $
\left(
J_{2m+1}(0),\diag(J_{2n}(0), -J_{2n}(0)^\tp)
\right)$    &\xrowht[()]{30pt}   \makecell[c]{$\left(
\begin{bsmallmatrix}
0 & 0 \\
\ch{_{$\fp$} S}(z) & T_{\fp}(w)
\end{bsmallmatrix}, -\begin{bsmallmatrix}
\ch{_{$\fp$} S}(w)  & 0 \\ 
\ch{^{$\fp$} T} (\hat{z}) &  0 
\end{bsmallmatrix}  \right)$ \vspace*{5pt} \\ 
$\left(
\begin{bsmallmatrix}
\ch{_{$\fp$} S}(z) & 0 & 0  & T_{\fp}(w)
\end{bsmallmatrix}, 
-\begin{bsmallmatrix}
0 \\
\ch{_{$\fp$} S}(w) \\ 
\ch{^{$\fp$} T}(\hat{z}) \\
0
\end{bsmallmatrix}  \right)$ }  &  &\xrowht[()]{15pt}  $(F_{2m+1},I_n\otimes H_2)$    \\  \cline{2-5}

                   & \xrowht[()]{45pt}  \makecell[c]{$
\diag(J_{m}(\lambda), -J_{m}(\lambda)^\tp)$ 
\vspace*{5pt} \\ 
$\diag(J_{n}(\lambda), -J_{n}(\lambda)^\tp)$  
\vspace*{5pt} \\
$\lambda  \ne 0$  }   &\xrowht[()]{30pt} $\left( \begin{bsmallmatrix}
0 & 0 \\
0 & T_{\fp}(w) \\
\ch{^{$\fp$} T}(z)  & 0 \\
0 & 0
\end{bsmallmatrix},  
-\begin{bsmallmatrix}
0 & 0 & 0 & T_{\fp}(w) &  \\
\ch{^{$\fp$} T} (z)  & 0  & 0 & 0
\end{bsmallmatrix}
\right)$    &    
$I_{2m} +  \frac{1}{2}\diag(J_m(\lambda)^2, (J_m(\lambda)^\tp)^2 )$   &\xrowht[()]{15pt}   $(I_m\otimes H_2,I_n \otimes H_2)$ \\ \cline{2-5}

                   & \xrowht[()]{45pt}  \makecell[c]{$
\diag(J_{m}(\lambda), -J_{m}(\lambda)^\tp)$
\vspace*{5pt} \\ 
$\diag(J_{n}(-\lambda), -J_{n}(-\lambda)^\tp)$ 
\vspace*{5pt} \\
$\lambda \ne 0$  }   &\xrowht[()]{30pt} $ \left( \begin{bsmallmatrix}
0 & 0 \\
\ch{_{$\fp$} S}(z) & 0 \\
0 & S^{\fp}(w)  \\
0 & 0
\end{bsmallmatrix},  
-\begin{bsmallmatrix}
\ch{_{$\fp$} S}(\hat{w}) & 0 & 0 & 0 &  \\
0 & 0  & 0 & S^{\fp}(\hat{z})
\end{bsmallmatrix}
\right)$      &        &\xrowht[()]{15pt}   $(I_m\otimes H_2,I_n \otimes H_2)$ \\ \cline{2-5}

                   & \xrowht[()]{25pt}  \makecell[c]{$
\diag(J_{2m}(0), -J_{2m}(0)^\tp)$ 
\vspace*{5pt} \\
$\diag(J_{2n}(0), -J_{2n}(0)^\tp)
$}   &\xrowht[()]{30pt} $\left( \begin{bsmallmatrix}
0 & 0 \\
\ch{_{$\fp$} S}(z) & T_{\fp}(w) \\
\ch{^{$\fp$} T}(u) & S^{\fp}(v)  \\
0 & 0
\end{bsmallmatrix},  
-\begin{bsmallmatrix}
\ch{_{$\fp$} S}(\hat{v}) & 0 & 0 & T_{\fp}(w) &  \\
\ch{^{$\fp$} T}(u) & 0  & 0 & S^{\fp}(\hat{z})
\end{bsmallmatrix}
\right)$    &    $I_{4m} + \frac{1}{2}\diag(J_{2m}(0)^2, (J_{2m}(0)^\tp)^2)  + \begin{bsmallmatrix}
\ch{_{$\fp$} S}(z) & 0 \\
0 & -S^{\fp}(\hat{z})
\end{bsmallmatrix}$    &\xrowht[()]{15pt}   $(I_{2m}\otimes H_2,I_{2n} \otimes H_2)$ \\ \cline{2-5} \hline

\multirow{5}{*}{2} 
&$\left(
J_{2m}(0), J_{2n}(0)
\right)$     &\xrowht[()]{25pt}  $\left(
\begin{bsmallmatrix}
0 \\ \ch{_{$\fp$} S}(z)
\end{bsmallmatrix}, \begin{bsmallmatrix}
\ch{_{$\fp$} S}(z) & 0
\end{bsmallmatrix} \right) $   &   $I_{2m} + \frac{1}{2} J_{2m}(0)^2 + \ch{_{$\fp$} S}(z)$   &  $(F_{2m},F_{2n})$   \\ \cline{2-5} 

                   & \xrowht[()]{15pt}   $
\left(
J_{2m}(0),\diag(J_{2n+1}(0), -J_{2n+1}(0)^\tp)
\right)$    &\xrowht[()]{30pt}   \makecell[c]{$ \left(
\begin{bsmallmatrix}
0 & 0 \\
\ch{_{$\fp$} S}(z) & T_{\fp}(w)
\end{bsmallmatrix}, 
\begin{bsmallmatrix}
-\ch{_{$\fp$} S}(w)  & 0 \\ 
\ch{^{$\fp$} T} (\hat{z}) &  0 
\end{bsmallmatrix}  \right)$ \vspace*{5pt} \\ 
$\left(
\begin{bsmallmatrix}
\ch{_{$\fp$} S}(z) & 0 & 0  & T_{\fp}(w)
\end{bsmallmatrix}, 
\begin{bsmallmatrix}
0 \\
-\ch{_{$\fp$} S}(w) \\ 
\ch{^{$\fp$} T}(\hat{z}) \\
0
\end{bsmallmatrix}  \right)$ }  &  &\xrowht[()]{15pt}  $(F_{2m},I_{2n+1}\otimes F_2)$    \\  \cline{2-5}

                   & \xrowht[()]{45pt}  \makecell[c]{$
\diag(J_{m}(\lambda), -J_{m}(\lambda)^\tp)$ 
\vspace*{5pt} \\
$\diag(J_{n}(\lambda), -J_{n}(\lambda)^\tp)
$  
\vspace*{5pt} \\
$\lambda \ne 0$  }   &\xrowht[()]{30pt} $\left( \begin{bsmallmatrix}
0 & 0 \\
0 & T_{\fp}(w) \\
\ch{^{$\fp$} T}(z)  & 0 \\
0 & 0
\end{bsmallmatrix},  
\begin{bsmallmatrix}
0 & 0 & 0 & T_{\fp}(w) &  \\
\ch{^{$\fp$} T} (z)  & 0  & 0 & 0
\end{bsmallmatrix}
\right)$    &    
$I_{2m} +  \frac{1}{2} \diag(J_m(\lambda)^2 , (J_m(\lambda)^\tp)^2) + \begin{bsmallmatrix}
0 & T_{\fp}(w) \\
\ch{^{$\fp$} T}(z)  & 0 
\end{bsmallmatrix} $   &\xrowht[()]{15pt}   $(I_m\otimes F_2,I_n \otimes F_2)$ \\ \cline{2-5}

                   & \xrowht[()]{45pt}  \makecell[c]{$
\diag(J_{m}(\lambda), -J_{m}(\lambda)^\tp)$
\vspace*{5pt} \\ 
$\diag(J_{n}(-\lambda), -J_{n}(-\lambda)^\tp)
$  
\vspace*{5pt} \\
$\lambda  \ne 0$  }   &\xrowht[()]{30pt} $\left( \begin{bsmallmatrix}
0 & 0 \\
\ch{_{$\fp$} S}(z) & 0 \\
0 & S^{\fp}(w)  \\
0 & 0
\end{bsmallmatrix},  
\begin{bsmallmatrix}
-\ch{_{$\fp$} S}(\hat{w}) & 0 & 0 & 0 &  \\
0 & 0  & 0 & -S^{\fp}(\hat{v})
\end{bsmallmatrix}
\right)$      &        &\xrowht[()]{15pt}   $(I_m\otimes F_2,I_n \otimes F_2)$ \\ \cline{2-5}

                   & \xrowht[()]{25pt}  
\makecell[c]{
$\diag(J_{2m+1}(0), -J_{2m+1}(0)^\tp)$
\vspace*{5pt} \\
$\diag(J_{2n+1}(0), -J_{2n+1}(0)^\tp)$}
   &\xrowht[()]{30pt} $\left( \begin{bsmallmatrix}
0 & 0 \\
\ch{_{$\fp$} S}(z) & T_{\fp}(w) \\
\ch{^{$\fp$} T}(u) & S^{\fp}(v)  \\
0 & 0
\end{bsmallmatrix},  
\begin{bsmallmatrix}
-\ch{_{$\fp$} S}(\hat{v}) & 0 & 0 & T_{\fp}(w) &  \\
\ch{^{$\fp$} T}(u) & 0  & 0 & -S^{\fp}(\hat{z})
\end{bsmallmatrix}
\right)$    &    $I_{4m+2} + \frac{1}{2}\diag(J_{2m+1}(0)^2, (J_{2m+1}(0)^\tp)^2)
+ \begin{bsmallmatrix}
\ch{_{$\fp$} S}(z) & T_{\fp}(w)  \\
\ch{^{$\fp$} T}(u) & -S^{\fp}(\hat{z})
\end{bsmallmatrix}$    &\xrowht[()]{15pt}   $(I_{2m+1}\otimes F_2,I_{2n+1} \otimes F_2)$ \\ \cline{2-5} \hline

\multirow{5}{*}{3} 
& \xrowht[()]{20pt} \makecell[c]{$\left(
J_{m}(\lambda), J_{n}(-\lambda)
\right)$ \vspace*{5pt} \\ $\operatorname{Re}(\lambda) = 0$ }    &\xrowht[()]{25pt}  $\left(
\begin{bsmallmatrix}
0 \\ \ch{_{$\fp$} S}(z)
\end{bsmallmatrix}, \kappa \kappa' (-\ci)^{m+n} \begin{bsmallmatrix}
\ch{_{$\fp$} S}(\overline{z}) & 0
\end{bsmallmatrix} \right) $   &  \makecell[c]{$I_m + \frac{1}{2} J_m(0)^2 + \ch{_{$\fp$} S}(z)$ \vspace*{5pt} \\ $z = (-1)^m \overline{z}$}   &  $(\kappa \ci^{m-1}F_{m}, \kappa' \ci^{n-1} F_{n})$   \\ \cline{2-5} 

                   & \xrowht[()]{45pt}   
\makecell[c]{$
\diag(J_{m}(\lambda), -J_{m}(\lambda)^\ast)$ 
\vspace*{5pt} \\
$\diag(J_{n}(\overline{\lambda}), -J_{n}(\overline{\lambda})^\ast)
$ 
\vspace*{5pt} \\
$\operatorname{Re}(\lambda) > 0$}    &  
$ \left(
\begin{bsmallmatrix}
0 & 0 \\
0 &  T_{\fp}(w) \\
\ch{^{$\fp$} T}(z) & 0 \\
0 & 0
\end{bsmallmatrix}, 
-\begin{bsmallmatrix}
0 & 0 & 0 & T_{\fp}(\overline{w})  \\ 
\ch{^{$\fp$} T} (\overline{z}) &  0  & 0 & 0
\end{bsmallmatrix}  \right)$   & \makecell[c]{$
I_{2m} + 
\frac{1}{2} \diag(J_m(\lambda)^2, (J_m(\lambda)^\ast)^2 )
+
\begin{bsmallmatrix}
0 &  T_{\fp}(w) \\
\ch{_{$\fp$} T}(z) & 0 \\
\end{bsmallmatrix}$ \vspace*{5pt} \\$\lambda$: real,\quad  $z,w$: pure imaginary}  &\xrowht[()]{15pt}  $(I_m \otimes H_2, I_{n}\otimes H_2)$    \\  \cline{2-5} 

& \xrowht[()]{20pt} \makecell[c]{$\left(
J_{m}(\lambda), J_{m}(\lambda)
\right)$ \vspace*{5pt} \\ $\operatorname{Re}(\lambda) = 0, \lambda \ne 0$ }    &\xrowht[()]{25pt}  $\left(
0, 0 \right) $   &  $I_m + \frac{1}{2} J_m(\lambda)^2 $   &  $(\kappa \ci^{m-1}F_{m}, \kappa' \ci^{m-1} F_{m})$   \\ \cline{2-5} 

\hline

\multirow{5}{*}{4} 
&$\left(
J_{2m+1}(0), J_{2n+1}(0)
\right)$     &\xrowht[()]{25pt}  $\left(
\begin{bsmallmatrix}
0 \\ \ch{_{$\fp$} S}(z)
\end{bsmallmatrix}, \kappa \kappa' (-1)^{m+n+1}\begin{bsmallmatrix}
\ch{_{$\fp$} S}(z) & 0
\end{bsmallmatrix} \right) $   &   $I_{2m+1} + \frac{1}{2} J_{2m+1}(0)^2$   &  $(\kappa (-1)^m F_{2m+1}, \kappa' (-1)^n F_{2n+1})$   \\ \cline{2-5} 

                   & \xrowht[()]{15pt}   $
\left(
J_{2m+1}(0),\diag(J_{2n}(0), -J_{2n}(0)^\tp)
\right)$    &\xrowht[()]{60pt}   \makecell[c]{$ \left(
\begin{bsmallmatrix}
0 & 0 \\
\ch{_{$\fp$} S}(z) & T_{\fp}(w)
\end{bsmallmatrix}, \kappa (-1)^{m+1} \begin{bsmallmatrix}
\ch{_{$\fp$} S}(w)  & 0 \\ 
\ch{^{$\fp$} T} (\hat{z}) &  0 
\end{bsmallmatrix}  \right)$ \vspace*{5pt} \\ 
$\left(
\begin{bsmallmatrix}
\ch{_{$\fp$} S}(z) & 0 & 0  & T_{\fp}(w)
\end{bsmallmatrix}, 
\kappa (-1)^{m+1} \begin{bsmallmatrix}
0 \\
\ch{_{$\fp$} S}(w) \\ 
\ch{^{$\fp$} T}(\hat{z}) \\
0
\end{bsmallmatrix}  \right)$ }  &  &\xrowht[()]{15pt}  $(\kappa (-1)^m F_{2m+1},I_{2n}\otimes H_2)$    \\  \cline{2-5} 

                   & \xrowht[()]{45pt}  \makecell[c]{$
\diag(J_{m}(\lambda), -J_{m}(\lambda)^\tp)$
\vspace*{5pt} \\
$\diag(J_{n}(\lambda), -J_{n}(\lambda)^\tp)
$  
\vspace*{5pt} \\
$\lambda > 0$  }   &\xrowht[()]{30pt} $\left( \begin{bsmallmatrix}
0 & 0 \\
0 & T_{\fp}(w) \\
\ch{^{$\fp$} T}(z)  & 0 \\
0 & 0
\end{bsmallmatrix},  
-\begin{bsmallmatrix}
0 & 0 & 0 & T_{\fp}(w) &  \\
\ch{^{$\fp$} T} (z)  & 0  & 0 & 0
\end{bsmallmatrix}
\right)$    &    
$I_{2m} +  \frac{1}{2} \diag(J_m(\lambda)^2, (J_m(\lambda)^\tp)^2)$   &\xrowht[()]{15pt}   $(I_m\otimes H_2,I_n \otimes H_2)$ \\ \cline{2-5} 

                   & \xrowht[()]{35pt}  \makecell[c]{$
\diag(J_{2m}(0), -J_{2m}(0)^\tp)$
\vspace*{5pt} \\
$\diag(J_{2n}(0), -J_{2n}(0)^\tp)
$}  &\xrowht[()]{30pt} $\left( \begin{bsmallmatrix}
0 & 0 \\
\ch{_{$\fp$} S}(z) & T_{\fp}(w) \\
\ch{^{$\fp$} T}(u) & S^{\fp}(v)  \\
0 & 0
\end{bsmallmatrix},  
-\begin{bsmallmatrix}
\ch{_{$\fp$} S}(\hat{v}) & 0 & 0 & T_{\fp}(w) &  \\
\ch{^{$\fp$} T}(u) & 0  & 0 & S^{\fp}(\hat{z})
\end{bsmallmatrix}
\right)$      &   $I_{4m} + \frac{1}{2} \diag(J_{2m}(0)^2, (J_{2m}(0)^\tp)^2 )  + \begin{bsmallmatrix}
\ch{_{$\fp$} S}(z) & 0  \\
0 & -S^{\fp}(\hat{z})  \\
\end{bsmallmatrix} $     &\xrowht[()]{15pt}   $(I_m\otimes H_2,I_n \otimes H_2)$ \\ \cline{2-5}

                   & \xrowht[()]{15pt}  $
\left(
J_{m}\left( \begin{bmatrix}
0 & b \\
-b & 0
\end{bmatrix}\right),
J_{n}\left( \begin{bmatrix}
0 & b \\
-b & 0
\end{bmatrix}\right)
\right), b > 0$ \ &\xrowht[()]{30pt} 
\makecell[c]{
$\left( 
\begin{bsmallmatrix}
0  \\
\ch{_{$\fp$} S}(Z)
\end{bsmallmatrix},  
-\kappa \kappa' \begin{bsmallmatrix}
\ch{_{$\fp$} S}( (F_2^{m-1} Z F_2^{n-1})^\tp ) & 0 
\end{bsmallmatrix}
\right)$ \vspace*{5pt}
\\ 
$Z_p = \begin{bsmallmatrix}
x_p & y_p \\
y_p & -x_p
\end{bsmallmatrix}, 1\le p \le n
$
}      & \xrowht[()]{20pt}  $I_{2m} + \frac{1}{2} J_m\left(\begin{bsmallmatrix}
0 & b \\
-b & 0
\end{bsmallmatrix}
\right)^2$   &   $(\kappa F_2^{m-1} \otimes F_m,\kappa' F_2^{n-1} \otimes F_n)$ \\ \cline{2-5}

                   & \xrowht[()]{80pt} \makecell[c]{  $                 
\diag\left(
J_{m}\left( \begin{bmatrix}
a & b \\
-b & a
\end{bmatrix}\right),
-J_{m}\left( \begin{bmatrix}
a & b \\
-b & a
\end{bmatrix}\right)^\tp \right)$ \vspace*{5pt} \\
$\diag\left(
J_{n}\left( \begin{bmatrix}
a & b \\
-b & a
\end{bmatrix}\right),
-J_{n}\left( \begin{bmatrix}
a & b \\
-b & a
\end{bmatrix}\right)^\tp \right)$
\vspace*{5pt} \\
$a,b > 0$}  &\xrowht[()]{30pt} 
\makecell[c]{
$\left( \begin{bsmallmatrix}
0  & 0 \\
0 & T_{\fp} (W) \\
\ch{^{$\fp$} T} (U) & 0 \\
0  & 0 
\end{bsmallmatrix},  
-\begin{bsmallmatrix}
0  & 0 & 0 & T_{\fp} (W^\tp)\\
\ch{^{$\fp$} T} (U^\tp)  & 0  & 0 & 0 
\end{bsmallmatrix}
\right)$
\vspace*{5pt} \\
$U_p = \begin{bsmallmatrix}
x_p & y_p \\
y_p & -x_p
\end{bsmallmatrix}, W_p = \begin{bsmallmatrix}
c_p & d_p \\
d_p & -c_p
\end{bsmallmatrix}, 1 \le p \le n
$
}      & \xrowht[()]{20pt}  $I_{4m} + \frac{1}{2} 
\diag\left(
J_m\left(\begin{bsmallmatrix}
a & b \\
-b & a
\end{bsmallmatrix}
\right)^2,
\left( J_m\left(\begin{bsmallmatrix}
a & b \\
-b & a
\end{bsmallmatrix}
\right)^\tp
\right)^2
\right)$   &   $(I_{2m}\otimes H_2,I_{2n} \otimes H_2)$ \\ \cline{2-5}
\hline
\end{tabular}}
\caption{Candidates of $Y_{ij}$ (No.~1--No.~4)}
\label{Tab:CandidatesYij}
\end{table}
\begin{table}[!htbp] 
\scalebox{0.52}{
\begin{tabular}{|c|c|c|c|c|}
\hline
No. & $(X_i,X_j), m_i \ge m_j$ &  $(Y_{ij},Y_{ji})$ &  $Y_{ii}$ & $ (B_i,B_j)$ \\ \hline
\multirow{6}{*}{5} 
&$\left(
J_{2m}(0), J_{2n}(0)
\right)$     &\xrowht[()]{25pt}  $\left(
\begin{bsmallmatrix}
0 \\ \ch{_{$\fp$} S}(z)
\end{bsmallmatrix}, \kappa \kappa' \begin{bsmallmatrix}
\ch{_{$\fp$} S}(z) & 0
\end{bsmallmatrix} \right) $   &   $I_{2m} + \frac{1}{2} J_{2m}(0)^2 + \ch{_{$\fp$} S}(z)$   &  $(\kappa  F_{2m}, \kappa'  F_{2n})$   \\ \cline{2-5} 

                   & \xrowht[()]{15pt}   $
\left(
J_{2m}(0),\diag(J_{2n+1}(0), -J_{2n+1}(0)^\tp)
\right)$    &\xrowht[()]{30pt}   \makecell[c]{$ \left(
\begin{bsmallmatrix}
0 & 0 \\
\ch{_{$\fp$} S}(z) & T_{\fp}(w)
\end{bsmallmatrix}, 
\kappa  \begin{bsmallmatrix}
-\ch{_{$\fp$} S}(w)  & 0 \\ 
\ch{^{$\fp$} T} (\hat{z}) &  0 
\end{bsmallmatrix}  \right)$ \vspace*{5pt} \\ 
$\left(
\begin{bsmallmatrix}
\ch{_{$\fp$} S}(z) & 0 & 0  & T_{\fp}(w)
\end{bsmallmatrix}, 
\kappa  \begin{bsmallmatrix}
0 \\
-\ch{_{$\fp$} S}(w) \\ 
\ch{^{$\fp$} T}(\hat{z}) \\
0
\end{bsmallmatrix}  \right)$ }  &  &\xrowht[()]{15pt}  $(\kappa F_{2m},I_{2n+1}\otimes F_2)$    \\  \cline{2-5} 

                   & \xrowht[()]{50pt}  \makecell[c]{$
\diag(J_{m}(\lambda), -J_{m}(\lambda)^\tp)$ 
\vspace*{5pt} \\
$\diag(J_{n}(\lambda), -J_{n}(\lambda)^\tp)$ 
\vspace*{5pt} \\ 
$\lambda > 0$  }   &\xrowht[()]{30pt} $\left( \begin{bsmallmatrix}
0 & 0 \\
0 & T_{\fp}(w) \\
\ch{^{$\fp$} T}(z)  & 0 \\
0 & 0
\end{bsmallmatrix},  
\begin{bsmallmatrix}
0 & 0 & 0 & T_{\fp}(w) &  \\
\ch{^{$\fp$} T} (z)  & 0  & 0 & 0
\end{bsmallmatrix}
\right)$    &    
$I_{2m} +  \frac{1}{2}
\diag(J_m(\lambda)^2, (J_m(\lambda)^\tp)^2)
+ \begin{bsmallmatrix}
0 & T_{\fp}(w) \\
\ch{^{$\fp$} T}(z)  & 0
\end{bsmallmatrix}$   &\xrowht[()]{15pt}   $(I_m\otimes F_2,I_n \otimes F_2)$ \\ \cline{2-5}

                   & \xrowht[()]{25pt}  \makecell[c]{$
\diag(J_{2m+1}(0), -J_{2m+1}(0)^\tp)$ 
\vspace*{5pt} \\
$\diag(J_{2n+1}(0), -J_{2n+1}(0)^\tp)
$}  &\xrowht[()]{30pt} 
$\left( \begin{bsmallmatrix}
0 & 0 \\
\ch{_{$\fp$} S}(z) & T_{\fp}(w) \\
\ch{^{$\fp$} S}(u) & S^{\fp}(v)  \\
0 & 0
\end{bsmallmatrix},  
\begin{bsmallmatrix}
-\ch{_{$\fp$} S}(\hat{v}) & 0 & 0 & T_{\fp}(w) &  \\
\ch{^{$\fp$} T}(u) & 0  & 0 & -S^{\fp}(\hat{z})
\end{bsmallmatrix}
\right)$      &   $I_{4m+2} + \frac{1}{2} 
\diag(J_{2m+1}(0)^2, (J_{2m+1}(0)^\tp)^2)
 + \begin{bsmallmatrix}
\ch{_{$\fp$} S}(z) & T_{\fp}(w)  \\
\ch{^{$\fp$} T}(u) & -S^{\fp}(\hat{z})  \\
\end{bsmallmatrix} $     &\xrowht[()]{15pt}   $(I_{2m+1}\otimes F_2,I_{2n+1} \otimes F_2)$ \\ \cline{2-5}

                   & \xrowht[()]{15pt}  $
\left(
J_{m}\left( \begin{bmatrix}
0 & b \\
-b & 0
\end{bmatrix}\right),
J_{n}\left( \begin{bmatrix}
0 & b \\
-b & 0
\end{bmatrix}\right)
\right), b > 0$ \ &\xrowht[()]{30pt} 
\makecell[c]{
$\left( \begin{bsmallmatrix}
0  \\
\ch{_{$\fp$} S}(Z)
\end{bsmallmatrix},  
\kappa \kappa' \begin{bsmallmatrix}
\ch{_{$\fp$} S}( (F_2^{m} Z F_2^{n})^\tp ) & 0 
\end{bsmallmatrix}
\right)$ \vspace*{5pt} \\
$Z_p  = \begin{bsmallmatrix}
x_p & y_p \\
y_p & -x_p
\end{bsmallmatrix}, 1 \le p \le n$
}     & \xrowht[()]{20pt} $I_{2m} + \frac{1}{2} J_m(\begin{bsmallmatrix}
0 & b \\
-b & 0
\end{bsmallmatrix}
)^2 + 
\ch{_{$\fp$} S}(Z)$      &   $(\kappa F_2^{m} \otimes F_m,\kappa' F_2^{n} \otimes F_n)$ \\ \cline{2-5}

                   & \xrowht[()]{80pt} \makecell[c]{  $
\diag\left(
J_{m}\left( \begin{bmatrix}
a & b \\
-b & a
\end{bmatrix}\right),
-J_{m}\left( \begin{bmatrix}
a & b \\
-b & a
\end{bmatrix}\right)^\tp \right)$ \vspace*{5pt} \\
$\diag\left(
J_{n}\left( \begin{bmatrix}
a & b \\
-b & a
\end{bmatrix}\right),
-J_{n}\left( \begin{bmatrix}
a & b \\
-b & a
\end{bmatrix}\right)^\tp \right)
$ 
\vspace*{5pt} \\
$a,b > 0$}  &\xrowht[()]{30pt} 
\makecell[c]{
$\left( \begin{bsmallmatrix}
0  & 0 \\
0 & T_{\fp} (W) \\
\ch{^{$\fp$} T} (U) & 0 \\
0  & 0 
\end{bsmallmatrix},  
\begin{bsmallmatrix}
0  & 0 & 0 & T_{\fp} (W^\tp)\\
\ch{^{$\fp$} T} (U^\tp)  & 0  & 0 & 0 
\end{bsmallmatrix}
\right)$
\vspace*{5pt} \\
$Z_p = \begin{bsmallmatrix}
x_p & y_p \\
y_p & -x_p
\end{bsmallmatrix},  W_p = \begin{bsmallmatrix}
c_p & d_p \\
d_p & -c_p
\end{bsmallmatrix}, 1 \le p \le n
$
}      & \xrowht[()]{20pt} $I_{4m} + \frac{1}{2} \diag\left(
J_m \left( 
\begin{bsmallmatrix}
a & b \\
-b & a
\end{bsmallmatrix}
\right)^2,
\left(J_m \left( 
\begin{bsmallmatrix}
a & b \\
-b & a
\end{bsmallmatrix}\right)^\tp
\right)^2
\right)
$    &   $(I_{2m}\otimes F_2,I_{2n} \otimes F_2)$ \\ \cline{2-5}\hline

\multirow{6}{*}{6} 
&$\left(
J_{m}\left(\begin{bsmallmatrix}
0 & 0 \\
0 & 0
\end{bsmallmatrix}
\right), J_{n}\left(\begin{bsmallmatrix}
0 & 0 \\
0 & 0
\end{bsmallmatrix}\right)
\right)$     &\xrowht[()]{25pt}  $\left(
\begin{bsmallmatrix}
0 \\ \ch{_{$\fp$} S}(Z)
\end{bsmallmatrix}, -
\kappa^{m} {\kappa'}^n\begin{bsmallmatrix}
\ch{_{$\fp$} S}( (F_2^{m-1}ZF_2^{n-1})^\ast ) & 0
\end{bsmallmatrix} \right) $   &   \makecell[c]{$I_m + \frac{1}{2} J_m(0)^2 + \ch{_{$\fp$} S}(Z)$ \vspace*{5pt} \\  $-
\kappa^{m+n} (F_2^{m-1} Z F_2^{m-1})^\ast = Z$ }  &  $(\kappa^m F_2^{m-1} \otimes  F_{m}, {\kappa'}^n F_2^{n-1}  \otimes F_{n})$   \\ \cline{2-5}

                   & \xrowht[()]{15pt}  $
\left(
J_{m}\left( \begin{bmatrix}
0 & b \\
-b & 0
\end{bmatrix}\right),
J_{n}\left( \begin{bmatrix}
0 & b \\
-b & 0
\end{bmatrix}\right)
\right), b > 0$ \ &\xrowht[()]{30pt} 
\makecell[c]{
$\left( \begin{bsmallmatrix}
0  \\
\ch{_{$\fp$} S}(Z)
\end{bsmallmatrix},  
-\kappa \kappa' \begin{bsmallmatrix}
\ch{_{$\fp$} S}( (F_2^{m-1} Z F_2^{n-1})^\ast ) & 0 
\end{bsmallmatrix}
\right)$ \vspace*{5pt}\\
$Z_p = \begin{bsmallmatrix} 
x_p & y_p \\
y_p & -x_p
\end{bsmallmatrix}, 1 \le p \le n
$}      & \xrowht[()]{25pt}  $I_{2m} + \frac{1}{2} J_m\left(
\begin{bsmallmatrix}
0 & b \\ 
-b & 0
\end{bsmallmatrix}
\right)^2 
$     &   $(\kappa F_2^{m-1} \otimes F_m,\kappa' F_2^{n-1} \otimes F_n)$ \\ \cline{2-5}

                   & \xrowht[()]{80pt}  \makecell[c]{$
\diag\left( 
J_{m}\left(\begin{bmatrix}
\lambda & 0 \\
0 & \overline{\lambda}
\end{bmatrix}\right), 
-J_{m}\left(\begin{bmatrix}
\lambda & 0 \\
0 & \overline{\lambda}
\end{bmatrix}\right)^\ast\right)$
\vspace*{5pt} \\
$\diag\left( 
J_{n}\left(\begin{bmatrix}
\lambda & 0 \\
0 & \overline{\lambda}
\end{bmatrix}\right), 
-J_{n}\left(\begin{bmatrix}
\lambda & 0 \\
0 & \overline{\lambda}
\end{bmatrix}\right)^\ast\right)
$ \vspace*{5pt} \\ 
$\lambda \in \mathbb{C}, \operatorname{Re}(\lambda) > 0, \operatorname{Im}(\lambda) \ge 0$  }   &\xrowht[()]{30pt} 
\makecell[c]{
$\left( \begin{bsmallmatrix}
0 & 0 \\
0 & T_{\fp}(W) \\
\ch{^{$\fp$} T}(Z)  & 0 \\
0 & 0
\end{bsmallmatrix},  
-\begin{bsmallmatrix}
0 & 0 & 0 & T_{\fp}(W^\ast) &  \\
\ch{^{$\fp$} T} (Z^\ast)  & 0  & 0 & 0
\end{bsmallmatrix}
\right)$  \vspace*{5pt} \\
$\operatorname{Im}(\lambda) > 0$: $Z_p = \begin{bsmallmatrix}
x_p \qj & y_p \\
z_p & w_p \qj
\end{bsmallmatrix}, W_p = \begin{bsmallmatrix}
c_p \qj & d_p \\
e_p & f_p \qj
\end{bsmallmatrix}$ \vspace*{5pt} \\ $x_p,y_p,z_p,w_p,c_p,d_p,e_p,f_p\in\mathbb{C}, 1\le p \le n$
}   &
\makecell[c]{    
$I_{2m} +  \frac{1}{2} 
\diag\left(
J_m\left(
\begin{bsmallmatrix}
\lambda & 0\\ 
0 & \overline{\lambda}
\end{bsmallmatrix}
\right)^2,
\left( 
J_m\left(
\begin{bsmallmatrix}
\lambda & 0\\ 
0 & \overline{\lambda}
\end{bsmallmatrix}
\right)^\ast
\right)^2
\right)
 + \begin{bsmallmatrix}
0 & T_{\fp}(W) \\
\ch{^{$\fp$} T}(Z)  & 0
\end{bsmallmatrix}$ \vspace*{5pt} \\ 
$Z^\ast = -Z, W^\ast = -W$
}   &\xrowht[()]{15pt}   $(I_{2m} \otimes H_2,I_{2n} \otimes H_2)$ \\ \cline{2-5} \hline

\multirow{6}{*}{7} 
&$\left(
J_{m}\left(\begin{bsmallmatrix}
0 & 0 \\
0 & 0
\end{bsmallmatrix}\right)
, J_{n}\left(
\begin{bsmallmatrix}
0 & 0 \\
0 & 0
\end{bsmallmatrix}
\right)
\right)$     &\xrowht[()]{25pt}  $\left(
\begin{bsmallmatrix}
0 \\ \ch{_{$\fp$} S}(Z)
\end{bsmallmatrix}, 
\kappa^{m-1} {\kappa'}^{n-1}\begin{bsmallmatrix}
\ch{_{$\fp$} S}( (F_2^{m}ZF_2^{n})^\ast ) & 0
\end{bsmallmatrix} \right) $   &   \makecell[c]{$I_{2m} + \frac{1}{2} J_m\left(
\begin{bsmallmatrix}
0 & 0 \\
0 &0
\end{bsmallmatrix}
\right)^2 + \ch{_{$\fp$} S}(Z)$ \vspace*{5pt} \\  $ (F_2^{m} Z F_2^{m})^\ast = Z$ }  &  $(\kappa^{m-1} F_2^{m}  \otimes F_{m}, {\kappa'}^{n-1} F_2^{n}\otimes  F_{n})$   \\ \cline{2-5}

                   & \xrowht[()]{15pt}  $
\left(
J_{m}\left( \begin{bmatrix}
0 & b \\
-b & 0
\end{bmatrix}\right),
J_{n}\left( \begin{bmatrix}
0 & b \\
-b & 0
\end{bmatrix}\right)
\right), b > 0$ \ &\xrowht[()]{30pt} 
\makecell[c]{
$\left( \begin{bsmallmatrix}
0  \\
\ch{_{$\fp$} S}(Z)
\end{bsmallmatrix},  
\kappa \kappa' \begin{bsmallmatrix}
\ch{_{$\fp$} S}( (F_2^{m} Z F_2^{n})^\ast ) & 0 
\end{bsmallmatrix}
\right)$  \vspace*{5pt}  \\
$Z_p = \begin{bsmallmatrix} 
x_p & y_p \\
y_p & -x_p
\end{bsmallmatrix}, 1 \le p \le n$
}& \xrowht[()]{25pt} \makecell[c]{  $I_{2m} + 
\frac{1}{2} J_m\left(
\begin{bsmallmatrix}
0 & b \\
-b &0
\end{bsmallmatrix}
\right)^2
+ \ch{_{$\fp$} S}(Z)$ \vspace*{5pt} \\ $(F_2^{m} Z F_2^{m-1})^\ast = Z$ }    &   $(\kappa F_2^{m} \otimes F_m,\kappa' F_2^{n} \otimes F_n)$ \\ \cline{2-5}

                   & \xrowht[()]{80pt}  \makecell[c]{$
\diag\left( 
J_{m}\left(\begin{bmatrix}
\lambda & 0 \\
0 & \overline{\lambda}
\end{bmatrix}\right), 
-J_{m}\left(\begin{bmatrix}
\lambda & 0 \\
0 & \overline{\lambda}
\end{bmatrix}\right)^\ast\right)$
\vspace*{5pt}
\\
$\diag\left( 
J_{n}\left(\begin{bmatrix}
\lambda & 0 \\
0 & \overline{\lambda}
\end{bmatrix}\right), 
-J_{n}\left(\begin{bmatrix}
\lambda & 0 \\
0 & \overline{\lambda}
\end{bmatrix}\right)^\ast\right)
$  
\vspace*{5pt}
\\ $\lambda \in \mathbb{C}, \operatorname{Re}(\lambda) > 0, \operatorname{Im}(\lambda) \ge 0$  }   &\xrowht[()]{30pt} 
\makecell[c]{
$\left( \begin{bsmallmatrix}
0 & 0 \\
0 & T_{\fp}(W) \\
\ch{^{$\fp$} T}(Z)  & 0 \\
0 & 0
\end{bsmallmatrix},  
\begin{bsmallmatrix}
0 & 0 & 0 & T_{\fp}(W^\ast) &  \\
\ch{^{$\fp$} T} (Z^\ast)  & 0  & 0 & 0
\end{bsmallmatrix}
\right)$ \vspace*{5pt} \\
$\operatorname{Im}(\lambda) > 0$: $Z_p = \begin{bsmallmatrix}
x_p \qj & y_p \\
z_p & w_p \qj
\end{bsmallmatrix}, W_p = \begin{bsmallmatrix}
c_p \qj & d_p \\
e_p & f_p \qj
\end{bsmallmatrix}$ \vspace*{5pt} \\ $x_p,y_p,z_p,w_p,c_p,d_p,e_p,f_p\in\mathbb{C}, 1\le p \le n
$
}    &
\makecell[c]{    
$I_{4m} +  \frac{1}{2}
\diag\left(
J_m\left(\begin{bsmallmatrix}
\lambda & 0 \\
0 & \overline{\lambda}
\end{bsmallmatrix}
\right)^2,
\left(J_m\left(\begin{bsmallmatrix}
\lambda & 0 \\
0 & \overline{\lambda}
\end{bsmallmatrix}
\right)^\ast\right)^2
\right)
 + \begin{bsmallmatrix}
0 & T_{\fp}(W) \\
\ch{^{$\fp$} T}(Z)  & 0
\end{bsmallmatrix}$ \vspace*{5pt} \\ 
$Z^\ast = Z, W^\ast = W$
}   &\xrowht[()]{15pt}   $(I_{2m} \otimes F_2,I_{2n} \otimes F_2)$ \\ \cline{2-5} 
\hline
\end{tabular}}
\caption{Candidates of $Y_{ij}$ (No.~5--No.~7)}
\label{Tab:CandidatesYijcon'd}
\end{table}
\end{lemma}
\begin{proof}
The characterization of $Y_{ij}$'s in Tables~\ref{Tab:CandidatesYij} and \ref{Tab:CandidatesYijcon'd} is obtained by solving equations \eqref{lem:block:eq1} and \eqref{lem:block:eq2}, which relies on Lemma~\ref{lem:lower triangular}. We need to split the discussion with respect to the seven cases in Table~\ref{Tab:indecomposable}. This leads to a lengthy calculation and we omit the proof here for clarity. A detailed proof can be found in Appendix~\ref{append:proof lem:Yij}.
\end{proof}

\begin{theorem}[Structure theorem]\label{thm:structure}
Assume $B\in \GL_n(\mathbb{F})$  satisfies $B^\sigma = \varepsilon B$ where $\varepsilon = \pm 1$. Let $\alpha: \mathbb{R} \to G_B(\mathbb{F})$ be a quadratic  rational curve on $G_B(\mathbb{F})$ with poles at $\pm \ci$.  There exists $R\in \GL_n(\mathbb{F})$ such that
\[
\alpha(t) =R\left( \frac{t^2  I_{n} + t \diag(X_1,\dots, X_s) + (Y_{pq})_{p,q=1}^s}{t^2 + 1} \right)R^{-1}, \quad B  = R \diag(B_1,\dots, B_s) R^{\sigma},
\]
where 
\begin{enumerate}[(i)]
\item $(X_p,B_p)\in \mathbb{F}^{m_p\times m_p}\times \GL_{m_p}(\mathbb{F})$ is as in Table~\ref{Tab:indecomposable}.
\item If $0\not\in \rho(X_p) + \rho(X_q)$ then $Y_{pq} = 0$.
\item If $0 \in \rho(X_p) + \rho(X_q)$ then $Y_{pq}$ is as in Tables~\ref{Tab:CandidatesYij} and \ref{Tab:CandidatesYijcon'd}. 
\item Moreover, $\{ Y_{pq}: 1 \le p, q \le s \}$ satisfies the equation
\begin{equation}\label{thm:structure:eq1}
\sum_{r=1}^s Y_{ur} B_r Y_{vr}^\sigma =\delta_{u,v} B_u,\quad 1 \le u \le v \le s.
\end{equation}
\end{enumerate}
In particular, if $\alpha'(0)$ has distinct eigenvalues, then $Y_{pq} = 0$ if $1 \le p \ne q \le s$ and $Y_{pp}\in G_{B_p}(\mathbb{F})$, i.e.,
\begin{equation*}
Y_{pp} B_p Y_{pp}^\sigma = B_p,\quad 1 \le p \le s.
\end{equation*}
\end{theorem}
\begin{proof}
This is a direct consequence of equation~\eqref{eq:quadratic4}, Lemmas~\ref{lem:invariance}, \ref{lem:indecomposable} and \ref{lem:Yij}. 
\end{proof}
\subsection{Quadratic  rational curves on unitary groups}
We notice that $G_B(\mathbb{F}) = \U_n$ (resp. $\mathfrak{g}_B(\mathbb{F}) = \mathfrak{u}_n$) for $(B,\mathbb{F},\sigma) = (I_n,\mathbb{C},\text{conjugate transpose})$. The lemma that follows is a basic fact.
\begin{lemma}[Normal form of skew Hermitian matrices]\label{lem:normal form skew Hermitian}
For each $A_1 \in \mathfrak{u}_n$, there exist $U\in \U_n$, positive integers $m_1,\dots, m_r$ and real numbers $\lambda_1 > \cdots > \lambda_r > 0$ such that 
\[
A_1 = \ci U \diag \left(
\lambda_1 I_{m_1}, -\lambda_1 I_{m_1}, \dots,  \lambda_r I_{m_r}, -\lambda_r I_{m_r}, 0,\dots, 0
\right) U^\ast.
\]
\end{lemma}
\begin{theorem}[Classification of quadratic  rational curves on $\U_n$]\label{thm:classification U_n}
If $\alpha$ is a quadratic  rational curve on $\U_n$ with poles at $\pm \ci$, then there exist $Q \in \U_n$, $2 \ge a_1 > \cdots > a_r > 0$ such that $\alpha(t) = Q \left( t^2  I_{n} + t A_1+ A_0 \right)/(t^2 + 1) Q^\ast$, where 
\begin{align*}
A_1 &= \ci \diag \left(
a_1 I_{m_1}, -a_1 I_{m_1}, \dots,  a_r I_{m_r}, -a_r I_{m_r}, 0,\dots, 0
\right), \\
A_0 &= \diag \left(
\begin{bsmallmatrix}
(1 -a_1^2/2) I_{m_1} & b_1 I_{m_1} \\
-b_1 I_{m_1} & (1 -a_1^2/2) I_{m_1}
\end{bsmallmatrix}, \dots,  
\begin{bsmallmatrix}
(1 -a_r^2/2) I_{m_r} & b_r I_{m_r} \\
-b_r I_{m_r}  & (1 -a_r^2/2) I_{m_r} 
\end{bsmallmatrix}, 
I_{n - 2 \sum_{j=1}^r m_r}
\right), \\
b_p &= a_p \sqrt{1-a_p^2/4},\quad 1 \le p \le r.
\end{align*}
In particular, every quadratic  rational curve $\alpha$ on $\U_n$ with poles at $\pm \ci$ can be written as $\alpha = Q \prod_{j=1}^r \beta_j Q^\ast$ where $Q\in \U_n$ and  
\[
\beta_j = \diag\left(
I_{2j-2},
\begin{bsmallmatrix}
\frac{t^2 + \ci f_j t + (1-f_j^2/2)}{t^2 + 1} & \frac{f_j \sqrt{1-f_j^2/4}}{t^2 + 1}  \\[2pt]
-\frac{f_j \sqrt{1-f_j^2/4}}{t^2 + 1} & \frac{t^2 - \ci f_j t + (1-f_j^2/2)}{t^2 + 1}
\end{bsmallmatrix} ,
I_{n - 2j}
\right),\quad f_j\in (0,2], \quad 1 \le j \le r.
\]
\end{theorem}
\begin{proof}
We write 
\[
\alpha(t) =R\left( \frac{t^2  I_{n} + t \diag(X_1,\dots, X_s) + (Y_{pq})_{p,q=1}^s}{t^2 + 1} \right)R^{-1}, \quad I_n  = R \diag(c_1 B_1,\dots, c_s B_s) R^{\ast},
\]
where $R$, $(X_1,\dots, X_s)$ and $(Y_{pq})_{p,q=1}^s$ are those given in Theorem~\ref{thm:structure}. Since eigenvalues of $A_1 \in \mathfrak{u}_n$ are pure imaginary, each $(X_p,B_p)$ must has the form $(J_m(x\ci) , \kappa \ci^{m-1} F_m)$ by Lemma~\ref{lem:Yij}. We also notice that $\diag(B_1,\dots, B_s) = R^{-1} (R^{-1})^\ast$ is positive definite. This implies that each $B_p = 1$. Consequently, we obtain $s =n$, $R\in \U_n$, $X_p = x_p \ci \in  \mathbb{R} \ci$ and $B_p = \kappa_p = 1, 1 \le p \le n$. 

According to Table~\ref{Tab:CandidatesYij}~No.~3, we have 
\[
Y_{pp} = \begin{cases}
1 + \ci y_p,&~\text{if}~x_p  = 0 \\
1 - x_p^2/2,&~\text{if}~x_p  \ne 0 \\
\end{cases}, \quad
-\overline{Y}_{qp} = Y_{pq} = 
\begin{cases}
y_{pq},&~\text{if}~x_p = -x_q \\
0,&~\text{if}~x_p \ne -x_q \\
\end{cases},
\] 
where $y_p\in \mathbb{R}, y_{pq} \in \mathbb{C}$, $1\le p \le q \le n$. Lemma~\ref{lem:normal form skew Hermitian} ensures the existence of a permutation matrix $P$ such that 
\[
P \diag(x_1 \ci ,\dots, x_n \ci ) P^\tp = 
\ci \diag \left(
a_1 I_{m_1}, -a_1 I_{m_1}, \dots,  a_r I_{m_r}, -a_r I_{m_r}, 0,\dots, 0
\right).
\]
Obviously, we have $\{x_1,\dots, x_n
\} = \{a_1,\dots, a_r\}$. It is straightforward to verify that 
\[
P (Y_{pq})_{p,q=1}^n P^\tp = 
\diag \left(
\begin{bsmallmatrix}
(1 -a_1^2/2) I_{m_1} & Z_1 \\
-Z_1^\ast & (1 -a_1^2/2) I_{m_1}
\end{bsmallmatrix}, \dots,  
\begin{bsmallmatrix}
(1 -a_r^2/2) I_{m_r} & Z_r \\
-Z_r^\ast & (1 -a_r^2/2) I_{m_r} 
\end{bsmallmatrix}, 
I_{n - 2\sum_{j=1}^r m_r} + \ci D
\right),
\]
where $D$ is a real diagonal matrix. Now \eqref{thm:structure:eq1} implies $D =0$ and $Z_p Z_p^\ast = a_p^2(1-a_p^2/4) I_{m_p}, 1\le p \le r$. In particular, we must have $a_p \in (0,2]$ since $Z_p Z_p^\ast$ is positive semidefinite. We observe that $Z_p =a_p \sqrt{1-a_p^2/4} Q_p$ for some $Q_p \in \U_{m_p}$. Thus we have  
\[
\begin{bsmallmatrix}
(1- a_p^2/2) I_{m_p} & Z_p \\
-Z_p^\ast & (1- a_p^2/2) I_{m_p}
\end{bsmallmatrix} =
\begin{bsmallmatrix}
Q_p & 0 \\
0 & I_{m_p}
\end{bsmallmatrix}
 \begin{bsmallmatrix}
(1- a_p^2/2) I_{m_p} & a_p \sqrt{1-a_p^2/4} I_{m_p} \\
-a_p \sqrt{1-a_p^2/4} I_{m_p} & (1- a_p^2/2) I_{m_p}
\end{bsmallmatrix}
\begin{bsmallmatrix}
Q_p & 0 \\
0 & I_{m_p}
\end{bsmallmatrix}^\ast
\] 
and this completes the proof.
\end{proof}
\begin{corollary}[Classification of quadratic  rational curves on $\SU_n$]\label{cor:classification SU_n}
A quadratic   rational curve on $\U_n$ is also a quadratic   rational curve on $\SU_n$.
\end{corollary}
\begin{proof}
Let $\alpha$ be a quadratic  rational curve on $\U_n$. We prove that $\det(\alpha) = 1$. Without loss of generality,  we assume that poles of $\alpha$ are $\pm \ci$. By Theorem~\ref{thm:classification U_n}, it suffices to prove 
\[
\det \left( 
\begin{bsmallmatrix}
(t^2 + \ci a t+ (1-a^2/2) ) I_m & a \sqrt{1-a^2/4} I_m \\
-a \sqrt{1-a^2/4} I_m & (t^2 - \ci a t + (1-a^2/2) ) I_m
\end{bsmallmatrix}
\right) = (t^2 + 1)^{2m},\quad a\in [0,2].
\] 
\end{proof}

\begin{example}\label{ex:U_n}
According to Theorem~\ref{thm:classification U_n}, up to a conjugation, a quadratic  rational curves on $\U_n$ and $\SU_n$ is
\[
\alpha(t) = \begin{bsmallmatrix}
\frac{t^2 + \ci at + (1-a^2/2)}{t^2 + 1} & \frac{a\sqrt{1-a^2/4}}{t^2 + 1}  \\[2pt]
-\frac{a \sqrt{1-a^2/4}}{t^2 + 1} & \frac{t^2 - \ci at + (1-a^2/2)}{t^2 + 1}
\end{bsmallmatrix} \\
\quad \text{or}\quad 
\alpha(t) = \begin{bsmallmatrix}
\frac{t^2 + \ci at + (1-a^2/2)}{t^2 + 1} & \frac{a\sqrt{1-a^2/4}}{t^2 + 1} & 0  \\[2pt]
-\frac{a \sqrt{1-a^2/4}}{t^2 + 1} & \frac{t^2 - \ci at + (1-a^2/2)}{t^2 + 1} & 0 \\
0 & 0 & 1
\end{bsmallmatrix},
\]
depending on $n = 2$ or $3$. For comparison, a quadratic  rational curve on $\U_4$ and $\SU_4$ can be written (up to a conjugation by some $R\in \U_4$) as
\begin{align*}
\alpha(t) &= 
\begin{bsmallmatrix}
\frac{t^2 + \ci at + (1-a^2/2)}{t^2 + 1} & \frac{a\sqrt{1-a^2/4}}{t^2 + 1} & 0   & 0 \\[2pt]
-\frac{a\sqrt{1-a^2/4}}{t^2 + 1}  & \frac{t^2 - \ci at + (1-a^2/2)}{t^2 + 1} & 0  & 0 \\[2pt]
0& 0 & 1 & 0 \\[2pt]
0 & 0 & 0 & 1
\end{bsmallmatrix}
\begin{bsmallmatrix}
1& 0 & 0   & 0 \\[2pt]
0  & 1 & 0  & 0 \\[2pt]
0& 0 & \frac{t^2 + \ci bt + (1-b^2/2)}{t^2 + 1} & \frac{b\sqrt{1-b^2/4}}{t^2 + 1} \\[2pt]
0 & 0 & -\frac{b\sqrt{1-b^2/4}}{t^2 + 1} & \frac{t^2 - \ci bt + (1-b^2/2)}{t^2 + 1}
\end{bsmallmatrix},
\quad a,b \in [0,2].
\end{align*}
\end{example}
\subsection{Quadratic  rational curves on real orthogonal groups}
We recall that for $(B, \mathbb{F}, \sigma)  = (I_n, \mathbb{R}, \text{transpose})$, $G_B(\mathbb{F})$ is the real orthogonal group $\O_n(\mathbb{R})$. Thus $\mathfrak{g}_B(\mathbb{F}) = \mathfrak{o}_n(\mathbb{R})$ consists of $n\times n$ skew symmetric matrices. The following normal form of skew symmetric matrices is well-known.  It can also be obtained from Table~\ref{Tab:indecomposable} (No.~4) by observing the signature of $B = I_n$ is $(n,0)$. 
\begin{lemma}[Normal form of skew symmetric matrices]\label{lem:normal form skew symmetric}
Given $A \in \mathfrak{o}_n(\mathbb{R})$, there exists $R\in \O_n(\mathbb{R})$ such that $A  = R  \diag(X_1,\dots, X_s) R^{\tp}$, where for each $1\le j \le s$, either $X_j = 0\in \mathbb{R}$ or $X_j = 
\begin{bsmallmatrix}
0 & b_j \\
-b_j & 0 
\end{bsmallmatrix}\in \mathbb{R}^{2\times 2}$ for some $b_j >0$. Moreover, we may require that
$\rho(X_i) = \rho(X_j)$ implies $\rho(X_i) = \rho(X_k)$ for each $1\le i \le k \le j \le s$.  
\end{lemma}
\begin{corollary}\label{cor:normal form skew symmetric}
Let $\sigma$ be the transpose and let $B = I_n$. For each solution $(A_0,A_1) \in \mathbb{R}^{n\times n} \times \mathbb{R}^{n\times n}$ of \eqref{eq:quadratic1}--\eqref{eq:quadratic3}, there exist $R\in \O_n(\mathbb{R})$ and positive numbers $\lambda_1 > \dots >  \lambda_r$ such that 
\begin{align*}
A_1  &= R   \diag\left(
\lambda_1 \begin{bsmallmatrix}
0 & I_{m_1} \\
-I_{m_1} & 0 
\end{bsmallmatrix},\dots, \lambda_r \begin{bsmallmatrix}
0 & I_{m_r} \\
-I_{m_r} & 0 
\end{bsmallmatrix}, 0,\dots, 0
\right) R^{\tp}, \\
A_0 &= R \diag\left(
\begin{bsmallmatrix}
H_1 & G_1 \\
G_1 & -H_1
\end{bsmallmatrix} + (1 - \frac{\lambda_1^2}{2}) I_{2m_1}
,\dots, \begin{bsmallmatrix}
H_r & G_r \\
G_r & -H_r
\end{bsmallmatrix}  + (1 - \frac{\lambda_r^2}{2}) I_{2m_r},
I_{n - 2\sum_{j=1}^r m_j} + \Lambda 
\right) R^{\tp},
\end{align*}
where $\Lambda\in \mathfrak{o}_{n - 2 \sum_{j=1}^r m_j}(\mathbb{R})$, $H_p, G_p \in \mathfrak{o}_{m_j}(\mathbb{R})$ for each $1 \le p \le r$. 
\end{corollary}
\begin{proof}
By Lemma~\ref{lem:normal form skew symmetric}, there exists $Q \in \O_n(\mathbb{R})$ such that $Q^\tp A_1  Q=  \diag(X_1,\dots, X_s)$, where for each $1\le j \le s$, either $X_j = 0\in \mathbb{R}$ or $X_j = 
\begin{bsmallmatrix}
0 & b_j \\
-b_j & 0 
\end{bsmallmatrix}\in \mathbb{R}^{2\times 2}$ for some $b_j >0$. Moreover, we have $\rho(X_j) = \rho(X_k)$ implies $\rho(X_j) = \rho(X_l)$ for each $1\le j \le l \le k \le s$. Since $(A_0, A_1)$ is a solution of \eqref{eq:quadratic1}--\eqref{eq:quadratic3}, Lemma~\ref{lem:Yij} (No.~4 of Table~\ref{Tab:CandidatesYij}) implies that $Q^\tp A_0 Q  = (Y_{ij}) $ where for $1 \le i,j\le s$, we have
\[
Y_{ii} = \begin{cases}
1,&~\text{if}~X_i = 0, \\
(1-b^2/2) I_2&~\text{if}~X_i = 
\begin{bsmallmatrix}
0 & b \\
-b & 0 
\end{bsmallmatrix} \\
\end{cases},\quad
-Y_{ji}^\tp = Y_{ij} = \begin{cases}
\begin{bsmallmatrix}
0 & 0 
\end{bsmallmatrix}
,&~\text{if}~(X_i,X_j) = \left(0, \begin{bsmallmatrix}
0 & b \\
-b & 0 
\end{bsmallmatrix} \right) \\
\begin{bsmallmatrix}
0 \\ 0 
\end{bsmallmatrix},&~\text{if}~(X_i,X_j) = \left( 
\begin{bsmallmatrix}
0 & b \\
-b & 0 
\end{bsmallmatrix},0 \right) \\
x,&~\text{if}~(X_i,X_j) = (0,0) \\
\begin{bsmallmatrix}
0 & 0 \\
0 & 0
\end{bsmallmatrix},&~\text{if}~
(X_i,X_j) = \left( 
\begin{bsmallmatrix}
0 & b \\
-b & 0 
\end{bsmallmatrix}, 
\begin{bsmallmatrix}
0 & b' \\
-b' & 0 
\end{bsmallmatrix}
\right), b \ne b'\\
\begin{bsmallmatrix}
x & y \\
y & -x
\end{bsmallmatrix},&~\text{if}~
(X_i,X_j) = \left( 
\begin{bsmallmatrix}
0 & b \\
-b & 0 
\end{bsmallmatrix}, 
\begin{bsmallmatrix}
0 & b \\
-b & 0 
\end{bsmallmatrix}
\right)
\end{cases}.
\]

Next we observe that there exist positive integers $m_1,\dots, m_r$ and positive real numbers $\lambda_1 >\cdots > \lambda_r$ such that
\[
\rho(X_{m_{p-1} + 1}) = \cdots = \rho(X_{m_p}) = \{ \ci \lambda_p, - \ci \lambda_p\},\quad 1\le p \le r.
\]
Here we adopt the convention that $m_0 = 0$. Indeed, we have $\{\lambda_1,\dots, \lambda_r\} = \{b_1,\dots, b_s\}$. Thus, $(A_0,A_1)$ can be written as 
\begin{align*}
Q^\tp A_0 Q  &=  \diag\left(
(1- \lambda_1^2/2) I_{2m_1} + \Lambda_1,\dots, (1-\lambda_r^2/2) I_{2m_r} + \Lambda_r, I_{m_{r+1}} + \Lambda_{r+1}
\right) ,  \\
Q^\tp A_1 Q  &=   \diag\left(
\lambda_1 \begin{bsmallmatrix}
0 & 1  & \cdots & 0 & 0 \\
-1 & 0  & \cdots & 0 & 0 \\
\vdots & \vdots  & \ddots & \vdots & \vdots \\
0 & 0  & \cdots & 0 & 1 \\
0 & 0  & \cdots & -1 & 0 \\
\end{bsmallmatrix},\dots, \lambda_r \begin{bsmallmatrix}
0 & 1  & \cdots & 0 & 0 \\
-1 & 0  & \cdots & 0 & 0 \\
\vdots & \vdots  & \ddots & \vdots & \vdots \\
0 & 0  & \cdots & 0 & 1 \\
0 & 0  & \cdots & -1 & 0 \\
\end{bsmallmatrix}, 0,\dots, 0
\right).
\end{align*}
Here $\Lambda_1,\dots, \Lambda_{r+1}$ are skew symmetric matrices and $m_{r+1} = n - 2\sum_{j=1}^r m_j$. Moreover, $\Lambda_1,\dots, \Lambda_{r}$ are block matrices of which each block has the form $\begin{bsmallmatrix}
x & y  \\
y & -x \end{bsmallmatrix}$. Clearly, there are permutation matrices $P_1 \in \mathbb{R}^{m_1\times m_1},\dots, P_r \in \mathbb{R}^{m_r \times m_r}$ such that 
\[
\begin{bsmallmatrix}
P_1 & \cdots & 0 \\
\vdots & \ddots & \vdots \\
0 & \cdots & P_r
\end{bsmallmatrix}^\tp Q^\tp  A_1 Q \begin{bsmallmatrix}
P_1 & \cdots & 0 \\
\vdots & \ddots & \vdots \\
0 & \cdot & P_r
\end{bsmallmatrix} = 
\diag\left(
\lambda_1 \begin{bsmallmatrix}
0 & I_{m_1} \\
-I_{m_1} & 0 
\end{bsmallmatrix},\dots, \lambda_r \begin{bsmallmatrix}
0 & I_{m_r} \\
-I_{m_r} & 0 
\end{bsmallmatrix}, 0,\dots, 0
\right).
\]
It is straightforward to verify that $R^\tp A_0 R$ has the desired form where $R = Q \begin{bsmallmatrix}
P_1 & \cdots & 0 \\
\vdots & \ddots & \vdots \\
0 & \cdot & P_r
\end{bsmallmatrix}$.
\end{proof}

\begin{theorem}[Classification of quadratic  rational curves on $\O_n(\mathbb{R})$]\label{thm:classification O_n}
Let $\alpha: \mathbb{R} \to \O_n(\mathbb{R})$ be a quadratic  rational curve with poles at $\pm \ci$. There exist $R\in \O_n(\mathbb{R})$, $2 = \lambda_0 > \lambda_1 > \cdots > \lambda_r > 0$ and $n_0 \ge 0, n_1,\dots, n_r > 0$ such that $\alpha(t) = R (t^2 I_n + t A_1 + A_0)/(t^2 + 1) R^\tp$ where 
\begin{align*}
A_1 &= \diag\left(
\lambda_0 \begin{bsmallmatrix}
0 & I_{n_0} \\
-I_{n_0} & 0 
\end{bsmallmatrix},
\lambda_1 \begin{bsmallmatrix}
0 & I_{2 n_1} \\
-I_{2n_1} & 0 
\end{bsmallmatrix},\dots, \lambda_r \begin{bsmallmatrix}
0 & I_{2n_r} \\
-I_{2n_r} & 0 
\end{bsmallmatrix}, 0,\dots, 0
\right),\\
A_0 &= \diag\left(
- I_{2n_0},
\mu_1 \begin{bsmallmatrix}
H_1 & G_1 \\
G_1 & -H_1 
\end{bsmallmatrix} + \left(1-\lambda_1^2/2 \right) I_{4n_1},\dots, \mu_r \begin{bsmallmatrix}
H_r & G_r \\
G_r & -H_r 
\end{bsmallmatrix} + \left(1-\lambda_r^2/2\right) I_{4 n_r}, I_{n-2n_0 - 4\sum_{j=1}^r n_j}
\right),\\
H_p &= \diag\left( 
h_{p,1} \begin{bsmallmatrix}
0 & 1 \\
- & 0 
\end{bsmallmatrix},\dots, 
h_{p,n_p} \begin{bsmallmatrix}
0 & 1 \\
- & 0 
\end{bsmallmatrix} \right),\quad G_p = \diag\left( 
g_{p,1} \begin{bsmallmatrix}
0 & 1 \\
- & 0 
\end{bsmallmatrix},\dots, 
g_{p,n_p} \begin{bsmallmatrix}
0 & 1 \\
- & 0 
\end{bsmallmatrix} \right),\quad 1 \le p \le r, \\
\mu_p &= \lambda_p \sqrt{1-\lambda_p^2/4}, \quad h_{p,1}^2 + g_{p,1}^2 = \cdots  = h_{p,n_p}^2 + g_{p,n_p}^2  = 1,\quad 1 \le p \le r. 
\end{align*}

\end{theorem}
\begin{proof}
By definition, we may parametrize $\alpha$ as $\alpha(t) = (t^2 I_n + t A_1 + A_0)/(t^2 + 1)$ for some $A_0,A_1$ satisfying \eqref{eq:quadratic1}--\eqref{eq:quadratic4}. By Corollary~\ref{cor:normal form skew symmetric}, it suffices to assume 
\begin{align*}
A_1  &= \diag\left(
\lambda_0 \begin{bsmallmatrix}
0 & I_{m_0} \\
-I_{m_0} & 0 
\end{bsmallmatrix},\dots, \lambda_r \begin{bsmallmatrix}
0 & I_{m_r} \\
-I_{m_r} & 0 
\end{bsmallmatrix}, 0,\dots, 0
\right), \\
A_0 &=  \diag\left(
\begin{bsmallmatrix}
H_0 & G_0 \\
G_0 & -H_0
\end{bsmallmatrix} + (1 - \frac{\lambda_0^2}{2}) I_{2m_0}
,\dots, \begin{bsmallmatrix}
H_r & G_r \\
G_r & -H_r
\end{bsmallmatrix}  + (1 - \frac{\lambda_r^2}{2}) I_{2m_r},
I_{n - 2\sum_{j=0}^r m_j} + \Lambda 
\right),
\end{align*}
where $\Lambda\in \mathfrak{o}_{n - 2 \sum_{j=1}^r m_j}(\mathbb{R})$, $H_p, G_p \in \mathfrak{o}_{m_j}(\mathbb{R})$ and $\lambda_p \in \mathbb{R}$ for each $0 \le p \le r$. According to \eqref{eq:quadratic4}, we have $A_0 A_0^\tp = I_n$, which implies 
\begin{align*}
\left( 
\begin{bsmallmatrix}
H_p & G_p \\
G_p & -H_p
\end{bsmallmatrix} + (1 - \frac{\lambda_p^2}{2}) I_{2m_p}
\right) \left( 
-\begin{bsmallmatrix}
H_p & G_p \\
G_p & -H_p
\end{bsmallmatrix} + (1 - \frac{\lambda_p^2}{2}) I_{2m_p}
\right) &= I_{2m_p}, \quad 0 \le p \le r, \\
(I_{n - 2\sum_{j=1}^r m_j} + \Lambda ) (I_{n - 2\sum_{j=1}^r m_j} - \Lambda) &= I_{n - 2\sum_{j=0}^r m_j}. 
\end{align*}
Thus $\Lambda = 0$ and for each $0 \le p \le r$, $H_p^2 + G_p^2 = - \mu_p^2 I_{m_p}, H_p G_p = G_p H_p$. Since $G_p$ and $H_p$ are commuting skew symmetric matrices, there exists some $R_p\in \O_{m_p}(\mathbb{R})$ such that
\begin{align*}
H_p &= R_p \diag\left( 
h_{p,1} \begin{bsmallmatrix}
0 & 1 \\
-1 & 0 
\end{bsmallmatrix},\dots, 
h_{p,n_p} \begin{bsmallmatrix}
0 & 1 \\
-1 & 0 
\end{bsmallmatrix}, \underbrace{0,\dots,0}_{m_p -2\sum_{j=1}^p n_p~\text{times}} \right)R_p^\tp, \\ 
G_p &=  R_p \diag\left( 
g_{p,1} \begin{bsmallmatrix}
0 & 1 \\
-1 & 0 
\end{bsmallmatrix},\dots, 
g_{p,n_p} \begin{bsmallmatrix}
0 & 1 \\
-1 & 0 
\end{bsmallmatrix}, \underbrace{0,\dots,0}_{m_p -2\sum_{j=1}^p n_p~\text{times}} \right) R_p^\tp.
\end{align*}
Since $H_p^2 + G_p^2 = -\mu_p^2 I_{m_p}$, we conclude that if $p \ge 1$ then $m_p = 2n_p$ and $h_{p,1}^2 + g_{p,1}^2 = \cdots  = h_{p,n_p}^2 + g_{p,n_p}^2  = \mu_p^2$, while $H_0 = G_0 = 0$.
\end{proof}

\begin{example}\label{ex:O_n}
For $n \le 3$, there is only one type of non-constant quadratic  rational curves with poles at $\pm \ci$ on $\O_n(\mathbb{R})$. In fact, up to a conjugation by some $R\in \O_n(\mathbb{R})$, such a curve can be written as $\alpha(t) =  (t^2 I_n+ t A_1 + A_0)/(t^2 + 1)$ where $A_0,A_1$ are block diagonal matrices characterized in Theorem~\ref{thm:classification O_n}. Since $n \le 3$, we must have $n_0 = 1$. This implies 
\[
\alpha(t) = \begin{bmatrix}
\frac{t^2 - 1}{t^2 + 1} & \frac{2t}{t^2 + 1} \\[2pt]
-\frac{2t}{t^2 + 1} & \frac{t^2 - 1}{t^2 + 1} 
\end{bmatrix} \quad \text{or} \quad \alpha(t) = \begin{bmatrix}
\frac{t^2 - 1}{t^2 + 1} & \frac{2t}{t^2 + 1} & 0 \\[2pt]
-\frac{2t}{t^2 + 1} & \frac{t^2 - 1}{t^2 + 1} & 0 \\[2pt]
0 & 0 & 1
\end{bmatrix},
\]
depending on $n =2$ or $3$. However, there are two families of non-constant quadratic  rational curves with poles at $\pm \ci$ on $\O_4(\mathbb{R})$. According to Theorem~\ref{thm:classification O_n}, we have 
\[
\alpha_1(t) = \begin{bsmallmatrix}
\frac{t^2 + (1-\lambda^2/2)}{t^2 + 1}  & \frac{\mu h}{t^2 + 1}  & \frac{\lambda t}{t^2 + 1} & \frac{\mu g}{t^2 + 1}  \\[2pt]
-\frac{\mu h}{t^2 + 1}   & \frac{t^2 + (1-\lambda^2/2)}{t^2 + 1}   &  -\frac{\mu g}{t^2 + 1} & \frac{\lambda t}{t^2 + 1} \\[2pt]
- \frac{\lambda t}{t^2+1 } & \frac{\mu g}{t^2 + 1} & \frac{t^2 + (1-\lambda^2/2) }{t^2 + 1}  & -\frac{\mu h}{t^2 + 1}  &  \\[2pt]
-\frac{\mu g}{t^2 + 1} & -\frac{\lambda t}{t^2 + 1} & \frac{\mu h}{t^2 + 1} & \frac{t^2 + (1-\lambda^2/2) }{t^2 + 1} 
\end{bsmallmatrix} \quad \text{and} \quad   
\alpha_2(t) =  \begin{bsmallmatrix}
\frac{t^2 - 1}{t^2 + 1} & \frac{2t}{t^2 + 1} & 0 & 0 \\[2pt]
-\frac{2t}{t^2 + 1} & \frac{t^2 - 1}{t^2 + 1} & 0 & 0 \\[2pt]
0 & 0 & 1 & 0 \\[2pt]
0 & 0 & 0 & 1
\end{bsmallmatrix},
\]
where $\lambda \in (0,2]$,  $\mu = \lambda \sqrt{1- \frac{\lambda^2}{4}}$ and $g^2 + h^2 = 1$. We notice that curves of type $\alpha_1$ consist an infinite family parametrized by $(\lambda, g, h) \in (0,2] \times \mathbb{S}^1$, while the family of type $\alpha_2$ is just a singleton. Clearly, a curve of type $\alpha_1$ is contained in the maximal normal subgroup $S \subseteq \SO_4(\mathbb{R})$ consisting of right isoclinic rotations and a curve of type $\alpha_2$ is contained in the subgroup $\SO_2(\mathbb{R}) \subseteq \SO_4(\mathbb{R})$. 
\end{example}
\subsection{Quadratic  rational curves on real indefinite orthogonal groups}
Next we consider $(B,\mathbb{F}, \sigma) = (I_{n,1},\mathbb{R}, \text{transpose})$. In this case, $G_B(\mathbb{F})$ is $\O_{n,1}(\mathbb{R})$, which is arguably the most important indefinite orthogonal group.
\begin{theorem}[Classification of quadratic  rational curves on $\O_{n,1}(\mathbb{R})$]\label{thm:classification O_n1}
Let $\alpha$ be a quadratic  rational curve on $\O_{n,1}(\mathbb{R})$. We denote 
\[
Q_3 \coloneqq \begin{bsmallmatrix}
0 & \frac{\sqrt{2}}{2} &  -\frac{\sqrt{2}}{2} \\
1 & 0 & 0 \\
0 & -\frac{\sqrt{2}}{2} & -\frac{\sqrt{2}}{2}
\end{bsmallmatrix}.
\] 
Then $\alpha$ is one of the following:
\begin{enumerate}[(i)]
\item\label{thm:classification O_n1:item1} There exist $P \in \O_{n,1}(\mathbb{R})$, an integer $0 \le r < n/2$ and a quadratic  rational curve $\beta$ on $\O_{2r}(\mathbb{R})$ with poles at $\pm \ci$ satisfying $\rank (\beta'(0)) = 2r$ such that
\[
\alpha(t) = P \begin{bsmallmatrix}
\beta(t) & 0 \\
0 &  I_{n+1-2r}
\end{bsmallmatrix} P^{-1}.
\] 
\item\label{thm:classification O_n1:item2} There exist $P \in \O_{n,1}$, $y \in \mathbb{S}^{n - 2r - 3}$, an integer $0 \le r \le  (n-3)/2$, a quadratic  rational curve $\beta$ on $\O_{2r}(\mathbb{R})$ with poles at $\pm \ci$ satisfying $\rank (\beta'(0)) = 2r$ such that
\[
\alpha(t) 
= P \begin{bsmallmatrix}
I_{n-2} & 0 \\
0 &  Q_3^\tp
\end{bsmallmatrix}
\begin{bsmallmatrix}
\beta(t) & 0 & 0 & 0  & 0\\
0  &  I_{n-2r - 2} & \frac{y}{t^2 + 1} & 0 & 0  \\
0  & 0 & 1 &  0 & 0  \\ 
0  & 0 & \frac{t}{t^2 + 1}  & 1\\
0  & \frac{y^\tp}{t^2+1} & \frac{1}{2(t^2 + 1)}  &  \frac{t}{t^2 + 1} & 1  
\end{bsmallmatrix}  
\begin{bsmallmatrix}
I_{n-2} & 0 \\
0 &  Q_3
\end{bsmallmatrix} P^{-1}.
\] 
\end{enumerate}
\end{theorem}
\begin{proof}
By Theorem~\ref{thm:structure}, there exists $R\in \GL_{n+1}(\mathbb{R})$ such that 
\[
\alpha(t) =R\left( \frac{t^2  I_{n+1} + t \diag(X_1,\dots, X_{s+1}) + (Y_{pq})_{p,q=1}^{s+1}}{t^2 + 1} \right)R^{-1}, \quad I_{n,1}  = R \diag(B_1,\dots, B_{s+1}) R^{\tp},
\]
where $(X_1,\dots, X_{s+1})$ and $(Y_{pq})_{p,q=1}^{s+1}$ are those in Table~\ref{Tab:indecomposable} and Table~\ref{Tab:CandidatesYij}~No.~4, respectively. 
By Sylvester's law of inertia, the congruence action does not change the signature of a symmetric matrix. Thus, $B_{s+1}$ has signature $(p_{s+1},q_{s+1}) = (p_{s+1},1)$ and $B_j$ has signature $(p_j, q_j) = (p_j,0)$ for $1 \le j \le s $. Furthermore, we have $(n,1) = \sum_{j=1}^{s+1} (p_j, q_j)$. 

According to Table~\ref{Tab:indecomposable}~No.~4, we have $p_j - q_j \in \{\pm 1, 0 , \pm 2\}, 1 \le j \le s+1$. This implies 
\[
p_{s+1} \in \{0,1,2,3\},\quad p_j \in \{1,2\},\quad 1 \le j \le s.
\]
A closer investigation indicates that $p_{s+1} \ne 3$. Therefore, $(p_j,q_j) = (1,0)$ or $(2,0)$ for $1 \le j \le s$. Correspondingly, $(X_j, B_j) = (0, \kappa_j)$ or $\left( \begin{bsmallmatrix}
0 & b_j \\
-b_j & 0
\end{bsmallmatrix}, \kappa_j I_2 \right)$, $b_j > 0$. We rearrange $(X_j,B_j)$'s so that $X_j = X_k, j \le k$ implies $X_j = X_l$ for any $j \le l \le k$. Moreover, we have 
\begin{enumerate}[(a)]
\item If $(p_{s+1},q_{s+1}) = (0,1)$ then $(X_{s+1}, B_{s+1}) =( 0, -1)$.
\item If $(p_{s+1},q_{s+1}) = (1,1)$ then $(X_{s+1}, B_{s+1}) = \left( \begin{bsmallmatrix}
\lambda & 0 \\
0 & -\lambda
\end{bsmallmatrix} , H_2 \right)$, $\lambda > 0$.
\item If $(p_{s+1},q_{s+1}) = (2,1)$ then $(X_{s+1}, B_{s+1}) =( J_3(0), -  F_3)$.
\end{enumerate} 
As a consequence, 
we obtain 
\[
I_{n,1} = \begin{cases}
R I_{n,1} R^\tp,&~\text{if}~(p_{s+1},q_{s+1}) = (0,1) \\
R \diag\left(I_{n-2}, H_2 \right) R^\tp,&~\text{if}~(p_{s+1},q_{s+1}) = (1,1) \\
R \diag\left(I_{n-3}, - F_3\right) R^\tp,&~\text{if}~(p_{s+1},q_{s+1}) = (2,1)
\end{cases}.
\]

Assume $(p_{s+1},q_{s+1}) = (0,1)$ and $(X_{s+1}, B_{s+1}) =( 0, -1)$. By Table~\ref{Tab:CandidatesYij}~No.~4, we have 
\[
Y_{s+1,s+1} = 1, Y_{s+1,j} = Y_{j,s+1}, Y_{j,j'} = - Y_{j',j}, Y_{j,j} = 1
\]
for $1\le  j \ne  j' \le s$ such that $X_j = X_{j'} = 0$. According to \eqref{thm:structure:eq1} of Theorem~\ref{thm:structure}, we obtain $Y_{s+1,j} = Y_{j,s+1} = Y_{j,j'} = Y_{j',j} = 0$. In this case, we have 
\[
\alpha(t) = R \begin{bsmallmatrix}
\beta(t) & 0 \\
0 &  I_{n-2r}
\end{bsmallmatrix} R^{-1}, 
\]
where $R \in O_{n,1}(\mathbb{R})$ and $\beta$ is a quadratic  rational curve on $\O_{2r}(\mathbb{R})$ with poles at $\pm \ci$ for some $r < n/2$ and $\rank (\beta'(0)) = 2r$.

Assume $(p_{s+1},q_{s+1}) = (1,1)$ and $(X_{s+1}, B_{s+1}) = \left( \begin{bsmallmatrix}
\lambda & 0 \\
0 &  -\lambda
\end{bsmallmatrix}, H_2 \right)$, $\lambda > 0$. Table~\ref{Tab:CandidatesYij}~No.~4 implies $Y_{s+1,s+1} = (1 + \lambda^2/2) I_2$. By Theorem~\ref{thm:structure}, we must have $Y_{s+1,s+1} H_2 Y_{s+1,s+1}^\tp = H_2$, which contradicts to the assumption that $\lambda > 0$. 

Assume $(p_{s+1},q_{s+1}) = (2,1)$ and $(X_{s+1}, B_{s+1}) = \left(J_3(0), -F_3 \right)$. Table~\ref{Tab:CandidatesYij}~No.~4 indicates that
\[
Y_{s+1,s+1} = \begin{bsmallmatrix}
1 & &  \\
 & 1 & \\
\frac{1}{2} & & 1
\end{bsmallmatrix}, \quad
Y_{s+1,j} = \begin{bsmallmatrix}
0  \\
0 \\
y_{s+1,j}\\
\end{bsmallmatrix},\quad
Y_{j,s+1} = \begin{bsmallmatrix}
y_{s+1,j} & 0 & 0
\end{bsmallmatrix}, \quad 
Y_{j,j'} = - Y_{j',j}, \quad  
Y_{j,j} = 1
\]
for $1\le  j \ne  j' \le s$ such that $B_j = B_{j'} = 0$. Theorem~\ref{thm:structure} implies $Y_{j,j'} = Y_{j',j} = 0$ and $\sum_{j} y_{s+1,j}^2 = 1$. Since $-F_3 = Q_3 I_{2,1} Q_3^\tp$, we obtain 
\[
\alpha(t) 
= P \begin{bsmallmatrix}
I_{n-2} & 0 \\
0 &  Q_3^\tp
\end{bsmallmatrix}
\begin{bsmallmatrix}
\beta(t) & 0 & 0 & 0  & 0\\
0  &  I_{n-2r - 2} & \frac{y}{t^2 + 1} & 0 & 0  \\
0  & 0 & 1 &  0 & 0  \\ 
0  & 0 & \frac{t}{t^2 + 1}  & 1 & 0\\
0  & \frac{y^\tp}{t^2+1} & \frac{1}{2(t^2 + 1)}  &  \frac{t}{t^2 + 1} & 1  
\end{bsmallmatrix}  
\left(
P\begin{bsmallmatrix}
I_{n-2} & 0 \\
0 &  Q_3^\tp
\end{bsmallmatrix}\right)^{-1}
\]
where $P \in \O_{n,1}$, $y \in \mathbb{S}^{n - 2r - 3}, r < (n-3)/2$ and $\beta$ is a quadratic  rational curve on $\O_{2r}(\mathbb{R})$ with poles at $\pm \ci$ and $\rank (\beta'(0)) = 2r$. 
\end{proof}
\begin{remark}
Curves of type \eqref{thm:classification O_n1:item1} in Theorem~\ref{thm:classification O_n1} are obtained by the natural inclusion $\O_{n}(\mathbb{R}) \subseteq \O_{n,1}(\mathbb{R})$.  
\end{remark}
\begin{example}\label{ex:O_n1}
All quadratic rational curves on $\O_{2,1}(\mathbb{R})$ are obtained from the inclusion $\O_{2}(\mathbb{R}) \subseteq \O_{2,1}(\mathbb{R})$.  However, for $n \ge 3$, curves of type \eqref{thm:classification O_n1:item2} in Theorem~\ref{thm:classification O_n1} appear. For instance, a curve of type~\eqref{thm:classification O_n1:item2} on $\O_{3,1}(\mathbb{R})$ has the form
\[
\alpha(t) 
= P 
\begin{bsmallmatrix}
1 & 0 & 0 & 0 \\
0& 0 & 1 &  0 \\
0 & \frac{\sqrt{2}}{2} & 0 & -\frac{\sqrt{2}}{2} \\
0 & -\frac{\sqrt{2}}{2} & 0  & -\frac{\sqrt{2}}{2}
\end{bsmallmatrix}
\begin{bsmallmatrix}
 1 & \frac{y}{t^2 + 1} & 0 & 0  \\
 0 & 1 &  0 & 0  \\ 
0 & \frac{t}{t^2 + 1}  & 1 & 0\\
\frac{y}{t^2+1} & \frac{1}{2(t^2 + 1)}  &  \frac{t}{t^2 + 1} & 1  
\end{bsmallmatrix}  
\begin{bsmallmatrix}
1 & 0 & 0 & 0 \\
0& 0 & \frac{\sqrt{2}}{2} &  -\frac{\sqrt{2}}{2} \\
0 & 1 & 0 & 0 \\
0 & 0 & -\frac{\sqrt{2}}{2} & -\frac{\sqrt{2}}{2}
\end{bsmallmatrix}
P^{-1},\quad y = \pm 1, P \in \O_{3,1}(\mathbb{R}).
\] 
We also notice that the \emph{conformal rotation}  \cite{Dorst2016construction, Kalkan2022study} on $\O_{2,1}(\mathbb{R})$ is of type~\eqref{thm:classification O_n1:item1}:
\[
\alpha(t) = \begin{bsmallmatrix}
\frac{t^2 - 1}{t^2 + 1} & -\frac{5t}{2(t^2 + 1)} & -\frac{3t}{2(t^2 + 1)} \\
\frac{5t}{2(t^2 + 1)} & \frac{8t^2 - 17}{8(t^2 + 1)} & -\frac{15}{8(t^2 + 1)} \\
-\frac{3t}{2(t^2 + 1)} & -\frac{15}{8(t^2 + 1)} & \frac{8t^2 + 17}{8(t^2 + 1)}
\end{bsmallmatrix} =  
P \begin{bsmallmatrix}
\frac{t^2-1}{t^2 + 1} & \frac{2t}{t^2 + 1} & 0 \\
-\frac{2t}{t^2 + 1}  & \frac{t^2-1}{t^2 + 1}  & 0 \\
0 & 0 & 1
\end{bsmallmatrix} P^{-1},
\]
where $P = \begin{bsmallmatrix}
0 & -1 & 0 \\
-\frac{5}{4}  & 0   &  -\frac{3}{4} \\
\frac{3}{4} & 0 & \frac{5}{4}  
\end{bsmallmatrix} \in \O_{2,1}(\mathbb{R})$. The \emph{circular translation} \cite{Hegedus2013factorization,Li2024quadratic}  on $\O_{3,1}(\mathbb{R})$ is of type~\eqref{thm:classification O_n1:item2}: 
\[
\alpha(t) = 
\begin{bsmallmatrix}
1 & 0 & \frac{t}{t^2 + 1} & -\frac{t}{t^2 + 1} \\
0 & 1 & -\frac{1}{t^2 + 1} & \frac{1}{t^2 + 1} \\
-\frac{t}{t^2 + 1} & \frac{1}{t^2 + 1} & \frac{2t^2 + 1}{2(t^2 + 1)} & \frac{1}{2(t^2 + 1)} \\
-\frac{t}{t^2 + 1} & \frac{1}{t^2 + 1} & -\frac{1}{2(t^2 + 1)} & \frac{2t^2 + 3}{2(t^2 + 1)}
\end{bsmallmatrix}
= 
P \begin{bsmallmatrix}
1 & 0 & 0 & 0 \\
0& 0 & 1 &  0 \\
0 & \frac{\sqrt{2}}{2} & 0 & -\frac{\sqrt{2}}{2} \\
0 & -\frac{\sqrt{2}}{2} & 0  & -\frac{\sqrt{2}}{2}
\end{bsmallmatrix}
\begin{bsmallmatrix}
 1 & \frac{1}{t^2 + 1} & 0 & 0  \\
 0 & 1 &  0 & 0  \\ 
0 & \frac{t}{t^2 + 1}  & 1 & 0\\
\frac{1}{t^2+1} & \frac{1}{2(t^2 + 1)}  &  \frac{t}{t^2 + 1} & 1  
\end{bsmallmatrix}  
\begin{bsmallmatrix}
1 & 0 & 0 & 0 \\
0& 0 & \frac{\sqrt{2}}{2} &  -\frac{\sqrt{2}}{2} \\
0 & 1 & 0 & 0 \\
0 & 0 & -\frac{\sqrt{2}}{2} & -\frac{\sqrt{2}}{2}
\end{bsmallmatrix} P^{-1},
\]
where $P = \begin{bsmallmatrix}
0 & 1 & 0 & 0 \\
1 & 0 & 0 & 0 \\
0 & 0 & -\frac{3\sqrt{2}}{4} & -\frac{\sqrt{2}}{4} \\
0 & 0 & -\frac{\sqrt{2}}{4}  & -\frac{3\sqrt{2}}{4}
\end{bsmallmatrix} \in \O_{3,1}(\mathbb{R})$.
\end{example}

\begin{theorem}\label{thm:classification O_n2}
Let $\alpha$ be a quadratic  rational curve on $\O_{n,2}(\mathbb{R})$ with poles at $\pm \ci$. We denote 
\begin{align*}
Q_{1,3} &\coloneqq \begin{bsmallmatrix}
0 & 0  & 1  & 0 \\
-\frac{\sqrt{2}}{2} & 0 & 0 & -\frac{\sqrt{2}}{2} \\
0 & 1 & 0 & 0 \\
\frac{\sqrt{2}}{2} & 0  & 0 & -\frac{\sqrt{2}}{2}
\end{bsmallmatrix}, \quad 
Q_{3,3} \coloneqq   \begin{bsmallmatrix}
0 & 0 & 0  & - \frac{\sqrt{2}}{2}   & -\frac{\sqrt{2}}{2} & 0 \\
1 & 0 & 0  & 0   & 0 & 0 \\
0 & 0 & 0  & \frac{\sqrt{2}}{2}   & -\frac{\sqrt{2}}{2} & 0 \\
0 & 0 & - \frac{\sqrt{2}}{2} & 0   & 0 & -\frac{\sqrt{2}}{2}  \\
0 & 1 & 0  & 0  & 0 & 0 \\
0 & 0 & \frac{\sqrt{2}}{2}   & 0   & 0 & -\frac{\sqrt{2}}{2}
\end{bsmallmatrix}, \\
Q_4 &\coloneqq \begin{bsmallmatrix}
0 & -\frac{\sqrt{2}}{2}  & -\frac{\sqrt{2}}{2}  & 0 \\
\frac{\sqrt{2}}{2} & 0 & 0 & \frac{\sqrt{2}}{2} \\
0 & -\frac{\sqrt{2}}{2} & \frac{\sqrt{2}}{2} & 0 \\
\frac{\sqrt{2}}{2} & 0  & 0 & -\frac{\sqrt{2}}{2}
\end{bsmallmatrix},\quad 
P_{m,n}  \coloneqq \begin{bsmallmatrix}
I_m & & &  & \\
 & 0 & & 1  &\\
 & & I_{n-m-1} & &\\
 & 1 &  & 0  &\\
 &   &  &  &  & 1
\end{bsmallmatrix}.
\end{align*}
Then we can find some $P\in \O_{n,2}(\mathbb{R})$ such that $ \alpha$ is one of the following:
\begin{enumerate}[(i)]
\item\label{thm:classification O_n2:item1} There exist an integer $1 \le m \le n$ and a quadratic  rational curve $\beta_1$ (resp. $\beta_2$) on $\O_{m,1}(\mathbb{R})$ (resp. $\O_{n-m,1}(\mathbb{R})$) with poles at $\pm \ci$, such that
\[
\alpha(t) = 
P P_{m,n} 
\begin{bsmallmatrix}
\beta_1(t) & 0 \\
0 & \beta_2(t)
\end{bsmallmatrix} 
P_{m,n} P^{-1}.
\]
\item \label{thm:classification O_n2:item2} There exist a quadratic  rational curve $\beta_1$ (resp. $\beta_2$) on $\O_{n}(\mathbb{R})$ (resp. $\O_{2}(\mathbb{R})$) with poles at $\pm 1$, such that
\[
\alpha(t) = 
P   
\begin{bsmallmatrix}
\beta_1(t) & 0 \\
0 & \beta_2(t)
\end{bsmallmatrix}
P^{-1}.
\]

\item\label{thm:classification O_n2:item3} There exist a quadratic  rational curve $\beta(t)$ on $\O_{n-2}(\mathbb{R})$ with poles at $\pm \ci$ and numbers $\delta \in \{-1,1\}$, $\lambda \ge 0$, $a,b\in \mathbb{R}$ satisfying $a^2 + b^2 = \lambda^2$ such that 
\[
\alpha(t) = P 
\begin{bsmallmatrix}
\beta(t) & 0 & 0 & 0 & 0  \\
 0 & 1 & \frac{\lambda}{t^2+1}   & \frac{a}{t^2 + 1} & \frac{b}{t^2 + 1} \\
  0 & -\frac{\lambda}{t^2+1} & 1  & \frac{\delta b}{t^2 + 1} & \frac{-\delta a}{t^2 + 1} \\
 0 & \frac{a}{t^2 + 1} & \frac{\delta b}{t^2 + 1}    & 1 & \frac{\delta \lambda}{t^2 + 1} \\ 
 0 &  \frac{b}{t^2 + 1} & -\frac{\delta a}{t^2 + 1}  & -\frac{\delta \lambda}{t^2 + 1} & 1 \\
\end{bsmallmatrix} 
P^{-1}.
\]
\item\label{thm:classification O_n2:item4} There exist integers $m\ge 1, r \ge 0$, a column vector $w \in \mathbb{S}^{m-1}$ and a quadratic  rational curve $\beta(t)$ on $\O_{2r}(\mathbb{R})$ with poles at $\pm \ci$ satisfying 
$ m + 2r = n -2, \rank (\beta'(0)) = 2r $
such that
\[
\alpha(t) = P 
\begin{bsmallmatrix}
I_{n-2} & 0 \\
0 & Q_{1,3}^\tp
\end{bsmallmatrix}
\begin{bsmallmatrix}
\beta(t) & 0 & 0  & 0 & 0 & 0\\
0 &I_m & 0 & \frac{w}{t^2 + 1} & 0 & 0 \\
0 &0 & 1 & 0 & 0 & 0 \\
0 &0 & 0 & 1 & 0 & 0 \\
0 &0 & 0 & 0 & 1 & 0 \\
0 & \frac{w^\tp}{t^2 + 1} & 0 & \frac{2 t + 1}{2 (t^2 + 1) } & 0 & 1 
\end{bsmallmatrix}
\begin{bsmallmatrix}
I_{n-2} & 0 \\
0 & Q_{1,3}
\end{bsmallmatrix}
P^{-1}
\]
\item\label{thm:classification O_n2:item5} There exist integers $m,r\ge 0$, a quadratic  rational curve $\beta(t)$ on $\O_{2r}(\mathbb{R})$ with poles at $\pm \ci$, column vectors $x,y\in \mathbb{R}^m$ and numbers $z_1\in \mathbb{R}$, $z_2\in [-1,1]$ satisfying
\[
2r + m = n - 4,\quad  \rank(\beta'(0)) = 2r,\quad  x^\tp y = 0, \quad x^\tp x = y^\tp y = 1 - z_2^2,
\]
such that 
\[
\alpha(t) = P \begin{bsmallmatrix}
I_{n-4} & 0 \\ 
0 & Q_{3,3}^\tp
\end{bsmallmatrix}
\begin{bsmallmatrix}
\beta(t) & 0 & 0 &  0 & 0  & 0 & 0 & 0 \\
0 & I_m & \frac{x}{t^2 + 1} & 0 & 0 & \frac{y}{t^2 + 1} & 0 & 0 \\
0 & 0 & 1 & 0 & 0 & 0 & 0 & 0 \\
0 & 0 & 0& 1 & 0 & \frac{z_2}{t^2 + 1} & 0 & 0 \\
0 & \frac{x^\tp}{t^2 + 1}  & \frac{2t + 1}{2(t^2+1)} & 0  & 1 & \frac{z_1}{t^2 + 1} & -\frac{z_2}{t^2 + 1} & 0 \\
0 & 0 & 0 & 0 & 0  & 1 & 0 & 0 \\
0 & 0 & -\frac{z_2}{t^2 + 1} & 0 & 0  & 0 & 1 & 0 \\
0 &\frac{y^\tp}{t^2 + 1} & -\frac{z_1}{t^2 + 1} & \frac{z_2}{t^2 + 1} & 0  & \frac{2t + 1}{2(t^2+1)} & 0 & 1 \\
\end{bsmallmatrix} 
\begin{bsmallmatrix}
I_{n-4} & 0 \\ 
0 & Q_{3,3}
\end{bsmallmatrix} P^{-1}
\]

\item\label{thm:classification O_n2:item6} There exist a quadratic  rational curve $\beta(t)$ on $\O_{n-2}(\mathbb{R})$ with poles at $\pm \ci$ and numbers $b > 2, x,y\in \mathbb{R}$ satisfying $x^2 + y^2 = b^2(b^2/4 - 1)$
such that 
\[
\alpha(t) = P \begin{bsmallmatrix}
\beta(t) & 0 & 0 & 0 & 0 \\
0 & \frac{2t^2 + (2 - b^2)}{2(t^2+1)}  & \frac{bt}{t^2+1}  & \frac{x}{t^2 + 1} & \frac{y}{t^2 + 1} \\
0 & -\frac{bt}{t^2+1} & \frac{2t^2 + (2 - b^2)}{2(t^2+1)}  & \frac{y}{t^2 + 1} & -\frac{x}{t^2 + 1} \\
0 & \frac{x}{t^2 + 1} & \frac{y}{t^2 + 1} & \frac{2t^2 + (2 - b^2)}{2(t^2+1)}  & \frac{bt}{t^2+1} \\
0 & \frac{y}{t^2 + 1} & -\frac{x}{t^2 + 1} & -\frac{bt}{t^2+1} & \frac{2t^2 + (2 - b^2)}{2(t^2+1)}
\end{bsmallmatrix} P^{-1}.
\]
\item\label{thm:classification O_n2:item7}  There exist a quadratic  rational curve $\beta(t)$ on $\O_{n-2}(\mathbb{R})$ with poles at $\pm \ci$ and a real number $a\in \mathbb{R}$ such that 
\[
\alpha(t) = P \begin{bsmallmatrix}
I_{n-2} & 0 \\
0 & Q_4^\tp
\end{bsmallmatrix}
\begin{bsmallmatrix}
\beta(t) & 0 & 0 & 0 & 0 \\ 
0 & 1 &  0 & 0 & 0 \\
0 & \frac{t + a}{t^2+1} &  1 & 0 & 0 \\
0 & 0 &  0 & 1 & -\frac{t + a}{t^2+1} \\
0 & 0 &  0 & 0 & 1 \\
\end{bsmallmatrix}
\begin{bsmallmatrix}
I_{n-2} & 0 \\
0 & Q_4
\end{bsmallmatrix} P^{-1}.
\]
\end{enumerate}
\end{theorem}
\begin{proof}
We postpone the proof to Appendix~\ref{append:proof of O_n2}.
\end{proof}
\begin{remark}
Curves of type  \eqref{thm:classification O_n2:item1} (resp. \eqref{thm:classification O_n2:item2}) can be constructed from those on $\O_{m,1}(\mathbb{R}) \times \O_{n-m,1}(\mathbb{R}) \subseteq \O_{n,2}(\mathbb{R})$ (resp.  $\O_{n}(\mathbb{R}) \subseteq \O_{n,2}(\mathbb{R})$)
\end{remark}
\begin{example}\label{ex:O_n2}
Assume that poles of $\alpha \in \Rat_2(\O_{2,2}(\mathbb{R}), I_2)$ are $\pm \ci$.  Theorem~\ref{thm:classification O_n2}  implies that $\alpha$ is of type \eqref{thm:classification O_n2:item1},  \eqref{thm:classification O_n2:item2},   \eqref{thm:classification O_n2:item3},  \eqref{thm:classification O_n2:item6} or \eqref{thm:classification O_n2:item7}.  Thus,  there is some $P\in \O_{2,2}(\mathbb{R})$ such that $\alpha$ has one of the following five forms:
\begin{align*}
\alpha(t) &= P 
\begin{bsmallmatrix}
\frac{t^2 - 1}{t^2 + 1} & \frac{2t}{t^2 + 1} & 0 & 0 \\
-\frac{2t}{t^2 + 1} & \frac{t^2 - 1}{t^2 + 1}  & 0 & 0 \\
0 & 0 &  1 & 0  \\
0 & 0 & 0 & 1
\end{bsmallmatrix}  P^{-1},\quad \alpha(t) = P
\begin{bsmallmatrix}
\frac{t^2 - 1}{t^2 + 1} & \frac{2t}{t^2 + 1} & 0 & 0 \\
-\frac{2t}{t^2 + 1} & \frac{t^2 - 1}{t^2 + 1}  & 0 & 0 \\
0 & 0 &  \frac{t^2 - 1}{t^2 + 1} & \frac{2t}{t^2 + 1}  \\
0 & 0 & -\frac{2t}{t^2 + 1} & \frac{t^2 - 1}{t^2 + 1} 
\end{bsmallmatrix}  P^{-1},  \\
\alpha(t) &= P 
\begin{bsmallmatrix}
1 & \frac{\lambda}{t^2+1}   & \frac{a}{t^2 + 1} & \frac{b}{t^2 + 1} \\
 -\frac{\lambda}{t^2+1} & 1  & \frac{\delta b}{t^2 + 1} & \frac{-\delta a}{t^2 + 1} \\
\frac{a}{t^2 + 1} & \frac{\delta b}{t^2 + 1}    & 1 & \frac{\delta \lambda}{t^2 + 1} \\ 
 \frac{b}{t^2 + 1} & -\frac{\delta a}{t^2 + 1}  & -\frac{\delta \lambda}{t^2 + 1} & 1
\end{bsmallmatrix} 
P^{-1},\quad \alpha(t) = P \begin{bsmallmatrix}
\frac{2t^2 + (2 - c^2)}{2(t^2+1)}  & \frac{ct}{t^2+1}  & \frac{x}{t^2 + 1} & \frac{y}{t^2 + 1} \\
-\frac{ct}{t^2+1} & \frac{2t^2 + (2 - c^2)}{2(t^2+1)}  & \frac{y}{t^2 + 1} & -\frac{x}{t^2 + 1} \\
\frac{x}{t^2 + 1} & \frac{y}{t^2 + 1} & \frac{2t^2 + (2 - c^2)}{2(t^2+1)}  & \frac{ct}{t^2+1} \\
 \frac{y}{t^2 + 1} & -\frac{x}{t^2 + 1} & -\frac{ct}{t^2+1} & \frac{2t^2 + (2 - c^2)}{2(t^2+1)}
\end{bsmallmatrix} P^{-1}, \\
\alpha(t) &= P \begin{bsmallmatrix}
I_{n-2} & 0 \\
0 & Q_4^\tp
\end{bsmallmatrix}
\begin{bsmallmatrix}
1 &  0 & 0 & 0 \\
\frac{t + d}{t^2+1} &  1 & 0 & 0 \\
 0 &  0 & 1 & -\frac{t + d}{t^2+1} \\
0 &  0 & 0 & 1 \\
\end{bsmallmatrix}
\begin{bsmallmatrix}
I_{n-2} & 0 \\
0 & Q_4
\end{bsmallmatrix} P^{-1}.
\end{align*}
Here $\delta  = \pm 1$,  $\lambda \ge 0$,  $(a,b) \in \mathbb{R}^2$,  $c > 2$,  $(x,y)\in \mathbb{R}^2$ and $d\in \mathbb{R}$ are constant numbers such that $a^2 + b^2 = \lambda^2$ and $x^2 + y^2 = c^2 (c^2/4 -1)$.

Rational curves of type \eqref{thm:classification O_n2:item4} (resp.  \eqref{thm:classification O_n2:item5}) appear only if $n \ge 3$ (resp.  $n \ge 4$).  As an example,  if $\alpha \in \Rat_2(\O_{4,2}(\mathbb{R}), I_6)$ with poles at $\pm \ci$ is of type   \eqref{thm:classification O_n2:item4} or \eqref{thm:classification O_n2:item5},  then there is some $P\in \O_{4,2}(\mathbb{R})$ such that 
\[
\alpha(t) = P 
\begin{bsmallmatrix}
I_{2} & 0 \\
0 & Q_{1,3}^\tp
\end{bsmallmatrix}
\begin{bsmallmatrix}
I_2 & 0 & \frac{w}{t^2 + 1} & 0 & 0 \\
0 & 1 & 0 & 0 & 0 \\
0 & 0 & 1 & 0 & 0 \\
0 & 0 & 0 & 1 & 0 \\
\frac{w^\tp}{t^2 + 1} & 0 & \frac{2 t + 1}{2 (t^2 + 1) } & 0 & 1 
\end{bsmallmatrix}
\begin{bsmallmatrix}
I_{2} & 0 \\
0 & Q_{1,3}
\end{bsmallmatrix}
P^{-1},\quad
\alpha(t) = P 
Q_{3,3}^\tp
\begin{bsmallmatrix}
1 & 0 & 0 & 0 & 0 & 0 \\
0& 1 & 0 & \frac{z_2}{t^2 + 1} & 0 & 0 \\
 \frac{2t + 1}{2(t^2+1)} & 0  & 1 & \frac{z_1}{t^2 + 1} & -\frac{z_2}{t^2 + 1} & 0 \\
0 & 0 & 0  & 1 & 0 & 0 \\
-\frac{z_2}{t^2 + 1} & 0 & 0  & 0 & 1 & 0 \\
\frac{z_2}{t^2 + 1} & 0  & \frac{2t + 1}{2(t^2+1)} & 0 & 1 \\
\end{bsmallmatrix} 
Q_{3,3}
 P^{-1}
\]
where $w \in \mathbb{S}^1,  z_1 \in \mathbb{R}$ and $z_2 = \pm 1$.
\end{example}
\section{Decomposition of rational curves on linear algebraic groups}\label{sec:dec}
This section is devoted to the decomposition of rational curves on $G_B(\mathbb{F})$ and $\ISO_{p,n-p}^+(\mathbb{R})$.  We first deal with $G_B(\mathbb{F})$ as $\SO_{p,n-p}^+(\mathbb{R})$ is a special case of $G_B(\mathbb{F})$,  and the proof for $\ISO_{p,n-p}^+(\mathbb{R})$ relies on the decomposition theorem for $\SO_{p,n-p}^+(\mathbb{R})$.
\subsection{Decomposition of rational curves on \texorpdfstring{$G_B(\mathbb{F})$}{GB}}\label{subsec:KempeGB}
\begin{lemma}\label{lem:relation}
Let $\gamma(t) = P(t)/q(t)$ be a  rational curve  on $G_B(\mathbb{F})$ and let $\zeta \in \mathbb{C}\setminus \mathbb{R}$ be a root of $q(t)$ with multiplicity $s\ge 1$. Then we have  
\[
\sum_{j=0}^{l} P^{(j)}(\zeta) B P^{(l - j)}(\zeta)^\sigma = 0,\quad l = 0,\dots, 2s-1.
\]
\end{lemma}
\begin{proof}
By definition, $P(t)$ and $q(t)$ satisfy the relation:
\begin{equation}\label{lem:relation:eq1}
P(t) B P(t)^\sigma  = q(t)^2 B.
\end{equation}
Since $\zeta$ is a root of $q(t)$ with multiplicity $s$, 
the desired relations for $P^{(j)}(\zeta)$'s are obtained immediately by differentiating \eqref{lem:relation:eq1} at $\zeta$. 
\end{proof}
\begin{lemma}\label{lem:rank}
Assume that $P_0, P_1,\dots, P_{2s-1} \in \mathbb{C}^{n\times n}$ satisfy 
\begin{equation}\label{lem:rank:eq1}
\sum_{j=0}^{l} P_j B P_{l - j}^\sigma = 0,\quad l = 0,\dots, 2s-1.
\end{equation}
Then we have 
\[
\rank \left( \begin{bmatrix}
P_0 & 0 & \cdots & 0  & 0 \\
P_1 & P_0 & \cdots & 0 & 0\\
\vdots &  \vdots & \ddots & \vdots & \vdots \\
P_{2s-2} & P_{2s-3} & \cdots & P_0 & 0\\
P_{2s-1} & P_{2s-2} & \cdots & P_1 & P_0
\end{bmatrix} \right) \le sn.
\]
\end{lemma}
\begin{proof}
We denote
\[
M_1 \coloneqq  \begin{bmatrix}
P_{2s-1} & P_{2s-2} & \cdots & P_{1} & P_{0} \\
P_{2s-2} & P_{2s-3} & \cdots & P_{0} & 0 \\
\vdots &  \vdots & \ddots & \vdots & \vdots \\
P_{1} & P_{0} & \cdots & 0 & 0 \\
P_{0} & 0 & \cdots & 0 & 0 \\
\end{bmatrix},\quad 
M_2 \coloneqq J \begin{bmatrix}
P_0^\sigma & 0 & \cdots & 0  & 0 \\
P_1^\sigma & P_0^\sigma & \cdots & 0 & 0\\
\vdots &  \vdots & \ddots & \vdots & \vdots \\
P_{2s-2}^\sigma & P_{2s-3}^\sigma & \cdots & P_0^\sigma & 0\\
P_{2s-1}^\sigma & P_{2s-2}^\sigma & \cdots & P_1^\sigma & P_0^\sigma
\end{bmatrix},
\]
where
\[
J = \diag(\underbrace{B,\dots, B}_{2s~\text{copies}}) \in \mathbb{C}^{2sn \times 2sn}.
\] 
According to Lemma~\ref{lem:involution}, we have $\rank(M_1) = \rank(M_2) \eqqcolon r$. Moreover, \eqref{lem:rank:eq1} implies that $M_1 M_2 = 0$, from which we derive 
\[
2sn - r = \dim (\ker (M_1)) \ge \rank (M_2) = r.
\]
This implies $r \le sn$.
\end{proof}
\begin{lemma}[Degree reduction for $G_B(\mathbb{R})$]\label{lem:existenceR}
For any $\gamma(t) = P(t)/q(t)\in \Rat_d(G_B(\mathbb{R}),  I_n)$ with a pole $\zeta \in \mathbb{C}\setminus \mathbb{R}$ of multiplicity $s$, there exists an $\alpha(t)\in \Rat_{2s}(G_B(\mathbb{R},I_n)$ with only poles at $\zeta$ and $\overline{\zeta}$ such that $\alpha(t) \gamma(t) \in \Rat_{d- 2s}(G_B(\mathbb{R},I_n) $.
\end{lemma}
\begin{proof}
By a linear change of coordinate, we may assume that $\zeta = \ci$. We write $A(t) = c  I_n t^{2s} +\sum_{j=0}^{2s - 1} A_j t^j$ where $c\in \mathbb{R}, A_{2s - 1},\dots, A_0\in \mathbb{R}^{n\times n}$ are coefficients to be determined. Then we have 
\[
A^{(m)}( \ci ) =\frac{(2s)!}{(2s - m)!} c  I_n \ci^{2s-m}  + \sum_{j=m}^{2s - 1} \frac{j!}{(j - m)!} A_j \ci^{j-m},\quad 0 \le m \le 2s-1.
\]
We consider the homogeneous system of linear equations:
\begin{equation}\label{lem:existence:eq1}
(AP)^{(l)}( \ci ) = \sum_{j = 0}^l A^{(l-j)}( \ci ) P_{j} = 0,\quad l = 0,\dots, 2s-1
\end{equation}
where $P_j = P^{(j)}( \ci ), 0 \le j \le 2s-1$. 

If \eqref{lem:existence:eq1} has a solution of the form $(1,A_{2s-1},\dots, A_0)\in \mathbb{R}\times (\mathbb{R}^{n\times n})^{2s}$, then $\alpha(t) \coloneqq A(t)/( t^2 + 1)^s$ is a desired  rational curve of degree $2s$. Indeed, by \eqref{lem:existence:eq1} we clearly have $(t - \ci )^{2s} | A(t) P(t)$. Since both $A(t)$ and $P(t)$ are real, we further have  $(t^2 + 1)^{2s} | A(t) P(t)$. We notice that 
\[
(t^2 + 1)^{4s} | A(t) P(t) B  (A(t) P(t))^\sigma  = q(t)^2 A(t) B A(t)^\sigma.
\]
Therefore, $(t^2 + 1)^{2s} | A(t) B A(t)^\sigma$ since $ \ci $ is a root of $q(t)$ of multiplicity $s$. Since $A(t) = I_n t^{2s} + \sum_{j=0}^{2s - 1} A_j t^j$, we have $A(t) B A(t)^\sigma  = B t^{4s} + O(t^{4s-1}) \ne 0$. This implies $A(t) B A(t)^\sigma = (t^2 + 1)^{2s}B$ and $\alpha(t)$ is a  rational curve of degree $2s$ on $G_B(\mathbb{R})$.

Thus, it is left to prove that \eqref{lem:existence:eq1} has a solution $(c,A_{2s-1},\dots,A_0)\in \mathbb{R}\times (\mathbb{R}^{n\times n})^s$ such that $c \ne 0$. To this end, we obverse that 
\begin{equation}\label{lem:existence:eq11}
A^{(m)}( \ci ) =
\begin{bmatrix}
cI_n & A_{2s-1} & \cdots & A_m & A_{m-1} & \cdots & A_0
\end{bmatrix}
\begin{bmatrix}
\frac{(2s)! \ci^{2s-m}}{(2s - m)!} I_n \\ 
\frac{(2s-1)! \ci^{2s-1-m} }{(2s - 1 - m)!} I_n \\ 
\vdots \\
\frac{(m)! \ci^0}{0!} I_n \\
0 \\ 
\vdots \\
0 
\end{bmatrix},\quad 0 \le m \le 2s -1.
\end{equation}
Thus we may rewrite \eqref{lem:existence:eq1} as  
\begin{equation}\label{lem:existence:eq2}
\begin{bmatrix}
cI_n & A_{2s-1} & \cdots & A_0
\end{bmatrix} C  M
= 0,
\end{equation}
where 
\begin{align*}
M &\coloneqq \begin{bmatrix}
P_0 & 0 & \cdots & 0  & 0 \\
P_1 & P_0 & \cdots & 0 & 0\\
\vdots &  \vdots & \ddots & \vdots & \vdots \\
P_{2s-2} & P_{2s-3} & \cdots & P_0 & 0\\
P_{2s-1} & P_{2s-2} & \cdots & P_1 & P_0
\end{bmatrix} \in \mathbb{C}^{2sn \times 2sn}, \\
C &\coloneqq \begin{bmatrix}
\frac{(2s)! \ci^{2s}}{(2s)!} I_n & \frac{(2s)! \ci^{2s-1}}{(2s-1)!} I_n &  \cdots & \frac{(2s)! \ci^{2}}{(2)!} I_n & \frac{(2s)! \ci^{1}}{(1)!} I_n  \\ 
\frac{(2s-1)! \ci^{2s-1} }{(2s - 1)!} I_n & \frac{(2s-1)! \ci^{2s-2} }{(2s - 2)!} I_n & \cdots &   \frac{(2s-1)! \ci}{(1)!} I_n & \frac{(2s-1)! \ci^{0}}{(0)!} I_n \\ 
\frac{(2s-2)! \ci^{2s-2} }{(2s - 2)!} I_n & \frac{(2s-2)! \ci^{2s-3} }{(2s - 3)!} I_n & \cdots &   \frac{(2s-2)! \ci^{0}}{(0)!} I_n & 0\\ 
\vdots & \vdots & \ddots & \vdots & \vdots \\
\frac{(1)! \ci^1}{1!} I_n & \frac{(1)! \ci^0}{0!} I_n & \cdots & 0 & 0  \\
\frac{(0)! \ci^0}{0!} I_n  & 0 & \cdots & 0 & 0  
\end{bmatrix} \in \mathbb{C}^{(2s+1)n \times 2sn}.
\end{align*}
Since $\gamma(t)$ is a  rational curve on $G_B(\mathbb{R})$, $P_0,\dots, P_{2s-1}$ satisfy relations in \eqref{lem:rank:eq1}. Lemma~\ref{lem:rank} implies that $\rank (C M) \le \rank(M) \le sn$. The homogeneous linear system \eqref{lem:existence:eq2} imposes at most $2 sn^2$ real constraints on $2sn^2 + 1$ real variables $(c,A_{2s-1},\dots, A_0)$. Therefore, \eqref{lem:existence:eq2} has a real solution. If all real solutions of \eqref{lem:existence:eq2} are contained in the hyperplane $c = 0$, then we must have 
\[
\rank \left( C M \right) = \rank \left( D \right),\quad D= \begin{bmatrix}
v_{(2s+1)n}^\tp & CM
\end{bmatrix},
\]
where $v_{(2s+1)n} = (1,0,\dots, 0)\in \mathbb{R}^{(2s+1)n}$. We observe that by row operations, $C$ can be transformed into 
\[
\begin{bmatrix}
0 & 0 &  \cdots & 0 & 0  \\ 
0 & 0& \cdots &   0 &  I_n \\ 
0 & 0 & \cdots &   I_n & 0\\ 
\vdots & \vdots & \ddots & \vdots & \vdots \\
0 & I_n & \cdots & 0 & 0  \\
I_n  & 0 & \cdots & 0 & 0  
\end{bmatrix} \in \mathbb{R}^{(2s+1)n \times 2sn},
\]
thus $CM$ and $D$ can be transformed by the same row operations into 
\[
\begin{bmatrix}
0 & 0 & \cdots & 0  & 0 \\
P_{2s-1} & P_{2s-2} & \cdots & P_1 & P_0 \\
P_{2s-2} & P_{2s-3} & \cdots & P_0 & 0\\
\vdots &  \vdots & \ddots & \vdots & \vdots \\
P_1 & P_0 & \cdots & 0 & 0\\
0 & 0 & \cdots & 0  & 0 
\end{bmatrix},\quad 
\begin{bmatrix}
v_n^\tp & 0 & 0 & \cdots & 0  & 0 \\
0 & P_{2s-1} & P_{2s-2} & \cdots & P_1 & P_0 \\
0 & P_{2s-2} & P_{2s-3} & \cdots & P_0 & 0\\
\vdots & \vdots &  \vdots & \ddots & \vdots & \vdots \\
0 & P_1 & P_0 & \cdots & 0 & 0\\
0 & 0 & 0 & \cdots & 0  & 0 
\end{bmatrix},
\]
where $v_{n} = (1,0,\dots, 0)\in \mathbb{R}^{n}$. This clearly contradicts the equality $\rank \left( C M \right) = \rank \left( D \right)$, so \eqref{lem:existence:eq2} must have a real solution $(c,A_{2s-1},\dots, A_0)$ such that $c \ne 0$.
\end{proof}
Next we consider the analogue of Lemma~\ref{lem:existenceR} for $G_B(\mathbb{C})$ and $G_B(\mathbb{H})$. 
\begin{lemma}[Degree reduction for $G_B(\mathbb{C})$ and $G_B(\mathbb{H})$]\label{lem:existenceC}
For any $\gamma(t) = P(t)/q(t) \in \Rat_d(G_B(\mathbb{C}), I_n)$ (resp.  $\gamma(t) = P(t)/q(t) \in \Rat_d(G_B(\mathbb{H}),  I_n)$) with a pole $\zeta \in \mathbb{C}\setminus \mathbb{R}$ of multiplicity $s$,  there exists an $\alpha(t) \in \Rat_{2s}(G_B(\mathbb{C}), I_n)$ (resp.  $\alpha(t) \in \Rat_{2s}(G_B(\mathbb{H}), I_n)$) with only poles at $\zeta$ and $\overline{\zeta}$ such that $\alpha(t) \gamma(t) \in \Rat_{d - 2s}(G_B(\mathbb{C}), I_n)$ (resp.  $\alpha(t) \in \Rat_{d-2s}(G_B(\mathbb{H}), I_n)$). 
\end{lemma}
\begin{proof}
We first deal with the case over $\mathbb{C}$. We recall that $\mathbb{C}^{n\times n}$ is embedded in $\mathbb{R}^{2n \times 2n}$ as an $\mathbb{R}$-subalgebra by 
\[
\psi: \mathbb{C}^{n\times n} \hookrightarrow \mathbb{R}^{2n\times 2n},\quad \psi(A + \ci B) = \begin{bmatrix}
A & B \\
-B & A
\end{bmatrix}.
\]

By a linear change of coordinate, we may assume that $\zeta = \ci$. We write $C(t) = c  I_n t^{2s} +\sum_{j=0}^{2s - 1} (A_j + \ci B_j) t^j$ where $c\in \mathbb{R}, A_{2s - 1},B_{2s-1}, \dots, A_0, B_0\in \mathbb{R}^{n\times n}$ are coefficients to be determined. We write $Z_j = \psi(A_j + \ci B_j), 0 \le j \le 2s-1$. Then we have $Z(t) = c  I_{2n} t^{2s} +\sum_{j=0}^{2s - 1} Z_j t^j$ and
\[
Z^{(m)}(\ci) =\frac{(2s)!}{(2s - m)!} c  I_{2n} \ci^{2s-m}  + \sum_{j=m}^{2s - 1} \frac{j!}{(j - m)!} Z_j \ci^{j-m}\in \mathbb{C}^{2n \times 2n},\quad 0 \le m \le 2s-1.
\]
We consider the homogeneous system of linear equations:
\begin{equation}\label{lem:existenceC:eq1}
(Z \psi(P))^{(l)}(\ci) = \sum_{j = 0}^l Z^{(l-j)}(\ci) P_j^\psi = 0,\quad l = 0,\dots, 2s-1
\end{equation}
where $P_j^\psi \coloneqq \psi(P)^{(j)}(\ci)\in \mathbb{C}^{2n \times 2n}, 0 \le j \le 2s-1$. 

If \eqref{lem:existenceC:eq1} has a solution of the form $(1,A_{2s-1},B_{2s-1}, \dots, A_0,B_0)\in \mathbb{R}\times (\mathbb{R}^{n\times n})^{4s}$, then $\alpha(t) \coloneqq C(t)/( t^2 + 1)^s$ is a desired  rational curve of degree $2s$. Indeed, by \eqref{lem:existenceC:eq1} we clearly have $(t - \ci)^{2s} | Z(t) \psi(P(t))$. Since both $Z(t)$ and $\psi(P(t))$ are real, we further have  $(t^2 + 1)^{2s} | Z(t) \psi(P(t))$. We notice that 
\[
(t^2 + 1)^{4s} | \left( Z \psi (P) \right) \psi(B) \left( Z \psi (P) \right)^\sigma
=
\psi ( (AP) B (AP)^\sigma)  = q^2 \psi(A B A^\sigma).
\]
Therefore, $(t^2 + 1)^{2s} | \psi(A B A^\sigma)$ as $\ci$ is a root of $q(t)$ of multiplicity $s$. By the definition of $\psi$, we derive $(t^2 + 1)^{2s} | A B A^\sigma$. The rest of the argument is the same as the one in the proof of Lemma~\ref{lem:existenceR}.

Thus, it suffices to prove that \eqref{lem:existenceC:eq1} has a solution $(c,A_{2s-1},B_{2s-1},\dots,A_0,B_0)\in \mathbb{R}\times (\mathbb{R}^{n\times n})^{4s}$ such that $c \ne 0$. To this end, 
we re-write \eqref{lem:existenceC:eq1} by   \eqref{lem:existence:eq11} as  
\begin{equation}\label{lem:existenceC:eq2}
\begin{bsmallmatrix}
cI_{2n} & Z_{2s-1} & \cdots & Z_0
\end{bsmallmatrix} C  M
= 0,
\end{equation}
where 
\begin{align*}
M &\coloneqq \begin{bsmallmatrix}
P_0^\psi & 0 & \cdots & 0  & 0 \\
P_1^\psi & P_0^\psi & \cdots & 0 & 0\\
\vdots &  \vdots & \ddots & \vdots & \vdots \\
P_{2s-2}^\psi & P_{2s-3}^\psi & \cdots & P_0^\psi & 0\\
P_{2s-1}^\psi & P_{2s-2}^\psi & \cdots & P_1^\psi & P_0^\psi
\end{bsmallmatrix} \in \mathbb{C}^{4sn \times 4sn}, \\
C &\coloneqq \begin{bsmallmatrix}
\frac{(2s)! \ci^{2s}}{(2s)!} I_{2n} & \frac{(2s)! \ci^{2s-1}}{(2s-1)!} I_{2n} &  \cdots & \frac{(2s)! \ci^{2}}{(2)!} I_{2n} & \frac{(2s)! \ci^{1}}{(1)!} I_{2n}  \\ 
\frac{(2s-1)! \ci^{2s-1} }{(2s - 1)!} I_{2n} & \frac{(2s-1)! \ci^{2s-2} }{(2s - 2)!} I_{2n} & \cdots &   \frac{(2s-1)! \ci^{1}}{(1)!} I_{2n} & \frac{(2s-1)! \ci^{0}}{(0)!} I_{2n} \\ 
\frac{(2s-2)! \ci^{2s-2} }{(2s - 2)!} I_{2n} & \frac{(2s-2)! \ci^{2s-3} }{(2s - 3)!} I_{2n} & \cdots &   \frac{(2s-2)! \ci^{0}}{(0)!} I_{2n} & 0\\ 
\vdots & \vdots & \ddots & \vdots & \vdots \\
\frac{(1)! \ci^1}{1!} I_{2n} & \frac{(1)! \ci^0}{0!} I_{2n} & \cdots & 0 & 0  \\
\frac{(0)! \ci^0}{0!} I_{2n}  & 0 & \cdots & 0 & 0  
\end{bsmallmatrix} \in \mathbb{C}^{(4s+2)n \times 4sn}.
\end{align*}
Since $\gamma(t)$ is a  rational curve on $G_B(\mathbb{C})$, $P_0,\dots, P_{2s-1}$ satisfy relations in \eqref{lem:rank:eq1}, where $P_j \coloneqq P^{(j)}(\ci)$, $0 \le j \le 2s-1$. Lemma~\ref{lem:rank} implies that $\rank (C M) \le \rank(M) \le 2sn$. Here, the last inequality follows from the observation that $\rank (\psi(Z)) = 2 \rank (Z)$ for any $Z\in \mathbb{C}^{n\times n}$. The homogeneous linear system \eqref{lem:existenceC:eq2} imposes at most $4 sn^2$ real constraints on $4sn^2 + 1$ real variables $(c,A_{2s-1},B_{2s-1}, \dots, A_0,B_0)$. Therefore, \eqref{lem:existenceC:eq2} has a real solution. The existence of a solution $(c,A_{2s-1},B_{2s-1},\dots,A_0,B_0)\in \mathbb{R}\times (\mathbb{R}^{n\times n})^{4s}$ where $c \ne 0$ follows by the argument in the proof of Lemma~\ref{lem:existenceR}. 

For the case over $\mathbb{H}$. We may embed $\mathbb{H}^{n\times n}$ into $\mathbb{R}^{4n \times 4n}$ as an $\mathbb{R}$-subalgebra by 
\[
\varphi: \mathbb{H}^{n\times n} \hookrightarrow \mathbb{R}^{4n\times 4n},\quad \varphi(A + \qi B + \qj C + \qk D) = \begin{bsmallmatrix}
A & B  & C & D\\
-B & A & -D & C \\
-C & D  & A & -B \\
-D & -C & B & A
\end{bsmallmatrix}.
\] 
The rest of the proof is the same as the one for the case over $\mathbb{C}$.
\end{proof}
\begin{theorem}[Decomposition of rational curves on $G_B(\mathbb{F})$]\label{thm:KempeGBF}
If $\gamma(t) \in \Rat (G_B(\mathbb{F}),I_n)$ has poles of multiplicities $s_1,\dots, s_l$, then $\gamma (t) = \beta_1(t) \cdots \beta_l(t)$ for some $\beta_j(t) \in \Rat_{2s_j} (G_B(\mathbb{F}),I_n)$,  $1 \le j \le l$.  In particular, if all the poles of $\gamma(t)$ are simple, then $\gamma(t)$ can be decomposed into a product of $d$ quadratic  rational curves.
\end{theorem}
\begin{proof}
By Lemmas~\ref{lem:existenceR} and \ref{lem:existenceC}, there exist  rational curves $\alpha_1(t),\dots, \alpha_l(t)$ of degrees $2s_1,\dots, 2s_l$ respectively such that $\alpha_l(t) \cdots \alpha_1(t)  \gamma(t) = 1$. For each $1 \le j \le l$, we let $\beta_j(t) = \alpha_j(t)^{-1}$. Proposition~\ref{prop:inverse} indicates that $\beta_j(t)$ is a  rational curve on $G_B$ of degree $2s_j$ and this completes the proof.
\end{proof}
\begin{remark}
One can easily construct a rational curve on $G_B(\mathbb{F})$ with multiple poles,  which can be further decomposed into a product of low degree rational curves.  However, Example~\ref{ex:KempeGBF} indicates the existence of quartic rational curves with multiple poles,  which can not be decomposed into a product of two quadratic rational curves.
\end{remark}
\begin{example}\label{ex:KempeGBF}
We first consider
\[
\gamma(t) = I_3 +\frac{1}{t^2+1} W_1 +\frac{1}{2(t^2+1)^2} W_2,   \quad W_1 \coloneqq
\begin{bsmallmatrix}
 0 & 1  &   -1  \\
 -1 &  0  &   0 \\
 -1 &  0  &   0  
\end{bsmallmatrix}, 
\quad W_2\coloneqq 
\begin{bsmallmatrix}
 0 & 0  &   0  \\
 0 &  -1  &   1 \\
 0 &  -1  &  1  
\end{bsmallmatrix}.  
\]
It is straightforward to verify that $\gamma \in \Rat_4(\SO^+_{2,1}(\mathbb{R}), I_3)$.  We prove that $\gamma$ is not a product of two quadratic rational curves on $\SO^+_{2,1}(\mathbb{R},I_3)$.  Assume on the contrary that $\gamma(t) = \alpha(t) \beta(t)$ for some $\alpha,  \beta \in  \Rat_2(\SO^+_{2,1}(\mathbb{R}), I_3)$ which are parametrized as 
\[
 \alpha(t)=\frac{t^2 I_3 +t A_1+A_0}{t^2+1}, \quad  \beta(t)=\frac{t^2 I_3 + tB_1 + B_0}{t^2+1}.
\]
Since $\alpha(t)^\tp I_{2,1} \gamma(t) =I_{2,1} \beta(t)$,  the numerator $N(t)$ of $\alpha(t)^\tp I_{2,1} \gamma(t)$ must be divisible by $(t^2 + 1)^2$,  where 
\[
N(t) = \left(t^2 I_3 + tA_1^{\tp}+A_0^{\tp}\right)I_{2,1}\left(2(t^2+1)^2 I_3 + 2(t^2+1) W_1 + W_2\right).
\]
The remainder of $N(t)$ divided by $(t^2 + 1)^2$ is 
\[
\left(tA_1^{\tp}+A_0^{\tp}- I_3 \right)I_{2,1}(2(t^2+1) W_1+W_2)+(t^2+1) I_{2,1} W_2 = 0. 
\]
Since $\alpha \in \Rat_{2}(\SO_{2,1}(\mathbb{R}), I_2)$,  we obtain the equations for $A_1$:
\[
A_1 I_{2,1}+I_{2,1} A_1^{\tp}=A_1^{\tp} I_{2,1} W_1=A_1^{\tp} I_{2,1} W_2=0.
\]
This implies $A_1 = 0$ and we have $A_0= I_3$ by $\gamma_A(t)^{T} I_{2,1} \gamma_A(t)= I_{2,1}$.  This leads to a contradictory equality $0 = N(t) = (t^2 + 1)I_{2,1} W_2$.

Next we consider the rational curve on $\O_4(\mathbb{C})$ defined by 
\[
\gamma(t) =  
I_4 + \frac{t}{(t^2+1)^2}  U,\quad 
U \coloneqq \begin{bsmallmatrix}
 0 & -\ci  &   1  & 0 \\
 \ci &  0  &   0  & 1 \\
 -1 &  0  &   0  & \ci \\
 0 & -1  &  -\ci  & 0
\end{bsmallmatrix}.
\]
It is straightforward to verify $U +U^\tp=0$ and $U^2=0$.  We claim that $\gamma(t)\neq \alpha(t)\beta(t)$,  where $\alpha,\beta \in \Rat_2(\O_4(\mathbb{C}), I_4)$.  Otherwise, we write 
\[
\alpha(t)=\frac{t^2 I_4 + tA_1+A_0}{t^2+1}, \quad  \beta(t)= \frac{t^2 I_4 +tB_1+B_0}{t^2+1},          
\]
where $A_0,A_1,B_0,B_1$ are matrices such that $\alpha(t)^{\tp}\alpha(t)=\beta(t)^{\tp}\beta(t)= I_4 $. We obtain 
\[
\frac{t^2 I_4 +tB_1+B_0}{t^2+1} = \beta(t) = \alpha(t)^\tp \gamma(t) = \frac{t^2 I_4 + tA_1^\tp+A_0^\tp}{t^2+1}
\frac{(t^2+1)^2 I_4 + t U}{ (t^2+1)^2 },
\]
from which we conclude that $(t^2 + 1)^2$ divides $( t^2 I_4 + tA_1^\tp+A_0^\tp ) \left( 
(t^2+1)^2 I_4 + t U \right)$. However, this leads to a contradiction that $(t^2 + 1)^2$ must divide $t U (t^2 I_4 + t A_1^\tp + A_0^\tp)$.
\end{example}
\subsection{Decomposition of rational curves on \texorpdfstring{$\ISO^+_{p,n-p}(\mathbb{R})$}{IGB}}\label{subsec:KempeEn}
We recall that
\[
\ISO^+_{p,n-p}(\mathbb{R}) \coloneqq \left\lbrace
\begin{bsmallmatrix}
Q & u \\
0 & 1
\end{bsmallmatrix} \in \GL_{n+1}(\mathbb{R}): Q\in \SO^+_{p,n-p}(\mathbb{R}),  u\in \mathbb{R}^n
\right\rbrace.
\]
A  rational curve $\gamma(t)$ on $\ISO^+_{p,n-p}(\mathbb{R})$ can be uniquely written as 
\begin{equation}\label{eq:curveSE}
\gamma(t) = 
\begin{bsmallmatrix}
\frac{Q(t)}{q_1(t)} & \frac{u(t)}{q_2(t)} \\
0 & 1
\end{bsmallmatrix},
\end{equation}
where $q_1(t)$ (resp. $q_2(t)$, $Q(t) = (Q_{ij}(t))_{i,j=1}^n$ and $u(t) = (u_i(t))_{i=1}^n$) is a real polynomial (resp. polynomial, $\mathbb{R}^{n\times n}$-valued polynomial and $\mathbb{R}^n$-valued polynomial) such that 
\begin{itemize}
\item $q_1(t), q_2(t)$ are monic with no real roots;
\item $\gcd(q_1(t),Q_{11}(t),\dots, Q_{nn}(t))  = \gcd(q_2(t), u_1(t),\dots, u_n(t)) = 1$;
\item $Q(t)^\tp I_{p,n-p} Q(t) = q_1(t)^2 I_{p,n-p}$;
\item $\lim_{t \to \infty} Q(t)/q_1(t) = I_n$;
\item $\lim_{t \to \infty} u(t)/q_2(t) = 0$.
\end{itemize}

\begin{lemma}\label{lem:existenceSE}
Let $\gamma(t) \in \Rat(\ISO_{p,n-p}^+(\mathbb{R}), I_n)$ be parametrized as in \eqref{eq:curveSE}.  Suppose that $q_1$ has $l$ roots of multiplicities $s_1,\dots, s_l$, respectively. Then there exist  rational curves $\alpha_1,\dots, \alpha_l$ on $\SO^+_{p,n-p}(\mathbb{R})$ of degrees $2s_1,\dots, 2s_l$ respectively such that 
\[
\gamma(t) = \begin{bsmallmatrix}
I_n & u(t)/q_2(t) \\
0 & 1 
\end{bsmallmatrix} 
\begin{bsmallmatrix}
\alpha_1(t) & 0 \\
0 & 1 
\end{bsmallmatrix} \cdots \begin{bsmallmatrix}
\alpha_l(t) & 0 \\
0 & 1 
\end{bsmallmatrix}.
\]
\end{lemma}
\begin{proof}
We denote $\eta(t) \coloneqq Q(t)/q_1(t)$ and $x(t) \coloneqq u(t)/q_2(t)$. 
It is straightforward to verify that $\eta$ is a  rational curve on $\SO^+_{p,n-p}(\mathbb{R})$ and $\gamma(t) =
\begin{bsmallmatrix}
I_n & x(t) \\
0 & 1 
\end{bsmallmatrix}
\begin{bsmallmatrix}
\eta(t) & 0 \\
0 & 1
\end{bsmallmatrix}$. The desired decomposition of $\gamma(t)$ follows immediately from the decomposition of $\eta$ whose existence is guaranteed by Theorem~\ref{thm:KempeGBF}.
\end{proof}
\begin{lemma}\label{lem:circular}
Assume that the image of $x(t) \in \Rat_{2d}(\mathbb{R}^n,0)$ lies in a two dimensional subspace $\mathbb{V} \subseteq \mathbb{R}^{n}$.  Then  there are 
 rotations 
$\tau_1,\dots, \tau_{4d} \in \Rat_{2}(\SE_{2}(\mathbb{R}) ,I_2)$ and $Q \in \SO_n(\mathbb{R})$ such that
\begin{equation}\label{lem:circular:eq1}
\begin{bsmallmatrix}
I_n & x(t) \\
0 & 1
\end{bsmallmatrix} = \begin{bsmallmatrix}
Q^\tp & 0 \\
0 & 1
\end{bsmallmatrix} 
\iota_n(\tau_1) \cdots \iota_n(\tau_{4d})
\begin{bsmallmatrix}
Q & 0 \\
0 & 1
\end{bsmallmatrix}.
\end{equation}
Here $\iota_n: \SE_2(\mathbb{R}) \hookrightarrow \SE_n(\mathbb{R})$ is defined by 
\[
\begin{bsmallmatrix}
A & u \\
0 & 1
\end{bsmallmatrix} \mapsto 
\begin{bsmallmatrix}
A & 0 & u \\
0 & I_{n-2} & 0 \\
0 &  0 & 1
\end{bsmallmatrix}.
\]
\end{lemma}
\begin{proof}
We denote $\beta(t) \coloneqq \begin{bsmallmatrix}
I_n & x(t) \\
0 & 1
\end{bsmallmatrix}$.  Let $Q\in \SO_n(\mathbb{R})$ be such that $Q \mathbb{V} = \mathbb{R}^2 \times \{ 0\} \subseteq \mathbb{R}^n$. We have 
\[
\begin{bsmallmatrix}
Q & 0 \\
0 & 1
\end{bsmallmatrix} 
\beta(t)
\begin{bsmallmatrix}
Q^\tp & 0 \\
0 & 1
\end{bsmallmatrix} = 
\begin{bsmallmatrix}
I_n & Q x(t) \\
0 & 1
\end{bsmallmatrix}.
\]
Since $x(t)$ lies in $\mathbb{V}$,  $Q x(t) = (y(t),  0)^\tp$ lies in $\mathbb{R}^2 \times \{0\} \subseteq \mathbb{R}^n$ and by \cite{GKLRSV17}
, there are rotations 
$\tau_1,\dots, \tau_{4d} \in \Rat_{2}(\SE_2(\mathbb{R}),  I_2)$ such that 
\[
\begin{bsmallmatrix}
I_2 & y(t)^\tp \\
0 & 1
\end{bsmallmatrix} = \tau_1(t) \cdots \tau_{4d}(t).
\]
The proof is complete by applying $\iota_n$ to both sides.  
\end{proof}

\begin{lemma}\label{lem:circular SO21}
For each $x(t) \in \Rat_{2d}(\mathbb{R}^3,0)$,  there exist $P\in \ISO_{2,1}(\mathbb{R})$ and rotations 
$\tau_1,\dots, \tau_{8d} \in \Rat(\SE_2(\mathbb{R})), I_3)$ such that 
\begin{equation}\label{lem:circular SO21:eq1}
\begin{bsmallmatrix}
I_3 & x(t) \\
0 & 1
\end{bsmallmatrix} = 
\iota_3(\tau_1) \cdots 
\iota_3(\tau_{4d})
P
\iota_3(\tau_{4d+1}) \cdots \iota_3(\tau_{8d})P^{-1},
\end{equation}
where $\iota_3: \SE_2(\mathbb{R}) \hookrightarrow \SE_3(\mathbb{R}) \cap \ISO_{2,1}(\mathbb{R})$ is the map defined in Lemma~\ref{lem:circular}.
\end{lemma}
\begin{proof}
We parametrize $x(t)$ as $x(t) = \begin{bsmallmatrix}
x_1(t) & x_2(t) & x_3(t)
\end{bsmallmatrix}^\tp$ and observe that 
\[
x(t) = \beta_1(t) + \beta_2(t),\quad 
\beta_1(t)  \coloneqq \begin{bsmallmatrix}
x_1(t) - x_3(t) \\
x_2(t) - x_3(t) \\ 
0
\end{bsmallmatrix},\quad
\beta_2(t) \coloneqq \begin{bsmallmatrix}
x_3(t) \\
x_3(t) \\ 
x_3(t)
\end{bsmallmatrix},
\]
which implies 
\begin{equation}\label{lem:circular SO21:eq2}
\begin{bsmallmatrix}
I_3 & x(t) \\
0 & 1
\end{bsmallmatrix}
= 
\begin{bsmallmatrix}
I_3 & \beta_1(t) \\
0 & 1
\end{bsmallmatrix}
\begin{bsmallmatrix}
I_3 & \beta_2(t) \\
0 & 1
\end{bsmallmatrix}.
\end{equation}
By Lemma~\ref{lem:circular},  the first factor of the right side of \eqref{lem:circular SO21:eq2} admits a decomposition of the form \eqref{lem:circular:eq1} with $Q = I_3$.  Therefore,  it suffices to decompose the second factor.  To this end,  we let 
\[
Q = \begin{bsmallmatrix}
-\sqrt{3} u, &  \sqrt{3}v & \sqrt{2}  \\
 v &  u &  0 \\
\sqrt{2}u &  -\sqrt{2}v  &  -\sqrt{3}
\end{bsmallmatrix},\quad u \coloneqq \frac{1 + \sqrt{3}}{2\sqrt{2}},\quad v \coloneqq \frac{1 - \sqrt{3}}{2\sqrt{2}}.
\]
It is straightforward to verify that $Q\in \SO_{2,1}^+(\mathbb{R})$ and 
\[
\begin{bsmallmatrix}
1 & 0 & 0  & x_3(t) \\
0 & 1 & 0 &  x_3(t) \\
0 & 0 & 1 & x_3(t) \\
0 & 0 & 0 & 1
\end{bsmallmatrix}
= 
P
\begin{bsmallmatrix}
1 & 0 & 0  & -\frac{\sqrt{2}}{2} x_3(t)  \\
0 & 1 & 0 &  \frac{\sqrt{2}}{2} x_3(t) \\
0 & 0 & 1 & 0 \\
0 & 0 & 0 & 1
\end{bsmallmatrix} P^{-1},\quad P \coloneqq \begin{bsmallmatrix}
I_{2,1}Q^\tp I_{2,1} & 0 \\
0 & 1
\end{bsmallmatrix}.
\]
By Lemma~\ref{lem:circular},  we obtain a decomposition of the second factor of the right side of \eqref{lem:circular SO21:eq2} and this completes the proof.
\end{proof}

\begin{theorem}[Decomposition of rational curves on $\ISO^+_{p,n-p}(\mathbb{R})$]\label{thm:KempeSE}
Let $p\le n$ be non-negative integers and let $\gamma(t) \in \Rat(\ISO_{p,n-p}^+(\mathbb{R}), I_n)$ be parametrized as in \eqref{eq:curveSE}.  Suppose that $\deg(q_2) = 2d_2$ and that $q_1$ has $l$ roots of multiplicities $s_1,\dots, s_l$, respectively.  Then there exist $N \coloneqq 4d_2 ( \lceil p/2 \rceil + \lceil (n-p)/2 \rceil)$ quadratic rational curves $\beta_1,\dots, \beta_N \in \Rat_2(\SO_2,I_2)$,  $N$ matrices $P_1,\dots, P_N \in \ISO_{p,n-p}^+(\mathbb{R})$ and $l$  rational curves $\alpha_1,\dots, \alpha_l$ on $\SO^+_{p,n-p}(\mathbb{R})$ of degrees $2s_1,\dots, 2s_l$ respectively such that  
\[
\gamma(t) = \left( P_1 
\begin{bsmallmatrix} 
\beta_1(t) & 0 \\
0 & I_{n-1}
\end{bsmallmatrix}
 P_1^{-1} \right) \cdots \left(P_N 
\begin{bsmallmatrix} 
\beta_N(t) & 0 \\
0 & I_{n-1}
\end{bsmallmatrix}
 P_N^{-1}\right) 
 \begin{bsmallmatrix}
\alpha_1(t) & 0 \\
0 & 1 
\end{bsmallmatrix} \cdots \begin{bsmallmatrix}
\alpha_l(t) & 0 \\
0 & 1 
\end{bsmallmatrix}.
\]
\end{theorem}
\begin{proof}
Denote $x(t) \coloneqq u(t)/q_2(t)$.  By Lemma~\ref{lem:existenceSE},  it is sufficient to decompose the curve 
\[
\beta (t) \coloneqq \begin{bsmallmatrix}
I_n & x(t) \\
0 & 1 
\end{bsmallmatrix},\quad x(t) \coloneqq \begin{bsmallmatrix}
x_1(t) \\
\vdots \\
x_n(t)
\end{bsmallmatrix}.
\] 

If $\min \{p, n - p\} \ne 1$,  we observe that
\[
\beta (t) =  \begin{bsmallmatrix}
I_n & y_1(t) \\
0 & 1 
\end{bsmallmatrix} \cdots \begin{bsmallmatrix}
I_n & y_{\lceil p/2 \rceil}(t) \\
0 & 1 
\end{bsmallmatrix}
\begin{bsmallmatrix}
I_n & z_1(t) \\
0 & 1 
\end{bsmallmatrix} \cdots \begin{bsmallmatrix}
I_n & z_{\lceil n-p/2 \rceil}(t) \\
0 & 1 
\end{bsmallmatrix},
\]
where 
\begin{align*}
y_i(t) &=\begin{cases}
 \begin{bsmallmatrix}
0_{2(i-1)} &\ x_{2i -1}(t) &\ x_{2i}(t) &\  0_{n - 2i} 
\end{bsmallmatrix}^\tp,&~\text{if}~2i \le p \\
 \begin{bsmallmatrix}
0_{p-2} &\ 0 &\  x_p (t) &\  0_{n-p}
\end{bsmallmatrix}^\tp,&~\text{if}~2i-1 = p 
\end{cases} \\ 
z_j(t) &=\begin{cases}
 \begin{bsmallmatrix}
0_{p + 2(j-1)} &\ x_{p + 2j -1}(t) &\ x_{p + 2j}(t) &\  0_{n - p - 2j} 
\end{bsmallmatrix}^\tp,&~\text{if}~2j \le n - p \\
 \begin{bsmallmatrix}
0_{n-2} &\ 0 &\  x_{n}(t)
\end{bsmallmatrix}^\tp,&~\text{if}~2j-1 = n-p
\end{cases}
\end{align*}
Here for each positive integer $k$, $0_k$ denotes the zero vector in $\mathbb{R}^k$.  By Lemma~\ref{lem:circular}, each $\begin{bsmallmatrix}
I_n & y_i(t) \\
0 & 1 
\end{bsmallmatrix}$ (resp.  $\begin{bsmallmatrix}
I_n & z_j(t) \\
0 & 1 
\end{bsmallmatrix}$) admits a decomposition of the form \eqref{lem:circular:eq1}.  
Moreover,  every rotation $\tau \in \Rat_2(\SE_2(\mathbb{R}),  I_2)$ has a decomposition \cite{Selig2007geometric}
\[
\tau(t) = Q \begin{bsmallmatrix}
\beta(t) & 0 \\
0 & 1
\end{bsmallmatrix} Q^{-1},\quad Q\in \SE_2(\mathbb{R}),\quad \beta \in \Rat_2(\SO_2(\mathbb{R}),I_2).
\]
Since neither $p$ nor $n-p$ is equal to $1$,  constant matrices appeared in these decompositions are ensured to be contained in $\ISO_{p,n-p}(\mathbb{R})$ and the desired decomposition of $\gamma (t)$ follows immediately.

If $\min\{p,n-p\} = 1$,  we assume without loss of generality that $p = n-1$.  We notice that 
\[
\beta (t) =  \begin{bsmallmatrix}
I_n & y_1(t) \\
0 & 1 
\end{bsmallmatrix} \cdots \begin{bsmallmatrix}
I_n & y_{\lceil (n-3)/2 \rceil}(t) \\
0 & 1 
\end{bsmallmatrix}
\begin{bsmallmatrix}
I_n & z(t) \\
0 & 1 
\end{bsmallmatrix},
\]
where 
\begin{align*}
y_i(t) &=\begin{cases}
 \begin{bsmallmatrix}
0_{2(i-1)} &\ x_{2i -1}(t) &\ x_{2i}(t) &\  0_{n - 2i} 
\end{bsmallmatrix}^\tp,&~\text{if}~2i \le n-3 \\
 \begin{bsmallmatrix}
0_{n-5} &\ 0 &\  x_{n-3} (t) &\ 0_{3}
\end{bsmallmatrix}^\tp,&~\text{if}~2i -1 = n-3 
\end{cases}\\ 
z(t) &=
 \begin{bsmallmatrix}
0_{n-3} &\ x_{n-2}(t) &\ x_{n-1}(t) &\  x_{n}(t)
\end{bsmallmatrix}^\tp
\end{align*}
By Lemma~\ref{lem:circular},  each $\begin{bsmallmatrix}
I_n & y_i(t) \\
0 & 1 
\end{bsmallmatrix}$ has a decomposition of the form \eqref{lem:circular:eq1}.  According to Lemma~\ref{lem:circular SO21},  $\begin{bsmallmatrix}
I_n & z_j(t) \\
0 & 1 
\end{bsmallmatrix}$ admits a decomposition of the form \eqref{lem:circular SO21:eq1}.  
Therefore, in summation, we obtain the desired decomposition of $\gamma(t)$.
\end{proof}
\section{Generalizations of Kempe's Universality Theorem}\label{sec:gekempe}
As an application of Theorems~\ref{thm:KempeGBF} and \ref{thm:KempeSE},  we generalize Kempe's Universality Theorem in a  different way from the existing ones \cite{Abbott08, GKLRSV17, Gao01, KM96, KM02, Kourganoff16, LJS18}. The underlying idea of our generalization is analogous to that of the Erlangen program \cite{Klein1893}. Let $G$ be a real linear algebraic group and let $X$ be a real algebraic variety. Suppose that $X$ is a homogeneous space of $G$.  For ease of reference,  we state below the problem we will address in this section.  
\begin{problem}[Kempe's problem for homogeneous spaces]\label{prob:KempeProb}
Given a rational curve $\gamma$ on $X$ passing through $x_0\in X$,  are there low degree rational curves $\alpha_1,\dots, \alpha_s$ on $G$ such that $\alpha_1(t) \cdots \alpha_s(t) x_0 = \gamma (t)$?
\end{problem}
Unlike the commonly adopted formulation in \cite{GKLRSV17,KM02,LJS18},  the statement of Problem~\ref{prob:KempeProb} neither involves linkages nor their realizations. However, it turns out that rational curves on $G$ play the role of linkages and their orbits on $X$ are analogues of realizations of linkages.  Before we proceed,  we elaborate on the connection between Problem~\ref{prob:KempeProb} and the original Kempe's Universality Theorem.  
\begin{example}[Revisit of Kempe's Universality Theorem for rational planar curves]\label{ex:prob:KempeProb:SE2}
Let $X = \mathbb{R}^2$ and $G = \SE_2(\mathbb{R})$.  Clearly $X$ is a homogeneous space of $G$. For $x_0 = (0,0)^\tp$, we have a map $p: \SE_2(\mathbb{R}) \to \mathbb{R}^2$ defined by $p (g) \coloneqq g x_0$. Since $\pi$ has a section $s$ defined by sending each $x\in \mathbb{R}^2$ to the Euclidean translation by $x$, every rational curve $\gamma$ on $\mathbb{R}^2$ can be lifted to a rational curve $\widetilde{\gamma} = s \circ \gamma$ on $\SE_2(\mathbb{R})$.  Moreover,  Theorem~\ref{thm:KempeSE} implies that $\widetilde{\gamma}$ admits a decomposition
\[
\widetilde{\gamma}(t) = \prod_{i=1}^{4d} P_i \begin{bsmallmatrix} 
\theta_i (t) & 0 \\
0 & 1 \\
\end{bsmallmatrix}
P_i^{-1}, 
\]
where $\theta_i\in \Rat_2(\SO_2(\mathbb{R}),I_2)$ and $P_i\in \SE_2(\mathbb{R}), 1\le i \le 4d$. As a consequence, we have 
\begin{equation}\label{ex:prob:KempeProb:SE2:eq1}
\gamma(t) = p (\widetilde{\gamma}(t)) = \prod_{i=1}^{4d} P_i \begin{bsmallmatrix} 
\theta_i (t) & 0 \\
0 & 1 \\
\end{bsmallmatrix}
P_i^{-1} x_0.
\end{equation}
Therefore,  every  rational curve of degree $2d$ on $\mathbb{R}^2$ can be traced by a product of $4d$ quadratic rotations, each of which is conjugated by some element in $\SE_2(\mathbb{R})$. Since each quadratic rotation can be realized by a simple linkage \cite{GKLRSV17, LJS18}, Kempe's Theorem for plannar rational curves is a direct consequence of the decomposition \eqref{ex:prob:KempeProb:SE2:eq1}.
\end{example}

We observe that Problem~\ref{prob:KempeProb} can be solved by two steps: The first step is to find a rational curve $\widetilde{\gamma}:\mathbb{P}_{\mathbb{R}}^1 \to G$ such that $p \circ \widetilde{\gamma} = \gamma$.  The second step is to decompose $\widetilde{\gamma}$ into a product of low degree rational curves.  In particular,  the desired $\widetilde{\gamma}$ must be a lift of $\gamma$. The two steps are pictorially summarized in the diagram below.
\[\begin{tikzcd}
	& G \\
	\mathbb{P}_{\mathbb{R}}^1 & X
	\arrow[" p ", from=1-2, to=2-2]
	\arrow["{\widetilde{\gamma} = \prod_{j=1}^s \alpha_j }", dashed, from=2-1, to=1-2]
	\arrow["\gamma"', from=2-1, to=2-2]
\end{tikzcd}\]
\subsection{Generalized Kempe's Universality Theorem for loops}\label{subsec:top lift}
Since $\mathbb{P}^1_{\mathbb{R}}$ is homeomorphic to $\mathbb{S}^1$,  there is no harm to identify $\mathbb{P}^1_{\mathbb{R}}$ with $\mathbb{S}^1$ in this subsection.  Let $X$ be a topological space.   A continuous map $\gamma: \mathbb{S}^1 \to X$ is called a \emph{loop} on $X$.  First we establish a criterion for the existence of a lift of a loop,  which is similar to the well-known lifting criterion for covering spaces \cite[Proposition~1.33]{Hatcher02}.
\begin{lemma}[Topological lifting criterion]\label{lem:lifting criterion-1}
Let $H$ be a topological group and let $p : P \to X$ be a principal $H$-bundle.  A continuous loop $\gamma: \mathbb{S}^1 \to X$ admits a lift if and only if $[\gamma] \in p_\ast (\pi_1(P)) \subseteq \pi_1(X)$.  In particular,  any $\gamma$ admits a lift if either $X$ is simply connected or $H$ is connected.
\end{lemma}
\begin{proof}
Clearly,  $\gamma$ has a lift implies that $[\gamma] \in p_\ast (\pi_1(P)) \subseteq \pi_1(X)$.  For the converse,  we consider the following diagram 
\[\begin{tikzcd}
	{\gamma^\ast(P)} & {P \simeq f^\ast(EH)} & EH \\
	{\mathbb{S}^1} & X & BH
	\arrow["{\iota_\gamma}", from=1-1, to=1-2]
	\arrow["\theta"', from=1-1, to=2-1]
	\arrow["{\iota_f}", from=1-2, to=1-3]
	\arrow["p", from=1-2, to=2-2]
	\arrow["\eta", from=1-3, to=2-3]
	\arrow["s", curve={height=-12pt}, dashed, from=2-1, to=1-1]
	\arrow["\beta", dashed, from=2-1, to=1-2]
	\arrow["\gamma"', from=2-1, to=2-2]
	\arrow["f"', from=2-2, to=2-3]
\end{tikzcd}\]
where $BH$ is the classifying space of $H$,  $\eta: EH \to BH$ is the universal principal $H$-bundle,  $f$ is a continuous map such that $P \simeq f^\ast (EH)$,  $\theta: \gamma^\ast(P) \to \mathbb{S}^1$ is the pull-back of $p: P \to X$ by $\gamma$,  $\iota_\gamma$ and $\iota_f$ are maps induced by $\gamma$ and $f$ respectively.  The commutativity of this diagram implies
\begin{align*}
\text{$\gamma$ has a lift $\beta$} &\iff \text{the principal $H$-bundle $\theta: \gamma^\ast(P) \to \mathbb{S}^1$ has a section $s$} \\
&\iff \text{$\theta: \gamma^\ast(P) \to \mathbb{S}^1$ is a trivial principal $H$-bundle} \\
&\iff \text{$f\circ \gamma$  is null-homotopic} \\
&\iff f_{\ast} ([\gamma]) = 0 \in \pi_1 (BH) \xrightarrow{\delta}  \pi_0(H).
\end{align*}
Here the map $\delta$ in the last line is the first boundary map in the long exact sequence of homotopy groups for the fibration $\eta: EH \to BH$.  Since $EH$ is contractible,  $\pi_1(EH) = 0$ and $\delta$ is an isomorphism.  In particular,  $f_\ast([\gamma]) = 0$ is always satisfied if either $X$ is simply connected or $H$ is connected,  from which we conclude that $\gamma$ has a lift.

In general,  if $[\gamma] = p_\ast ([\alpha])$ for some $[\alpha]\in \pi_1(P)$,  then we have $f_\ast ([\gamma]) = f_\ast \circ p_\ast ([\alpha]) = \eta_\ast \circ (\iota_f)_\ast ([\alpha]) = 0$ as $\pi_1(EH)=0$.  Therefore,  $\gamma$ admits a lift.
\end{proof}

\begin{proposition}\label{prop:approx-lift}
Let $G$ be a real linear algebraic group and let $X$ be a homogeneous variety of $G$.  Assume $x_0\in X$ is a fixed point and $p:G \to X$ is the map defined by $p(g) = g x_0$.  If each class in $\pi_1(G)$ is represented by a rational curve on $G$,  then for any loop $\gamma: \mathbb{S}^1 \to X$ such that $[\gamma] \in p_\ast (\pi_1(G))$,  there exists a sequence of rational curves $\{\beta_n\}_{n=1}^\infty$ on $G$ such that $\{  \beta_n x_0  \}_{n=1}^\infty$ converges to $\gamma$ uniformly. 
\end{proposition}
\begin{proof}
Since $[\gamma] \in p_\ast (\pi_1(G))$,  Lemma~\ref{lem:lifting criterion-1} ensures that $\gamma$ has a lift $\beta: \mathbb{S}^1 \to G$.  By assumption,  $\beta$ is homotopic to a rational curve on $G$.  According to Theorem~\ref{thm:approx by regular maps},  $\beta$ is uniformly approximated by rational curves on $G$.
\end{proof}
\begin{remark}\label{ex:prop:approx-lift}
If $\gamma: \mathbb{S}^1 \to X$ can be uniformly approximated by $\{ \beta_n x_0 \}_{n=1}^\infty$ for a sequence $\{\beta_n\}_{n=1}^\infty$ of rational curves on $G$,  then Theorem~\ref{thm:approx by regular maps} implies that $\gamma$ is homotopic to a rational curve $\alpha$ on $X$.  In fact,  we must have $[\gamma] = [\alpha] \in p_\ast(\pi_1(G))$.  However,  it is not true that for any $\gamma: \mathbb{S}^1 \to X$ which is homotopic to a rational curve,  there exists a sequence of rational curves $\{\beta_n\}_{n=1}^\infty$ on $G$ such that $ \{\beta_n x_0\}_{n=1}^\infty$ uniformly converges to $\gamma$.  As an example,  we consider $(G,X) = (\mathbb{R}, \mathbb{S}^1)$ and $\gamma = \Id_{\mathbb{S}^1}$.  It is clear that $\gamma$ is a rational curve on $\mathbb{S}^1$,  but it has no lift since $[\gamma] = 1\in \mathbb{Z} \simeq \pi_1(\mathbb{S}^1)$.
\end{remark}

\begin{corollary}\label{cor:approx-lift}
If both $G$ and $X$ are simply connected,  then for every loop $\gamma: \mathbb{S}^1 \to X$,  there exists a sequence of rational curves $\{\beta_n\}_{n=1}^\infty$ on $G$ such that $\{ \beta_n x_0 \}_{n=1}^\infty$ uniformly converges to $\gamma$.  
\end{corollary}

Given non-negative integers $p < n$ and $0 < n_1 <  \cdots <  n_k < n$,  we denote 
\begin{align*}
\H_{p+1,n-p} &\coloneqq \{ x = (x_0,\dots, x_n) \in \mathbb{R}^{n+1}: x^\tp I_{p+1,n-p} x = 1,  x_0 \text{ for $p=0$}\},  \\
\V_{p,n}(\mathbb{R}) &\coloneqq \{ X \in \mathbb{R}^{n \times p}:  X^\tp X = I_p \},  \\
\Flag^o(n_1,\dots,  n_k; \mathbb{R}^n) &\coloneqq \lbrace
(\mathbb{V}_1,  \cdots,  \mathbb{V}_k): \mathbb{V}_j \subseteq \mathbb{V}_{j+1} \subseteq \mathbb{R}^n,  \dim \mathbb{V}_j  = n_j,\text{$\mathbb{V}_j$ is an oriented subspace}
\rbrace,  \\
\V_{p,n}(\mathbb{C}) &\coloneqq \{ X \in \mathbb{C}^{n \times p}:  X^\ast  X = I_p \},  \\
\Flag(n_1,\dots,  n_k; \mathbb{C}^n) &\coloneqq \lbrace
(\mathbb{V}_1,  \cdots,  \mathbb{V}_k): \mathbb{V}_j \subseteq \mathbb{V}_{j+1} \subseteq \mathbb{C}^n,  \dim \mathbb{V}_j  = n_j,\text{$\mathbb{V}_j$ is a subspace}
\rbrace,  \\
\V_{p,n}(\mathbb{H}) &\coloneqq \{ X \in \mathbb{H}^{n \times p}:  X^\ast X = I_p \}.
\end{align*}
We recall that all of these are homogeneous spaces:
\begin{align*}
\H_{p+1,n-p} &\simeq \SO_{p+1,n-p}^+/\SO_{p,n-p}^+,\quad  \V_{p,n}(\mathbb{H}) \simeq \Sp_n(\mathbb{H})/\Sp_{n-p}(\mathbb{H}),  \\
\V_{p,n}(\mathbb{R}) &\simeq \SO_n(\mathbb{R})/\SO_{n-p}(\mathbb{R}),\quad \Flag^o(n_1,\dots,  n_k; \mathbb{R}^n) \simeq \SO_{n}(\mathbb{R}) / \prod_{j=0}^k \SO_{n_{j+1} - n_{j}}(\mathbb{R}),\\ \V_{p,n}(\mathbb{C}) &\simeq \SU_n/\SU_{n-p},\quad \Flag(n_1,\dots,  n_k; \mathbb{C}^n) \simeq \SU_{n}/\prod_{j=0}^k \SU_{n_{j+1} - n_{j}}.
\end{align*} 
Moreover,   let $\mathbb{R}^{p,n-p-1}$ be $\mathbb{R}^{n-1}$ equipped with the standard pseudo Riemannian metric of signature $(p,n-p-1)$.  Then the conformal group $\Conf(\mathbb{R}^{p,n-p-1})$ is isomorphic to $\SO_{p+1,n-p}^+(\mathbb{R})$ \cite[Chapter 2]{Schottenloher97mathematical}.  In particular,  $\mathbb{R}^{p,n-p-1}$ is a homogeneous space of $\SO_{p+1,n-p}^+(\mathbb{R})$.
\begin{theorem}[Generalized Kempe's Universality Theorem I]\label{thm:approx-lift}
Let $(G,X)$ be one of the following pairs:
\begin{enumerate}[(i)]
\item\label{prop:ex-approx-lift:item1} $(G,X) = (\SO_{n}(\mathbb{R}),   \V_{p,n}(\mathbb{R})),  n\ge 3,  n - p \ge 2$.  
\item\label{prop:ex-approx-lift:item2} $(G,X) = (\SE_{n}(\mathbb{R}),  \mathbb{R}^n)$.
\item\label{prop:ex-approx-lift:item3} $(G,X) = (\SO_{n}(\mathbb{R}),  \Flag^o(n_1,\dots,  n_k; \mathbb{R}^n)),  n\ge 3$. 
\item\label{prop:ex-approx-lift:item4} $(G,X) = ( \SO^+_{p+1,  n-p}(\mathbb{R}),  \H_{p+1,  n-p})$.
\item\label{prop:ex-approx-lift:item4'} $(G,X) = ( \SO^+_{p+1,  n-p}(\mathbb{R}),  \mathbb{R}^{p,n-p-1})$.
\item\label{prop:ex-approx-lift:item5}  $(G,X) = (\ISO_{p+1,n-p}(\mathbb{R}),  \mathbb{R}^{n+1})$
\item\label{prop:ex-approx-lift:item6} $(G,X) = (\SU_n,  \V_{p,n}(\mathbb{C}))$ for $n \ge 2$.
\item\label{prop:ex-approx-lift:item7} $(G,X) = ( \SU_n,  \Flag(n_1,\dots,  n_k; \mathbb{C}^n))$.
\item\label{prop:ex-approx-lift:item8} $(G,X) = (\Sp_{n}(\mathbb{H}),  \V_{p,n}(\mathbb{H}))$.
\end{enumerate}
Then for every loop $\gamma: \mathbb{S}^1 \to X$,  there exists a sequence of rational curves $\{\beta_n\}_{n=1}^\infty$ on $G$ such that $\{ \beta_n x_0 \}_{n=1}^\infty$ uniformly converges to $\gamma$.  Moreover,  each $\beta_n$ can be decomposed as $\beta_n = \alpha_{n,1} \cdots \alpha_{n,s_n}$, where $\alpha_{n,j}\in \Rat(G)$ only has poles at $\{c_{n,j}, \overline{c}_{n,j} \}$ such that $\deg(\beta_n) = \sum_{j=1}^{s_n} \deg(\alpha_{n,j})$ and $\{c_{n,j}, \overline{c}_{n,j} \}$ $\ne$ $\{c_{n,k}, \overline{c}_{n,k} \}$ if $j\ne k$.
\end{theorem}
\begin{proof}
According to Theorems~\ref{thm:KempeGBF} and \ref{thm:KempeSE},  it suffices to prove the existence of $\{\beta_n\}_{n=1}^\infty$.  

In \eqref{prop:ex-approx-lift:item1}--\eqref{prop:ex-approx-lift:item3},  we have $\pi_1(G) = \mathbb{Z}_2$.  It is clear that the non-trivial class of $\pi_1(G)$ is represented by the non-trivial quadratic rational curve on $\SO_2(\mathbb{R})$ (cf. ~Example~\ref{ex:O_n}) via the natural embedding $\SO_2(\mathbb{R}) \hookrightarrow G$.  Since $X$ in \eqref{prop:ex-approx-lift:item1} and \eqref{prop:ex-approx-lift:item2} are simply connected,  the result immediately follows from Proposition~\ref{prop:approx-lift}.  For \eqref{prop:ex-approx-lift:item3},  we have $X = G/H$ where 
\[
H \coloneqq  \SO_{n_1}(\mathbb{R}) \times \SO_{n_2-n_1}(\mathbb{R}) \cdots \times \SO_{n_k - n_{k-1}}(\mathbb{R}) \times \SO_{n - n_k}(\mathbb{R}).
\]
Since $H$ is connected,  the result is obtained by Lemma~\ref{lem:lifting criterion-1} and Proposition~\ref{prop:approx-lift}.

For \eqref{prop:ex-approx-lift:item4} and \eqref{prop:ex-approx-lift:item4'}  we observe that $\SO^+_{p+1, n-p}(\mathbb{R})$ is homotopy equivalent to its subgroup $\SO_{p+1}(\mathbb{R}) \times \SO_{n-p}(\mathbb{R})$.  Thus,  $\pi_1(\SO^+_{p+1, n-p}(\mathbb{R})) = \pi_1(\SO_{p+1}(\mathbb{R})) \times \pi_1(\SO_{n-p}(\mathbb{R}))$ and each class can be represented by a quadratic rational curve on $\SO_{p+1}(\mathbb{R})$ or $\SO_{n-p}(\mathbb{R})$.  Since $\mathbb{R}^{n-1}$ and $H_{p+1, n-p}$ are simply connected,  the proof is complete by Proposition~\ref{prop:approx-lift}.  The proof for \eqref{prop:ex-approx-lift:item5} is similar,  as $\ISO^+_{p+1,n-p}(\mathbb{R})$ is homotopic equivalent to $\SO^+_{p+1,n-p}(\mathbb{R})$.

Lastly,  we notice that $G$ and $X$ in \eqref{prop:ex-approx-lift:item6}--\eqref{prop:ex-approx-lift:item8} are all simply connected.  Thus,  Corollary~\ref{cor:approx-lift} applies.  
\end{proof}
\begin{remark}
On the one hand,  homogeneous spaces considered in Theorem~\ref{thm:approx-lift} are of great importance in mathematics and physics.  For instance,  Stiefel manifolds $\V_{p,n}(\mathbb{R})$ and oriented flag manifolds $\Flag^o(n_1,\dots, n_k;\mathbb{R}^n)$ are important computational platforms in algebraic topology \cite{Milnor74,Steenrod62} and manifold optimization \cite{EAS99,YWL22}.  The hyperbolic space $H_{1,n} \simeq \SO_{1,n}^+/\SO_{n}$ is the model space for hyperbolic geometry \cite{Ratcliffe1994foundations}.  The pseudo-Euclidean space $\mathbb{R}^{p,q}$ plays a fundamental role in both Lorentzian geometry \cite{BEE96} and the study of general relativity \cite{HE23}.  The de Sitter spacetime (resp.  anti de Sitter spacetime) $H_{n,1} \simeq \SO^+_{n,1}/\SO^+_{n-1,1} $ (resp.  $H_{2,n-1} \simeq \SO^+_{2,n-1}/\SO^+_{1,n-1} $) is extensively studied in cosmology and quantum field theory \cite{AM84,Gazeau07}.  

On the other hand,  different choices of $G$ for the same $X$ allow us to study curves on $X$ with respect to different geometries.  Take $X = \mathbb{R}^n$ for example.  Theorem~\ref{thm:approx-lift} for $G = \SE_n(\mathbb{R})$ (cf.  item~\eqref{prop:ex-approx-lift:item2}) means any continuous loop in $\mathbb{R}^n$ can be approximately traced out by rational curves of rigid transformations,  while Theorem~\ref{thm:approx-lift} for $G = \SO^+_{p+1,n-p + 1}(\mathbb{R})$ (cf. item~\eqref{prop:ex-approx-lift:item4'}) (resp.  $G = \ISO_{n-1,1}(\mathbb{R})$ (cf. item~\eqref{prop:ex-approx-lift:item5}) implies continuous loops can be approximately traced out by rational curves of conformal (resp. spacetime preserving) transformations.
\end{remark}
\subsection{Generalized Kempe's Universality Theorem} Let $G$ be a real linear algebraic group and let $X$ be a homogeneous $G$-variety.  Assume $x_0 \in X$ is a fixed point and $p: G \to X$ is the map $p(g) = g x_0$.  According to Lemma~\ref{lem:lifting criterion-1},  the existence of a continuous lift (in Euclidean topology) of a rational curve $\gamma: \mathbb{P}_{\mathbb{R}}^1 \to X$ passing through $x_0$ is determined by its class $[\gamma] \in \pi_1(X)$.  However,  Problem~\ref{prob:KempeProb} requires the lift to be rational.  This subsection is devoted to a discussion of the rationality of a lift,  from which we obtain a generalized Kempe's Universality Theorem.

Let $G$ be a real linear algebraic group and let $X$ be a homogeneous $G$-variety.  Assume that $x_0 \in X$ is a fixed point,  $p: G \to X$ is the map defined by $p(g) = g x_0$.  We denote  $H \coloneqq \Stab_{x_0}(G)$ and consider the following diagram
\[\begin{tikzcd}
	{\gamma^\ast G} & G \\
	{\mathbb{P}_{\mathbb{R}}^1} & X \simeq G/H
	\arrow["\theta"', from=1-1, to=2-1]
	\arrow["p", from=1-2, to=2-2]
	\arrow["s", curve={height=-12pt}, dashed, from=2-1, to=1-1]
	\arrow["\beta", dashed, from=2-1, to=1-2]
	\arrow["\gamma"', from=2-1, to=2-2]
\end{tikzcd}\]
where $\gamma$ is a rational curve on $X$ passing through $x_0$,  $\theta$ is the projection map of the principal $H$-bundle $\gamma^\ast G$ over $\mathbb{S}^1$.  Clearly,  we have 
\begin{align*}
\text{$\gamma$ has a rational lift $\beta$} &\iff \text{$\gamma^\ast G$ admits a rational section $s$} \\ 
&\iff \text{$\gamma^\ast G$ is a trivial algebraic principal $H$-bundle}.
\end{align*}
We notice that $\mathbb{P}_{\mathbb{R}}^1$ is a smooth affine curve over $\mathbb{R}$.  The lemma that follows is a direct consequence of Proposition~\ref{prop:rational triviality}. 
\begin{lemma}[Rational lifting criterion]\label{lem:lifting criterion-2}
Assume that $H$ is semisimple and simply connected.  If $\theta: \gamma^\ast G \to \mathbb{P}_{\mathbb{R}}^1$ is Zariski locally trivial,  then $\gamma$ admits a rational lift.  In particular,  if $p:G \to X$ is Zariski locally trivial,  then every rational curve on $X$ has a rational lift.
\end{lemma}
\begin{theorem}[Generalized Kempe's Universality Theorem II]\label{thm:Kempe G-space}
Let $(G,X)$ be one of the nine pairs listed in Theorem~\ref{thm:approx-lift}.  For every $\gamma\in \Rat(X,x_0)$,  there exist $\alpha_1,\dots,  \alpha_s \in \Rat(G,  I)$ such that 
\begin{enumerate}[(a)]
\item Each $ \alpha_j$only has poles at $\{c_j,\overline{c}_j\}$,  $1 \le j \le s$.  
\item If $j \ne k$ then $\{c_j,\overline{c}_j\} \ne \{c_k,\overline{c}_k\}$. 
\item $\prod_{j=1}^s \alpha(t) x_0  = \gamma(t)$.
\end{enumerate}
Here $I$ denotes the identity element in $G$.
\end{theorem}
\begin{proof}
By Theorems~\ref{thm:KempeGBF} and \ref{thm:KempeSE},  it is sufficient to prove the existence of a rational lift of $\gamma$. 

For \eqref{prop:ex-approx-lift:item4'},  we consider 
\[\begin{tikzcd}
	{\ISO^+_{p,n-p-1}(\mathbb{R})} & {\SO^+_{p+1,n-p}(\mathbb{R})} \\
	\\
	{\mathbb{P}_{\mathbb{R}}^1} & {\mathbb{R}^{n-1}}
	\arrow["j", hook, from=1-1, to=1-2]
	\arrow["{p \circ j}"', from=1-1, to=3-2]
	\arrow["p", from=1-2, to=3-2]
	\arrow["\beta", dashed, from=3-1, to=1-1]
	\arrow["\gamma", from=3-1, to=3-2]
\end{tikzcd}\]
where $j$ is the inclusion of $ \ISO^+_{p,n-p-1}(\mathbb{R})$ into $\SO^+_{p+1,n-p}(\mathbb{R})$  and $p$ is the projection map defined by the action of $\SO^+_{p+1,n-p}(\mathbb{R})$ on $\mathbb{R}^{n-1}$.  Hence it is reduced to prove \eqref{prop:ex-approx-lift:item5}.

For \eqref{prop:ex-approx-lift:item1}--\eqref{prop:ex-approx-lift:item4} and \eqref{prop:ex-approx-lift:item5},  we let $\widetilde{G}$ the universal covering of $G$.  Since $G$ in each of these cases is a semi-direct product of some $\SO^+_{p,q}$ and $\mathbb{R}^m$,  the corresponding $\widetilde{G}$ is also a semi-direct product of $\Spin_{p,q}(\mathbb{R})$  and $\mathbb{R}^m$.  In particular,  $\widetilde{G}$ is a real linear algebraic group.  Thus we may consider the following diagram.:
\[\begin{tikzcd}
	& {\widetilde{G}} \\
	& G \\
	{\mathbb{P}_{\mathbb{R}}^1} & {X \simeq G/H \simeq \widetilde{G}/\widetilde{H}}
	\arrow["\pi", from=1-2, to=2-2]
	\arrow["p", from=2-2, to=3-2]
	\arrow["{\widetilde{\beta}}", curve={height=-12pt}, dashed, from=3-1, to=1-2]
	\arrow["\beta", dashed, from=3-1, to=2-2]
	\arrow["\gamma"', from=3-1, to=3-2]
\end{tikzcd}\]
where $\pi: \widetilde{G} \to G$ is the covering map and $\widetilde{H} \coloneqq \Stab_{\widetilde{G}}(x_0)$.  Obviously,  if $\widetilde{\beta}$ is a rational lift of $\gamma$ to $\widetilde{G}$,  then $\beta \coloneqq \pi \circ \widetilde{\beta}$ is a rational lift of $\gamma$ to $G$.  Since $H$ is the product of $\SO^+_{p,q}$,  $\widetilde{H}$ is a product of $\Spin_{p,q}(\mathbb{R})$,  which is a semisimple and simply connected algebraic group.  It is straightforward to verify that $p: G \to X$ is Zariski locally trivial.  Thus,  $p \circ \pi: \widetilde{G} \to X$ is also Zariski locally trivial.  The existence of $\widetilde{\beta}$ and $\beta$ follows from Lemma~\ref{lem:lifting criterion-2}.

For \eqref{prop:ex-approx-lift:item6}--\eqref{prop:ex-approx-lift:item8},  we notice that $p: G \to X$ is Zariski locally trivial and $H$ is a semi-simple and simply connected algebraic group.  Hence Lemma~\ref{lem:lifting criterion-2} is applicable.
\end{proof}
\subsection{Examples of small dimensions}
In this subsection,  we briefly discuss some low dimensional examples,  which have been well-studied in geometric algebra and theoretical mechanism.  We notice that in the literature \cite{Hegedus2013factorization},  rational curves are sometimes allowed to have poles in the real line.  In this context,  rational curves considered in this paper correspond to bounded  \emph{motion polynomials} \cite{GKLRSV17,Hegedus2013factorization,LJS18}. 
\subsubsection*{Rational curves on \texorpdfstring{$\SO_3(\mathbb{R})$}{SO3} and their geometric algebra model} Let $\mathbb{H}_1$ be the group of unit quaternions in $\mathbb{H}$.  We consider the 2-1 covering map $p: \mathbb{H}_1 \to \SO_3(\mathbb{R})$ given by 
\[
p(a + b \qi + c \qj + d \qk) \coloneqq \begin{bsmallmatrix}
1 - 2c^2 - 2d^2 & 2bc  - 2 ad & 2 bd + 2 ac  \\ 
2bc  + 2 ad &  1 - 2b^2 - 2d^2 &   2cd - 2 ab \\
2bd - 2 ac & 2cd + 2 ab &  1 - 2b^2 - 2c^2
\end{bsmallmatrix}.
\] 
Let $\gamma \in \Rat_2(\SO_3(\mathbb{R}), I_2)$ be given by
\[
\gamma(t) = \begin{bmatrix}
\frac{t^2 - 1}{t^2 + 1} & \frac{2t}{t^2 + 1} & 0 \\[2pt]
-\frac{2t}{t^2 + 1} & \frac{t^2 - 1}{t^2 + 1} & 0 \\[2pt]
0 & 0 & 1
\end{bmatrix}.
\]
Since $[\gamma] = 1 \in \mathbb{Z}_2 =  \pi_1(\SO_3(\mathbb{R}))$,  Lemma~\ref{lem:lifting criterion-1} implies that there is no $\beta \in \Rat(\mathbb{H}_1)$ such that $p \circ \beta = \gamma$.  Thus,  $\Rat(\SO_3(\mathbb{R}))$ is a strictly bigger set than $\Rat(\mathbb{H}_1)$.  

However,  by the path lifting property for a covering space \cite[Proposition~1.30]{Hatcher02},  there must exist some $f: [0,1] \to \mathbb{H}_1$ such that $p \circ f = \gamma$ with $f(0) \ne f(1)$.  Indeed,  it is straightforward to verify that $f(x)=(cos(-\pi/2+x\pi)-sin(-\pi/2+x\pi)\qk)$ is such a map.  We notice that $\gamma(t)=p \circ g$ with $g(t)=(t-\qk)/\sqrt{t^2+1}$ and $t=cos(-\pi/2+x\pi)/sin(-\pi/2+x\pi)$ is the normalization of the motion polynomial $t-\qk$ discussed in \cite{Hegedus2013factorization}.

\subsubsection*{Rational planar curves in Euclidean geometry} We consider  $(G,X) = (\SE_2(\mathbb{R}), \mathbb{R}^2)$.  As in Example~\ref{ex:prob:KempeProb:SE2}, we have Kempe's Universality Theorem for rational planar curves \cite{GKLRSV17}.

\subsubsection*{Rational space curves in Euclidean geometry} We consider $(G,X) = (\SE_3(\mathbb{R}), \mathbb{R}^3)$.  By Theorems~\ref{thm:KempeSE} and \ref{thm:Kempe G-space} (cf. Example~\ref{ex:prob:KempeProb:SE2}),  we obtain Kempe's Universality Theorem for rational space curves \cite{LJS18}.

\subsubsection*{Rational planar curves in conformal geometry}\label{subsubsec:planar curve conformal}
Let $(G,X,  x_0) = (\SO^+_{3,1}(\mathbb{R}), \mathbb{R}^2, (0,0)^\tp)$.  By Example~\ref{ex:prob:KempeProb:SE2}, every $\gamma \in \Rat_{2d}(\mathbb{R}^2, x_0)$ can be written as 
\begin{equation}\label{subsubsec:planar curve conformal:eq1}
\gamma(t)  = \prod_{i=1}^{4d} P_i \begin{bsmallmatrix} 
\theta_i (t) & 0\\
0 & 1
\end{bsmallmatrix}
P_i^{-1} x_0,
\end{equation}
for some $P_i\in \SE_2(\mathbb{R})$ and $\theta_i \in \Rat_2(\SO_2(\mathbb{R}),I_2)$,  $1 \le i \le 4d$.  Since $\SE_2(\mathbb{R})$ is a subgroup of $\SO_{3,1}^+(\mathbb{R})$ and the induced inclusion $\Rat(\SE_2(\mathbb{R}),I_2) \subseteq \Rat(\SO^+_{3,1}(\mathbb{R}),I_4)$ preserves the degree of $\begin{bsmallmatrix} 
\theta_i (t) & 0\\
0 & 1
\end{bsmallmatrix}$,  we conclude that 
\begin{equation}\label{subsubsec:planar curve conformal:eq2}
\gamma(t)  = \prod_{j=1}^s \alpha_j(t) x_0
\end{equation}
for some $s \le 4d$ and $\alpha_j \in \Rat_2(\SO^+(3,1),  I_4)$,  $1 \le j \le s$.  Moreover,  by Examples~\ref{ex:O_n} and \ref{ex:O_n1},  each $\alpha_j$ has one of the following two forms: 
\[
P \begin{bsmallmatrix}
\frac{(t-a)^2 - b^2}{(t-a)^2 + b^2} & \frac{2b(t-a)}{(t-a)^2 + b^2} & 0 & 0 \\
-\frac{2b(t-a)}{(t-a)^2 + b^2} & \frac{(t-a)^2 - b^2}{(t-a)^2 + b^2} & 0 & 0 \\
0 & 0 & 1 & 0 \\
0 & 0 & 0 & 1 \\
\end{bsmallmatrix} P^{-1},  \quad P 
\begin{bsmallmatrix}
1 & 0 & 0 & 0 \\
0& 0 & 1 &  0 \\
0 & \frac{\sqrt{2}}{2} & 0 & -\frac{\sqrt{2}}{2} \\
0 & -\frac{\sqrt{2}}{2} & 0  & -\frac{\sqrt{2}}{2}
\end{bsmallmatrix}
\begin{bsmallmatrix}
 1 & \frac{b^2 y}{(t-a)^2 + b^2} & 0 & 0  \\
 0 & 1 &  0 & 0  \\ 
0 & \frac{b (t-a)}{(t-a)^2 + b^2}  & 1 & 0\\
\frac{b^2 y}{(t-a)^2 + b^2} & \frac{b^2}{2((t-a)^2 + b^2)}  &  \frac{b(t-a)}{(t-a)^2 + b^2} & 1  
\end{bsmallmatrix}  
\begin{bsmallmatrix}
1 & 0 & 0 & 0 \\
0& 0 & \frac{\sqrt{2}}{2} &  -\frac{\sqrt{2}}{2} \\
0 & 1 & 0 & 0 \\
0 & 0 & -\frac{\sqrt{2}}{2} & -\frac{\sqrt{2}}{2}
\end{bsmallmatrix}
P^{-1},
\]
where $(a,b)\in \mathbb{R} \times (\mathbb{R} \setminus \{0\})$,  $y\in \{-1,1\}$ and $P\in \SO^+_{3,1}(\mathbb{R})$.  We notice that both rotations and \emph{conformal rotations} \cite{Dorst2016construction, Kalkan2022study} are of the first type,  while \emph{circular translations} \cite{Hegedus2013factorization,Li2024quadratic} in $\mathbb{R}^2$ are of the second type.  In particular,  by comparing \eqref{subsubsec:planar curve conformal:eq1} and \eqref{subsubsec:planar curve conformal:eq2},  there is no essential distinction between 2D kinematics in Euclidean geometry and Conformal geometry,  in the sense of rational curves. 
\subsubsection*{Rational space curve in conformal geometry} 
Let $(G,X, x_0) = (\SO^+_{4,1}(\mathbb{R}),  \mathbb{R}^3, (0,0,0)^\tp)$.  Since $\SE_3(\mathbb{R})$ is a subgroup of $\SO^+_{4,1}(\mathbb{R})$,  the same argument as for $(\SO^+_{3,1}(\mathbb{R}), \mathbb{R}^2, (0,0)^\tp)$ implies that every $\gamma \in \Rat_{2d}(\mathbb{R}^3, x_0)$ can be written as 
\[
\gamma(t) = \prod_{j=1}^s \alpha_j(t) x_0
\]
for some $s \le 4d$ and $\alpha_1,\dots, \alpha_s \in \Rat_2(\SO^+_{4,1}(\mathbb{R}), I_5)$.  Furthermore,  Theorem~\ref{thm:classification O_n1} implies that each $\alpha_j$ must have one of the following three forms:
\begin{align*}
&P \begin{bsmallmatrix}
\frac{(t-a)^2 - b^2}{(t-a)^2 + b^2} & \frac{2b(t-a)}{(t-a)^2 + b^2} & 0 & 0 & 0 &\\
-\frac{2b(t-a)}{(t-a)^2 + b^2} & \frac{(t-a)^2 - b^2}{(t-a)^2 + b^2} & 0 & 0 & 0 \\
0 & 0 & 1 & 0 & 0 \\
0 & 0 & 0 & 1 & 0 \\
0 & 0 & 0 & 0 & 1 \\
\end{bsmallmatrix} P^{-1},  
P \begin{bsmallmatrix}
\frac{(t-a)^2 + b^2(1-\lambda^2/2)}{(t-a)^2 + b^2}  & \frac{b^2 \lambda \sqrt{1 - \frac{\lambda^2}{4}} h}{(t-a)^2 + b^2}  & \frac{b \lambda (t-a)}{(t-a)^2 + b^2} & \frac{b^2 \lambda \sqrt{1 - \frac{\lambda^2}{4}} g}{(t-a)^2 + b^2} & 0  \\
-\frac{b^2 \lambda \sqrt{1 - \frac{\lambda^2}{4}} h}{(t-a)^2 + b^2}   & \frac{(t-a)^2 + b^2(1-\lambda^2/2)}{(t-a)^2 + b^2}   &  -\frac{b^2\lambda \sqrt{1 - \frac{\lambda^2}{4}} g}{(t-a)^2 + b^2} & \frac{b\lambda (t-a)}{(t-a)^2 + b^2} & 0  \\
- \frac{b \lambda (t-a)}{ (t-a)^2 + b^2 } & \frac{b^2 \lambda \sqrt{1 - \frac{\lambda^2}{4}} g}{(t-a)^2 + b^2} & \frac{(t-a)^2 + b^2(1-\lambda^2/2) }{(t-a)^2 + b^2}  & -\frac{b^2 \lambda \sqrt{1 - \frac{\lambda^2}{4}} h}{(t-a)^2 + b^2}  & 0  \\
-\frac{b^2 \lambda \sqrt{1 - \frac{\lambda^2}{4}} g}{(t-a)^2 + b^2} & -\frac{b \lambda (t-a)}{(t-a)^2 + b^2} & \frac{b^2 \lambda \sqrt{1 - \frac{\lambda^2}{4}} h}{(t-a)^2 + b^2} & \frac{(t-a)^2 + b^2 (1-\lambda^2/2) }{(t-a)^2 + b^2} & 0  \\
0 & 0 & 0 & 0 & 1
\end{bsmallmatrix} P^{-1}, \\
&P \begin{bsmallmatrix}
1 & 0 & 0 &  0 & 0 \\
0 & 1 & 0 & 0 & 0 \\
0 & 0 & 0 &  1 & 0 \\
0 & 0 &  \frac{\sqrt{2}}{2} & 0 & - \frac{\sqrt{2}}{2} \\
0 & 0 & - \frac{\sqrt{2}}{2} & 0 & -  \frac{\sqrt{2}}{2} 
\end{bsmallmatrix}
\begin{bsmallmatrix}
1  & 0 &  \frac{b^2 g}{(t-a)^2 + b^2} & 0 & 0  \\
0 & 1  & \frac{b^2 h}{(t-a)^2 + b^2} & 0 & 0 \\
0 & 0  & 1 & 0 & 0  \\
0 & 0 & \frac{b(t-a)}{(t-a)^2 + b^2} & 1 & 0 \\
\frac{b^2 g}{(t-a)^2 + b^2} & \frac{b^2 h}{(t-a)^2 + b^2} & \frac{b^2}{2((t-a)^2 + b^2 )} & \frac{b(t-a)}{(t-a)^2 + b^2} & 1
\end{bsmallmatrix} 
\begin{bsmallmatrix}
1 & 0 & 0 &  0 & 0 \\
0 & 1 & 0 & 0 & 0 \\
0 & 0 & 0 &   \frac{\sqrt{2}}{2} & - \frac{\sqrt{2}}{2}  \\
0 & 0 &  1 & 0 & 0 \\
0 & 0 & 0 & - \frac{\sqrt{2}}{2} & -  \frac{\sqrt{2}}{2} 
\end{bsmallmatrix}
P^{-1},
\end{align*}
where $(a,b)\in \mathbb{R} \times (\mathbb{R} \setminus \{0\})$,  $\lambda \in (0,2]$,  $(g,h)\in \mathbb{S}^1$ and $P\in \SO^+_{4,1}(\mathbb{R})$.  We remark that  \emph{the conformal Villarceau motion} \cite{Dorst2019conformal, Li2024quadratic} is a product of two curves of the first type and \emph{the circular translation} in $\mathbb{R}^3$ \cite{Li2024quadratic} is a special case of the third type by setting $(g,h) = (0,1)$ (cf.  Example~\ref{ex:O_n1}).  

We notice that on $\SO^+_{4,1}(\mathbb{R})$,  there are (up to a conjugation and a linear change of variable) infinitely many quadratic rational curves.  For comparison,  there are only three (up to a conjugation and a linear change of variable) quadratic rational curves on $\SO^+_{3,1}(\mathbb{R})$.  Thus,  from the perspective of rational curves,  3D kinematics is more complicated than 2D kinematics in conformal geometry.  However,  as we have already seen, 3D kinematics in Euclidean geometry,  2D kinematics in Euclidean geometry and 2D kinematics in Conformal geometry are essentially the same. 
\appendix
\section{Proof of Lemma~\ref{lem:lower triangular}}\label{append:proof of lem:lower triangular}
\begin{proof} 
We prove \eqref{lem:lower triangular:item1}--\eqref{lem:lower triangular:item6} case by case.
\begin{enumerate}[(a)]
\item We observe that $J_m(\lambda) Y + Y J_n(-\lambda) = 0$ is equivalent to $J_m(0) Y + Y J_n(0) = 0$. Thus we may assume $\lambda = 0$. Let $Y = (y_{ij})_{i,j=1}^{m,n}$. Then the equation can be written as
\[
y_{i,j+1} + y_{i-1,j} = 0.
\]
This implies that $Y$ is a lower triangular alternating Toeplitz matrix. 
\item The proof is the same as that of \eqref{lem:lower triangular:item1}.
\item If $\lambda \ne 0$ then the solution $Y = I_m + J_m(\lambda)^2/2$ is unique. If $\lambda = 0$ then we have 
\[
J_m(0) Y + Y J_m(0) = 2J_m(0)  + J_m(0)^3.
\]
We notice that a solution of this equation must have the form 
\[
Y = I_m + J_m(0)^2/2 + T,
\]
where $T$ satisfies $J_m(0) T + T J_m(0) = 0$, which is lower triangular alternating Toeplitz.
\item We write $Y = (Y_{ij})_{i,j=1}^{m\times n}$ where $Y_{ij}\in \mathbb{F}^{2\times 2}$. Then we have 
\[
\varphi(Y_{ij}) = - (Y_{i,j+1} + Y_{i-1,j}),\quad 1 \le i \le m, 1\le j \le n,
\]
where $\varphi:\mathbb{F}^{2\times 2} \to \mathbb{F}^{2\times 2}$ is the map defined by 
\[
\varphi_b(X) =b\left( \begin{bsmallmatrix}
0 & 1\\
-1 & 0
\end{bsmallmatrix} X + X \begin{bsmallmatrix}
0 & 1\\
-1 & 0
\end{bsmallmatrix} \right).
\]
Here we adopt the convention $Y_{ij} = 0$ if either $i < 1$ or $j>n$. We observe that if $b > 0$ then
\[
\varphi_b(\mathbb{F}^{2\times 2}) = \left\lbrace
\begin{bsmallmatrix}
x & y \\
-y & x
\end{bsmallmatrix}\in \mathbb{F}^{2\times 2}: x,y\in \mathbb{F}
\right\rbrace, \quad
\ker(\varphi_b) = \left\lbrace
\begin{bsmallmatrix}
x & y \\
y & -x
\end{bsmallmatrix}\in \mathbb{F}^{2\times 2}: x,y\in \mathbb{F}
\right\rbrace.
\]
If $b = 0$ then
\[
\varphi_b(\mathbb{F}^{2\times 2}) = \{0\}, \quad
\ker(\varphi_b) = \mathbb{F}^{2\times 2}.
\]
This implies $\varphi_b(\mathbb{F}^{n\times n}) \cap \ker(\varphi_b) = \{0\}$. Since $\varphi_b (Y_{1n}) = 0, \varphi_b(Y_{1,n-1}) = -Y_{1n}, \varphi_b(Y_{2n}) = - Y_{1n}$ we have $Y_{1n} = 0$ and $Y_{1,n-1}, Y_{2n}\in \ker(\varphi_b)$. By induction on $n$ and $m$, we may conclude that $Y$ is a block lower triangular alternating Toeplitz matrix.
\item The proof is similar to that of \eqref{lem:lower triangular:item2}.
\item We write $Y = (Y_{ij})_{i,j=1}^{m\times n}$ where $Y_{ij}\in \mathbb{C}^{2\times 2}$. Then we have 
\[
\varphi(Y_{ij}) = Y_{i-1,j} - Y_{i,j-1},\quad 1 \le i \le m, 1\le j \le n,
\]
where $\varphi:\mathbb{C}^{2\times 2} \to \mathbb{C}^{2\times 2}$ is the map defined by 
\[
\varphi(X) =\begin{bsmallmatrix}
a & b\\
-b & a
\end{bsmallmatrix} X + X \begin{bsmallmatrix}
a & b\\
-b & a
\end{bsmallmatrix}.
\]
Here we adopt the convention $Y_{ij} = 0$ if either $i < 1$ or $j>n$. We notice that $\varphi(\mathbb{C}^{2\times 2}) = \mathbb{C}^{2\times 2}, \ker(\varphi) = \{0\}$. This implies $\varphi(\mathbb{C}^{2\times 2}) \cap \ker(\varphi) = \{0\}$ and the rest of the proof is the same as that of \eqref{lem:lower triangular:item3}. 
\item We write $Y = (Y_{ij})_{i,j=1}^{m\times n}$ where $Y_{ij}\in \mathbb{H}^{2\times 2}$. Then we have 
\[
\varphi_{\lambda}(Y_{ij}) = Y_{i-1,j} - Y_{i,j-1},\quad 1 \le i \le m, 1\le j \le n,
\]
where $\varphi_{\lambda}:\mathbb{H}^{2\times 2} \to \mathbb{H}^{2\times 2}$ is the map defined by 
\[
\varphi_{\lambda}(X_1 + jX_2) = \left( \begin{bsmallmatrix}
\lambda & 0\\
0 & \lambda^\ast
\end{bsmallmatrix} X_1 - X_1\begin{bsmallmatrix}
\lambda^\ast & 0\\
0 & \lambda
\end{bsmallmatrix} \right) + \qj \left( \begin{bsmallmatrix}
\lambda^\ast & 0\\
0 & \lambda
\end{bsmallmatrix} X_2 - X_2\begin{bsmallmatrix}
\lambda^\ast & 0\\
0 & \lambda
\end{bsmallmatrix} \right). 
\]
Here $X_1,X_2\in \mathbb{C}^{2\times 2}$ and we adopt the convention $Y_{ij} = 0$ if either $i < 1$ or $j>n$. We notice that if $\operatorname{Im}(\lambda) > 0$
\begin{align*}
\varphi_{\lambda}(\mathbb{H}^{2\times 2}) &= \left\lbrace
\begin{bsmallmatrix}
x & 0 \\
0 & y
\end{bsmallmatrix} + \qj \begin{bsmallmatrix}
0 & z \\
w & 0
\end{bsmallmatrix}  \in \mathbb{H}^{2\times 2} : x,y,z,w \in \mathbb{C}
\right\rbrace, \\
\ker(\varphi_{\lambda}) &= \left\lbrace
\begin{bsmallmatrix}
0 & z \\
w & 0
\end{bsmallmatrix} + \qj \begin{bsmallmatrix}
x & 0 \\
0 & y
\end{bsmallmatrix}  \in \mathbb{H}^{2\times 2} : x,y,z,w \in \mathbb{C}
\right\rbrace.
\end{align*}
If $\operatorname{Im}(\lambda) = 0$ then $\varphi(\mathbb{H}^{2 \times 2}) = \{0\}$ and $\ker(\varphi) = \mathbb{H}^{2\times 2}$. This implies $\varphi(\mathbb{F}^{2 \times 2}) \cap \ker(\varphi) = \{0\}$ and the rest of the proof is the same as that of \eqref{lem:lower triangular:item3}. 
\end{enumerate}
\end{proof}

\section{Proof of Lemma~\ref{lem:Yij}}\label{append:proof lem:Yij}
\begin{proof}
We recall that \eqref{lem:block:eq1} is
\begin{align}
X_i Y_{ij} + Y_{ij} X_j &= 0, \quad i \ne j\label{lem:Yij:eq1} \\
X_i Y_{ii} + Y_{ii} X_i &= 2X_i + X_i^3. \label{lem:Yij:eq2}
\end{align}
We notice that a solution of \eqref{lem:Yij:eq2} is $Y_{ii} = I_{m_i} + X_i^2/2 + T_i$ where $T_i$ satisfies 
\begin{equation}\label{lem:Yij:eq3}
X_i T_i + T_i X_i = 0.
\end{equation}
It is straightforward to verify that $B_i^{-1} = \varepsilon B_i$ for each $1  \le i \le s$, thus \eqref{lem:block:eq2} implies 
\begin{align*}
i \ne j: Y_{ji}^\sigma &= -\varepsilon B_iY_{ij} B_j,\\
i = j: Y_{ii} &=\varepsilon ( B_i \left( (2 I_{m_i} + X_i^2) -  Y_{ii} \right) B_i)^\sigma.
\end{align*}
A direct calculation implies $\varepsilon (B_i X_i^2 B_i)^\sigma = X_i^2$. Thus, if $Y_{ii} = I_{m_i} + X_i^2/2 + T_i$ for some $T_i$ then 
\begin{equation}\label{lem:Yij:eq0} 
T_i =- \varepsilon (B_i T_i B_i)^\sigma.
\end{equation}
We also observe that $B_i^\sigma = \varepsilon B_i$.

Thus, to solve \eqref{lem:block:eq1} and \eqref{lem:block:eq2}, it is sufficient to consider the following system:
\begin{align}
X_1 Y_{12} + Y_{12} X_2 &= 0, \label{lem:Yij:eq4} \\
Y_{21} + \varepsilon (B_1 Y
_{12} B_2)^\sigma &=0. \label{lem:Yij:eq5}
\end{align}
where $(X_1,B_1), (X_2,B_2)$ are normal forms listed in Table~\ref{Tab:indecomposable} such that $0\in \rho(X_1) + \rho(X_2)$. We split the discussion with respect to the seven cases in Table~\ref{Tab:indecomposable}. 
\begin{enumerate}[No. 1:]
\item\label{lem:Yij:No1} $\varepsilon = 1$ and $\sigma$ is the transpose.
\begin{enumerate}[(a)]
\item $(X_1,X_2) =(J_{2m+1}(0),  J_{2n+1}(0))$. We have $(B_1, B_2) =(F_{2m+1}, F_{2n+1})$ and \eqref{lem:Yij:eq4} becomes
\[
J_{2m+1}(0) Y_{12} + Y_{12} J_{2n+1}(0) = 0.
\]
Thus $Y_{12}\in \mathbb{C}^{(2m+1) \times (2n+1)}$ is  lower triangular alternating Toeplitz by Lemma~\ref{lem:lower triangular} \eqref{lem:lower triangular:item1}. It has one of the following two forms depending on $m \ge n$ or $m < n$:
\[
\begin{bsmallmatrix}
0\\
\ch{_{$\fp$}S}(z_1,\dots, z_{2m+1})
\end{bsmallmatrix}~\quad~\text{or}~\quad~ 
\begin{bsmallmatrix}
\ch{_{$\fp$}S}(z_1,\dots, z_{2n+1}) & 0
\end{bsmallmatrix}
\]
and $Y_{21}$ is 
\[
-\begin{bsmallmatrix}
\ch{_{$\fp$}S}(z_1,\dots, z_{2m+1}) & 0
\end{bsmallmatrix}~\quad~\text{or}~\quad~ 
-\begin{bsmallmatrix}
0\\
\ch{_{$\fp$}S}(z_1,\dots, z_{2n+1})
\end{bsmallmatrix}.
\]
In particular, $Y_{12} = Y_{21}$ implies $Y_{12} = Y_{21} = 0$.
\item $(X_1, X_2)=( J_{2m+1}(0), \diag(J_{2n}(0), -J_{2n}(0)^\tp ) )$. We have $(B_1, B_2) =(F_{2m+1}, I_{2n} \otimes H_2)$ and \eqref{lem:Yij:eq4} becomes
\[
J_{2m+1}(0) Y_{12} + Y_{12} \diag(J_{2n}(0), -J_{2n}(0)^\tp) = 0.
\]
We partition $Y_{12} \in \mathbb{C}^{(2m+1) \times 4n}$ as $Y_{12} = \begin{bmatrix}
Z & W
\end{bmatrix}$ where $Z,W\in \mathbb{C}^{(2m+1) \times 2n}$ to obtain 
\[
J_{2m+1}(0) Z + Z J_{2n}(0) = 0,\quad J_{2m+1}(0) W -  W J_{2n}(0)^\tp = 0.
\]
Therefore, $Z$ (resp. $WH_{2n}$) is lower triangular alternating Toeplitz (resp. lower triangular Toeplitz) by Lemma~\ref{lem:lower triangular} \eqref{lem:lower triangular:item1}. This implies that $Y_{12}$ has one of the following two forms, depending on $m\ge n$ or $m < n$:
\[
\begin{bsmallmatrix}
0 & 0 \\
\ch{_{$\fp$} S(z_1,\dots, z_{2n})} & T_{\fp}(w_1,\dots, w_{2n})
\end{bsmallmatrix}~\quad~\text{or}~\quad~ 
\begin{bsmallmatrix}
\ch{_{$\fp$} S(z_1,\dots, z_{2m+1})} & 0 & 0 & T_{\fp}(w_1,\dots, w_{2m+1})
\end{bsmallmatrix},
\]
and $Y_{21}$ is 
\begin{equation*}
-\begin{bsmallmatrix}
\ch{ _{$\fp$} S} (w_1,\dots, w_{2n}) & 0\\ 
\ch{ ^{$\fp$} T} (z_{1},-z_{2},\dots,z_{2n-1}, -z_{2n}) & 0 \\
\end{bsmallmatrix}~\quad~\text{or}~\quad~
-\begin{bsmallmatrix}
0 \\
\ch{ _{$\fp$} S} (w_1,\dots, w_{2m+1})\\ 
\ch{ ^{$\fp$} T} (z_{1},-z_{2}, \dots, - z_{2m}, z_{2m+1}) \\
0
\end{bsmallmatrix}.
\end{equation*}

\item $(X_1, X_2)=(\diag(J_{m}(\lambda), -J_{m}(\lambda)^\tp ), \diag(J_{n}(\mu), -J_{n}(\mu)^\tp )) )$ where $\lambda^2 = \mu^2$. In this case we have $(B_1, B_2)  =(I_m \otimes H_2, I_n \otimes H_2)$ and \eqref{lem:Yij:eq4} becomes
\[
\diag(J_{m}(\lambda), -J_{m}(\lambda)^\tp )  Y_{12} + Y_{12} \diag(J_{n}(\mu), -J_{n}(\mu)^\tp ) = 0.
\]
We partition $Y_{12} \in \mathbb{C}^{2m \times 2n}$ as $Y_{12} = \begin{bmatrix}
Z & W \\
U & V
\end{bmatrix}$ where $Z, W, U, V\in \mathbb{C}^{m\times n}$ to obtain 
\begin{align*}
J_{m}(\lambda) Z + Z J_{n}(\mu) &= 0,\quad \ J_{m}(\lambda) W - W J_{n}(\mu)^\tp = 0, \\
-J_{m}(\lambda)^\tp U + U J_{n}(\mu) &= 0,\quad -J_{m}(\lambda)^\tp V - V J_{n}(\mu)^\tp = 0.
\end{align*}

If $\lambda = \mu \ne 0$ then $Z =  V = 0$ and $H_m U, WH_n$ are lower triangular Toeplitz. This implies that $Y_{12}$ has one of the following two forms, depending on $m\ge n$ or $m < n$:
\begin{equation*}
\begin{bsmallmatrix}
0 & 0 \\
0 & T_{\fp} (w_1,\dots, w_{2n}) \\
\ch{^{$\fp$} T}(u_1,\dots, u_{2n}) & 0\\
0 & 0
\end{bsmallmatrix}
~\quad~\text{or}~\quad~
\begin{bsmallmatrix}
0 & 0 & 0 & T_{\fp} (w_1,\dots, w_{2m}) \\
\ch{^{$\fp$} T}(u_1,\dots, u_{2m}) & 0 & 0 & 0
\end{bsmallmatrix}
\end{equation*}
and $Y_{21}$ is 
\begin{equation*}
-\begin{bsmallmatrix}
0 & 0 & 0 & T_{\fp} (w_1,\dots, w_{2n}) \\
\ch{^{$\fp$} T}(u_1,\dots, u_{2n}) & 0 & 0 & 0
\end{bsmallmatrix}
~\quad~\text{or}~\quad~
-\begin{bsmallmatrix}
0 & 0 \\
0 & T_{\fp} (w_1,\dots, w_{2m}) \\
\ch{^{$\fp$} T}(u_1,\dots, u_{2m}) & 0\\
0 & 0
\end{bsmallmatrix}.
\end{equation*}
In particular, if $Y_{12} = Y_{21}$ then $Y_{12} = Y_{21} = 0$.

If $\lambda = -\mu \ne 0$ then $W = U = 0$ and $Z, H_mV H_n$ are lower triangular alternating Toeplitz. Therefore $Y_{12}$ has one of the following two forms, depending on $m\ge n$ or $m < n$:
\begin{equation*}
\begin{bsmallmatrix}
0 & 0 \\
\ch{_{$\fp$} S}(z_1,\dots, z_{2n}) & 0 \\
0 & S^{\fp} (v_1,\dots, v_{2n})\\
0 & 0
\end{bsmallmatrix}
~\quad~\text{or}~\quad~
\begin{bsmallmatrix}
\ch{_{$\fp$} S}(z_1,\dots, z_{2m})  & 0 & 0 & 0\\
0& 0 & 0 & S^{\fp} (v_1,\dots, v_{2m}) 
\end{bsmallmatrix}
\]
and $Y_{21}$ is 
\begin{align*}
&-\begin{bsmallmatrix} 
\ch{_{$\fp$} S}(v_1, -v_2,\dots, v_{2n-1}, -v_{2n})  & 0 & 0 & 0\\
0& 0 & 0 & S^{\fp} (z_1, -z_2,\dots, z_{2n-1}, -z_{2n}) 
\end{bsmallmatrix}
~\quad~\text{or}~\quad~ \\
&-\begin{bsmallmatrix}
0 & 0 \\
\ch{_{$\fp$} S}(v_1, -v_2,\dots, v_{2m-1}, -v_{2m}) & 0 \\
0 & S^{\fp} (z_1, -z_2,\dots, z_{2m-1}, -z_{2m})\\
0 & 0
\end{bsmallmatrix}.
\end{align*}

If $\lambda = \mu = 0$ then $m, n$ are even, $Z, H_mV H_n$ are lower triangular alternating Toeplitz and $H_m U, WH_n$ are lower triangular Toeplitz. Thus, $Y$ has one of the following two forms, depending on $m\ge n$ or $m < n$:
\[
\begin{bsmallmatrix}
0& 0 \\
\ch{_{$\fp$} S}(z_1,\dots, z_{n})   & T_{\fp}(w_1,\dots, w_{n})\\
\ch{^{$\fp$} T}(u_1,\dots, u_{n}) & S^{\fp} (v_1,\dots, v_{n})  \\
0& 0
\end{bsmallmatrix}
~\quad~\text{or}~\quad~
\begin{bsmallmatrix}
\ch{_{$\fp$} S}(z_1,\dots, z_{m}) & 0 & 0   & T_{\fp}(w_1,\dots, w_{m})\\
\ch{^{$\fp$} T}(u_1,\dots, u_{m}) & 0 & 0 & S^{\fp} (v_1,\dots, v_{m})
\end{bsmallmatrix}.
\]
Thus $Y_{21}$ is 
\begin{align*}
&-\begin{bsmallmatrix}
\ch{_{$\fp$} S}(v_1,-v_2,\dots, v_{n-1},  -v_{n}) & 0 & 0   & T_{\fp}(w_1,\dots, w_{n})\\
\ch{^{$\fp$} T}(u_1,\dots, u_{n}) & 0 & 0 & S^{\fp} (z_1,-z_2,\dots, z_{n-1},  -z_{n})
\end{bsmallmatrix}~\quad~\text{or}~\quad~\\
&-\begin{bsmallmatrix}
0& 0 \\
\ch{_{$\fp$} S}(v_1, -v_2,\dots, v_{n-1},-v_{n})   & T_{\fp}(w_1,\dots, w_{n})\\
\ch{^{$\fp$} T}(u_1,\dots, u_{n}) & S^{\fp} (z_1, -z_2,\dots, z_{n-1}, -z_{n})  \\
0& 0
\end{bsmallmatrix}.
\end{align*}
In particular, $Y_{12} = Y_{21}$ implies 
\[
Y_{12} = Y_{21} = \begin{bsmallmatrix}
\ch{_{$\fp$} S}(z_1, \dots, z_{n}) & 0 \\
0 & S^{\fp} (-z_1,z_2,\dots, -z_{n - 1},  z_{n})
\end{bsmallmatrix}.
\]
\end{enumerate}
\item\label{lem:Yij:No2} $\varepsilon = -1$ and $\sigma$ is the transpose. 
\begin{enumerate}[(a)]
\item $(X_1,X_2) =(J_{2m}(0),  J_{2n}(0))$. We have $(B_1, B_2) =(F_{2m}, F_{2n})$ and \eqref{lem:Yij:eq4} becomes
\[
J_{2m}(0) Y_{12} + Y_{12} J_{2n}(0) = 0.
\]
Thus $Y_{12}\in \mathbb{C}^{2m \times 2n}$ is  lower triangular alternating Toeplitz by Lemma~\ref{lem:lower triangular} \eqref{lem:lower triangular:item1}. It has one of the following two forms depending on $m \ge n$ or $m < n$:
\[
\begin{bsmallmatrix}
0\\
\ch{_{$\fp$}S}(z_1,\dots, z_{2n})
\end{bsmallmatrix}~\quad~\text{or}~\quad~ 
\begin{bsmallmatrix}
\ch{_{$\fp$}S}(z_1,\dots, z_{2m}) & 0
\end{bsmallmatrix}
\]
and $Y_{21}$ is 
\[
-\begin{bsmallmatrix}
\ch{_{$\fp$}S}(z_1,\dots, z_{2n}) & 0
\end{bsmallmatrix}~\quad~\text{or}~\quad~ 
-\begin{bsmallmatrix}
0\\
\ch{_{$\fp$}S}(z_1,\dots, z_{2m})
\end{bsmallmatrix}.
\]
In particular, $Y_{12} = Y_{21}$ implies $Y_{12} = Y_{21} = 0$.
\item $(X_1, X_2)=( J_{2m}(0), \diag(J_{2n+1}(0), -J_{2n+1}(0)^\tp ) )$. We have $(B_1, B_2) =(F_{2m}, I_{2n+1} \otimes F_2)$ and \eqref{lem:Yij:eq4} becomes
\[
J_{2m}(0) Y_{12} + Y_{12} \diag(J_{2n+1}(0), -J_{2n+1}(0)^\tp) = 0.
\]
We partition $Y_{12} \in \mathbb{C}^{2m \times (4n+2)}$ as $Y_{12} = \begin{bmatrix}
Z & W
\end{bmatrix}$ where $Z,W\in \mathbb{C}^{2m \times (4n+2)}$ to obtain 
\[
J_{2m}(0) Z + Z J_{2n+1}(0) = 0,\quad J_{2m}(0) W -  W J_{2n+1}(0)^\tp = 0.
\]
Therefore, $Z$ (resp. $WH_{2n}$) is lower triangular alternating Toeplitz (resp. lower triangular Toeplitz) by Lemma~\ref{lem:lower triangular} \eqref{lem:lower triangular:item1}. This implies that $Y_{12}$ has one of the following two forms, depending on $m\ge n+1$ or $m \le n$:
\[
\begin{bsmallmatrix}
0 & 0 \\
\ch{_{$\fp$} S(z_1,\dots, z_{2n+1})} & T_{\fp}(w_1,\dots, w_{2n+1})
\end{bsmallmatrix}~\quad~\text{or}~\quad~ 
\begin{bsmallmatrix}
\ch{_{$\fp$} S(z_1,\dots, z_{2m})} & 0 & 0 & T_{\fp}(w_1,\dots, w_{2m})
\end{bsmallmatrix},
\]
and $Y_{21}$ is 
\begin{equation*}
\begin{bsmallmatrix}
\ch{ _{$\fp$} S} (w_1,\dots, w_{2n+1}) & 0\\ 
\ch{ ^{$\fp$} T} (-z_{1},z_2,\dots,z_{2n}, -z_{2n+1}) & 0 \\
\end{bsmallmatrix}~\quad~\text{or}~\quad~
\begin{bsmallmatrix}
0 \\
\ch{ _{$\fp$} S} (w_1,\dots, w_{2m})\\ 
\ch{ ^{$\fp$} T} (-z_{1}, z_{2}, \dots,-z_{2m-1}, z_{2m}) \\
0
\end{bsmallmatrix}.
\end{equation*}
\item $(X_1, X_2)=(\diag(J_{m}(\lambda), -J_{m}(\lambda)^\tp ), \diag(J_{n}(\mu), -J_{n}(\mu)^\tp )) )$ where $\lambda^2 = \mu^2$. In this case we have $(B_1, B_2)  =(I_m \otimes F_2, I_n \otimes F_2)$ and \eqref{lem:Yij:eq4} becomes
\[
\diag(J_{m}(\lambda), -J_{m}(\lambda)^\tp )  Y_{12} + Y_{12} \diag(J_{n}(\mu), -J_{n}(\mu)^\tp ) = 0.
\]
We partition $Y_{12} \in \mathbb{C}^{2m \times 2n}$ as $Y_{12} = \begin{bmatrix}
Z & W \\
U & V
\end{bmatrix}$ where $Z, W, U, V\in \mathbb{C}^{m\times n}$ to obtain 
\begin{align*}
J_{m}(\lambda) Z + Z J_{n}(\mu) &= 0,\quad \ J_{m}(\lambda) W - W J_{n}(\mu)^\tp = 0, \\
-J_{m}(\lambda)^\tp U + U J_{n}(\mu) &= 0,\quad -J_{m}(\lambda)^\tp V - V J_{n}(\mu)^\tp = 0.
\end{align*}

If $\lambda = \mu \ne 0$ then $Z =  V = 0$ and $H_m U, WH_n$ are lower triangular Toeplitz. This implies that $Y_{12}$ has one of the following two forms, depending on $m\ge n$ or $m < n$:
\begin{equation*}
\begin{bsmallmatrix}
0 & 0 \\
0 & T_{\fp} (w_1,\dots, w_{n}) \\
\ch{^{$\fp$} T}(u_1,\dots, u_{n}) & 0\\
0 & 0
\end{bsmallmatrix}
~\quad~\text{or}~\quad~
\begin{bsmallmatrix}
0 & 0 & 0 & T_{\fp} (w_1,\dots, w_{m}) \\
\ch{^{$\fp$} T}(u_1,\dots, u_{m}) & 0 & 0 & 0
\end{bsmallmatrix}
\end{equation*}
and $Y_{21}$ is 
\begin{equation*}
-\begin{bsmallmatrix}
0 & 0 & 0 & T_{\fp} (w_1,\dots, w_{n}) \\
\ch{^{$\fp$} T}(u_1,\dots, u_{n}) & 0 & 0 & 0
\end{bsmallmatrix}
~\quad~\text{or}~\quad~
-\begin{bsmallmatrix}
0 & 0 \\
0 & T_{\fp} (w_1,\dots, w_{m}) \\
\ch{^{$\fp$} T}(u_1,\dots, u_{m}) & 0\\
0 & 0
\end{bsmallmatrix}.
\end{equation*}
In particular, if $Y_{12} = Y_{21}$ then $Y_{12} = Y_{21} = 0$.

If $\lambda = -\mu \ne 0$ then $W = U = 0$ and $Z, H_mV H_n$ are lower triangular alternating Toeplitz. Therefore $Y_{12}$ has one of the following two forms, depending on $m\ge n$ or $m < n$:
\begin{equation*}
\begin{bsmallmatrix}
0 & 0 \\
\ch{_{$\fp$} S}(z_1,\dots, z_{n}) & 0 \\
0 & S^{\fp} (v_1,\dots, v_n)\\
0 & 0
\end{bsmallmatrix}
~\quad~\text{or}~\quad~
\begin{bsmallmatrix}
\ch{_{$\fp$} S}(z_1,\dots, z_{m})  & 0 & 0 & 0\\
0& 0 & 0 & S^{\fp} (v_1,\dots, v_{m}) 
\end{bsmallmatrix}
\]
and $Y_{21}$ is 
\begin{align*}
&\begin{bsmallmatrix} 
\ch{_{$\fp$} S}(v_1, -v_2,\dots, (-1)^{n-2}v_{n-1}, (-1)^{n-1}v_{n})  & 0 & 0 & 0\\
0& 0 & 0 & S^{\fp} (z_1, -z_2,\dots, (-1)^{m-2}z{m-1}, (-1)^{m-1}z_{m}) 
\end{bsmallmatrix}
~\quad~\text{or}~\quad~ \\
&
\begin{bsmallmatrix}
0 & 0 \\
\ch{_{$\fp$} S}(v_1, -v_2,\dots, (-1)^{m-2}v_{m-1}, (-1)^{m-1}v_{m}) & 0 \\
0 & S^{\fp} (z_1, -z_2,\dots, (-1)^{m-2}z_{m-1}, (-1)^{m-1}z_{m})\\
0 & 0
\end{bsmallmatrix}.
\end{align*}
If $\lambda = \mu = 0$ then $m, n$ are odd, $Z, H_mV H_n$ are lower triangular alternating Toeplitz and $H_m U, WH_n$ are lower triangular Toeplitz. Thus, $Y$ has one of the following two forms, depending on $m\ge n$ or $m < n$:
\[
\begin{bsmallmatrix}
0& 0 \\
\ch{_{$\fp$} S}(z_1,\dots, z_{n})   & T_{\fp}(w_1,\dots, w_{n})\\
\ch{^{$\fp$} T}(u_1,\dots, u_{n}) & S^{\fp} (v_1,\dots, v_{n})  \\
0& 0
\end{bsmallmatrix}
~\quad~\text{or}~\quad~
\begin{bsmallmatrix}
\ch{_{$\fp$} S}(z_1,\dots, z_{m}) & 0 & 0   & T_{\fp}(w_1,\dots, w_{m})\\
\ch{^{$\fp$} T}(u_1,\dots, u_{m}) & 0 & 0 & S^{\fp} (v_1,\dots, v_{m})
\end{bsmallmatrix}.
\]
Thus $Y_{21}$ is 
\begin{align*}
&\begin{bsmallmatrix}
\ch{_{$\fp$} S}(v_1,-v_2,\dots, -v_{n-1},  v_{n}) & 0 & 0   & -T_{\fp}(w_1,\dots, w_{n})\\
-\ch{^{$\fp$} T}(u_1,\dots, u_{n}) & 0 & 0 & S^{\fp} (z_1,-z_2,\dots,  -z_{n-1}, z_{n})
\end{bsmallmatrix}~\quad~\text{or}~\quad~\\
&\begin{bsmallmatrix}
0& 0 \\
\ch{_{$\fp$} S}(v_1, -v_2,\dots,-v_{m-1}, v_{m})   & -T_{\fp}(w_1,\dots, w_{m})\\
-\ch{^{$\fp$} T}(u_1,\dots, u_{m}) & S^{\fp} (z_1, -z_2,\dots,-z_{m-1},  z_{m})  \\
0& 0
\end{bsmallmatrix}.
\end{align*}
In particular, $Y_{12} = Y_{21}$ implies 
\[
Y_{12} = Y_{21} = \begin{bsmallmatrix}
\ch{_{$\fp$} S}(z_1,\dots, z_{n}) & 0 \\
0 & S^{\fp} (z_1,-z_2,\dots, -z_{n - 1},  z_n)
\end{bsmallmatrix}.
\]
\end{enumerate}
\item\label{lem:Yij:No3} $\varepsilon = 1$ and $\sigma$ is the conjugate transpose. Observing that 
\[
\sigma \left( J_{m}(\lambda) \right) \cap \left( -\sigma\left( \diag(J_{m}(\mu), -J_{m}(\mu)^\ast ) \right) \right) = \emptyset,
\]
since $\operatorname{Re}(\lambda) = 0$ and $\operatorname{Re}(\mu) > 0$, we only need to consider two sub-cases. 
\begin{enumerate}[(a)]
\item $(X_1,X_2) = (J_{m}(\lambda), J_{n}(-\lambda))$ where $\operatorname{Re}(\lambda) = 0$. We have 
\[
(B_1,B_2) = (\kappa \ci^{m-1} F_{m}, \kappa \ci^{n-1} F_{n})
\] 
and \eqref{lem:Yij:eq4} becomes $J_{m}(\lambda) Y_{12} + Y_{12} J_{n}(-\lambda) = 0$. Therefore $Y_{12}$ is lower triangular Toeplitz by Lemma~\ref{lem:lower triangular} \eqref{lem:lower triangular:item1}. It has one of the following two forms depending on $m \ge n$ or $m < n$:
\[
\begin{bsmallmatrix}
0\\
\ch{_{$\fp$}S}(z_1,\dots, z_{n})
\end{bsmallmatrix}~\quad~\text{or}~\quad~ 
\begin{bsmallmatrix}
\ch{_{$\fp$}S}(z_1,\dots, z_{m}) & 0
\end{bsmallmatrix}
\]
and $Y_{21}$ is 
\[
 \ci^{m+n} (-1)^{n-1} \begin{bsmallmatrix}
\ch{_{$\fp$}S}(\overline{z}_1,\dots,\overline{z}_n) & 0
\end{bsmallmatrix}~\quad~\text{or}~\quad~ 
 \ci^{m+n} (-1)^{n-1}\begin{bsmallmatrix}
0\\
 \ch{_{$\fp$}S}(\overline{z}_1,\dots, \overline{z}_m)
\end{bsmallmatrix}.
\]
In particular, $Y_{12} = Y_{21}$ implies $Y_{12} = \ch{_{$\fp$}S}(z_1,\dots, z_{n})$ where $\operatorname{Re} (z_j) = 0$ for each $1 \le j \le n$.
\item $(X_1,X_2) = ( \diag(J_{m}(\lambda), -J_{m}(\lambda)^\ast ),\diag(J_{n}(\mu), -J_{n}(\mu)^\ast ))$. Here we must have $\lambda = \overline{\mu}$ and $\operatorname{Re}(\lambda) = \operatorname{Re}(\mu) > 0$. In this case we have $(B_1,B_2) = (I_m \otimes H_2, I_n \otimes H_2)$ and \eqref{lem:Yij:eq4} becomes
\[
\diag(J_{m}(\lambda), -J_{m}(\lambda)^\ast )  Y + Y \diag(J_{n}(\overline{\lambda}), -J_{n}(\overline{\lambda})^\ast ) = 0.
\]
We partition $Y\in \mathbb{C}^{2m \times 2n}$ as $Y = \begin{bmatrix}
Z & W \\
U & V
\end{bmatrix}$ where $Z, W, U, V\in \mathbb{C}^{m\times n}$ to obtain 
\begin{align*}
J_{m}(\lambda) Z + Z J_{n}(\overline{\lambda}) &= 0,\quad \ J_{m}(\lambda) W - W J_{n}(\lambda)^\tp = 0, \\
J_{m}(\overline{\lambda})^\tp U - U J_{n}(\overline{\lambda}) &= 0,\quad J_{m}(\overline{\lambda})^\tp V + V J_{n}(\lambda)^\tp = 0.
\end{align*} 
Since $\operatorname{Re}(\lambda) > 0$, we must have $Z = V = 0$ and $H_m U, WH_n$ are lower triangular Toeplitz matrices. This implies that $Y_{12}$ has one of the following two forms, depending on $m\ge n$ or $m < n$:
\[
\begin{bsmallmatrix}
0& 0 \\
0   & T_{\fp}(w_1,\dots, w_{n})\\
\ch{^{$\fp$} T}(u_1,\dots, u_{n}) & 0 \\
0& 0
\end{bsmallmatrix}
~\quad~\text{or}~\quad~
\begin{bsmallmatrix}
0& 0 & 0   & T_{\fp}(w_1,\dots, w_{m})\\
\ch{^{$\fp$} T}(u_1,\dots, u_{m}) & 0 & 0 & 0
\end{bsmallmatrix}
\]
and $Y_{21}$ is 
\[
\begin{bsmallmatrix}
0& 0 & 0   & T_{\fp}(\overline{w}_1,\dots,\overline{w}_{m})\\
\ch{^{$\fp$} T}(\overline{u}_1,\dots, \overline{u}_{m}) & 0 & 0 & 0
\end{bsmallmatrix}
~\quad~\text{or}~\quad~
\begin{bsmallmatrix}
0& 0 \\
0   & T_{\fp}(\overline{w}_1,\dots,\overline{w}_{n})\\
\ch{^{$\fp$} T}(\overline{u}_1,\dots, \overline{u}_{n}) & 0 \\
0& 0
\end{bsmallmatrix}.
\]
If $Y_{12} = Y_{21}$ then $m = n$ and $\operatorname{Im}(\lambda) = 0$. This implies that 
\[
Y_{12} = Y_{21} =\begin{bsmallmatrix}
0   & T_{\fp}(w_1,\dots, w_{n})\\
\ch{^{$\fp$} T}(u_1,\dots, u_{n}) & 0 
\end{bsmallmatrix}
\]
where $u_1,\dots, u_n, w_1,\dots, w_n$ are all real numbers.
\end{enumerate} 
\item\label{lem:Yij:No4}  $\varepsilon = 1$ and $\sigma$ is the transpose. First we observe that for the four types of normal forms, we have 
\begin{align*}
0&\not\in \sigma \left( J_{2m+1}(0) \right) + \sigma\left( J_n\left(\begin{bsmallmatrix}
0 & b \\
-b & 0
\end{bsmallmatrix}\right)
\right),\\
0&\not\in \sigma \left( J_{2m+1}(0) \right)  +\sigma\left( \diag
\left(
J_n\left(\begin{bsmallmatrix}
a & b \\
-b & a
\end{bsmallmatrix}\right),
-J_n\left(\begin{bsmallmatrix}
a & b \\
-b & a
\end{bsmallmatrix}\right)^\tp
\right)
\right), \\
0&\not\in \sigma \left( 
\diag\left( 
J_{m}(\lambda), -J_{m}(\lambda)^\tp  
\right)
\right) + \sigma\left( J_n\left(\begin{bsmallmatrix}
0 & b \\
-b & 0
\end{bsmallmatrix}\right)
\right),\\
0&\not\in \sigma \left( 
\diag\left( 
J_{m}(\lambda), -J_{m}(\lambda)^\tp  
\right)
\right)  +\sigma\left( \diag
\left(
J_n\left(\begin{bsmallmatrix}
a & b \\
-b & a
\end{bsmallmatrix}\right),
-J_n\left(\begin{bsmallmatrix}
a & b \\
-b & a
\end{bsmallmatrix}\right)^\tp
\right)
\right), \\
0&\not\in
\sigma\left( J_m\left(\begin{bsmallmatrix}
0 & c \\
-c & 0
\end{bsmallmatrix}\right)
\right) + 
 \sigma\left( \diag
\left(
J_n\left(\begin{bsmallmatrix}
a & b \\
-b & a
\end{bsmallmatrix}\right),
-J_m\left(\begin{bsmallmatrix}
a & b \\
-b & a
\end{bsmallmatrix}\right)^\tp
\right)
\right),
\end{align*}
where $a,b,c >0$ and $\lambda \ge 0$. Hence we only need to consider five sub-cases.
\begin{enumerate}[(a)]
\item $(X_1,X_2) = (J_{2m+1}(0), J_{2n+1}(0) )$. We have 
\[
(B_1,B_2) = (\kappa (-1)^{m} F_{2m+1}, \kappa (-1)^{n} F_{2n+1})
\] 
and \eqref{lem:Yij:eq4} becomes $J_{2m+1}(0) Y_{12} + Y_{12} J_{2n+1}(0) = 0$. Therefore $Y_{12}$ is lower triangular Toeplitz by Lemma~\ref{lem:lower triangular} \eqref{lem:lower triangular:item1}. It has one of the following two forms depending on $m \ge n$ or $m < n$:
\[
\begin{bsmallmatrix}
0\\
\ch{_{$\fp$}S}(z_1,\dots, z_{2n+1})
\end{bsmallmatrix}~\quad~\text{or}~\quad~ 
\begin{bsmallmatrix}
\ch{_{$\fp$}S}(z_1,\dots, z_{2m+1}) & 0
\end{bsmallmatrix}
\]
and $Y_{21}$ is 
\[
(-1)^{m+n+1} \begin{bsmallmatrix}
\ch{_{$\fp$}S}(z_1,\dots,z_{2n+1}) & 0
\end{bsmallmatrix}~\quad~\text{or}~\quad~ 
 (-1)^{m+n+1}\begin{bsmallmatrix}
0\\
 \ch{_{$\fp$}S}(z_1,\dots, z_{2m+1})
\end{bsmallmatrix}.
\]
Thus $Y_{12} = Y_{21}$ implies $Y_{12} = Y_{21} = 0$ otherwise.
\item$(X_1,X_2) = (J_{2m+1}(0), \diag(J_{2n}(0), -J_{2n}(0)^\tp ) )$. \[
(B_1,B_2) = (\kappa (-1)^{m} F_{2m+1}, I_{2n}\otimes H_2),\quad Y = \begin{bmatrix}
Z & W
\end{bmatrix}
\] 
where $Z, W\in \mathbb{R}^{(2m+1) \times 2n}$ satisfy
\[
J_{2m+1}(0) Z + Z J_{2n}(0) = 0,\quad J_{2m+1}(0) W -  W J_{2n}(0)^\tp = 0.
\]
Therefore, $Z$ (resp. $WH_{2n}$) is lower triangular alternating Toeplitz (resp. lower triangular Toeplitz) by Lemma~\ref{lem:lower triangular} \eqref{lem:lower triangular:item1}. This implies that $Y_{12}$ has one of the following two forms, depending on $m\ge n$ or $m < n$:
\[
\begin{bsmallmatrix}
0 & 0 \\
\ch{_{$\fp$} S(z_1,\dots, z_{2n})} & T_{\fp}(w_1,\dots, w_{2n})
\end{bsmallmatrix}~\quad~\text{or}~\quad~ 
\begin{bsmallmatrix}
\ch{_{$\fp$} S(z_1,\dots, z_{2m+1})} & 0 & 0 & T_{\fp}(w_1,\dots, w_{2m+1})
\end{bsmallmatrix},
\]
and $Y_{21}$ is 
\begin{equation*}
\kappa (-1)^{m+1} \begin{bsmallmatrix}
\ch{ _{$\fp$} S} (w_1, \dots, w_{2n}) & 0\\ 
\ch{ ^{$\fp$} T} (z_{1}, -z_2,\dots,z_{2n-1}, -z_{2n}) & 0 \\
\end{bsmallmatrix}~\quad~\text{or}~\quad~
\kappa (-1)^{m+1} \begin{bsmallmatrix}
0 \\
\ch{ _{$\fp$} S} (w_1, \dots, w_{2m+1})\\ 
\ch{ ^{$\fp$} T} (z_{1}, -z_{2}, \dots, -z_{2m}, z_{2m+1}) \\
0
\end{bsmallmatrix}.
\end{equation*} 
\item $(X_1,X_2) = ( \diag(J_{m}(\lambda), -J_{m}(\lambda)^\tp ),\diag(J_{n}(\lambda), -J_{n}(\lambda)^\tp ))$, $\lambda \ge 0$. In this case we have $(B_1,B_2) = (I_m \otimes H_2, I_n \otimes H_2)$. We partition $Y\in \mathbb{R}^{2m \times 2n}$ as $Y = \begin{bmatrix}
Z & W \\
U & V
\end{bmatrix}$ where $Z, W, U, V\in \mathbb{R}^{m\times n}$. Then \eqref{lem:Yij:eq4} becomes
\begin{align*}
J_{m}(\lambda) Z + Z J_{n}({\lambda}) &= 0,\quad \ J_{m}(\lambda) W - W J_{n}(\lambda)^\tp = 0, \\
J_{m}({\lambda})^\tp U - U J_{n}({\lambda}) &= 0,\quad J_{m}({\lambda})^\tp V + V J_{n}(\lambda)^\tp = 0.
\end{align*} 
If $\lambda \ne 0$ then $Z = V = 0$ and $H_m U, WH_n$ are lower triangular Toeplitz matrices. This implies that $Y_{12}$ has one of the following two forms, depending on $m\ge n$ or $m < n$:
\[
\begin{bsmallmatrix}
0& 0 \\
0   & T_{\fp}(w_1,\dots, w_{n})\\
\ch{^{$\fp$} T}(u_1,\dots, u_{n}) & 0 \\
0& 0
\end{bsmallmatrix}
~\quad~\text{or}~\quad~
\begin{bsmallmatrix}
0& 0 & 0   & T_{\fp}(w_1,\dots, w_{m})\\
\ch{^{$\fp$} T}(u_1,\dots, u_{m}) & 0 & 0 & 0
\end{bsmallmatrix}
\]
and $Y_{21}$ is 
\[
-\begin{bsmallmatrix}
0& 0 & 0   & T_{\fp}({w}_1,\dots,{w}_{m})\\
\ch{^{$\fp$} T}({u}_1,\dots, {u}_{m}) & 0 & 0 & 0
\end{bsmallmatrix}
~\quad~\text{or}~\quad~
-\begin{bsmallmatrix}
0& 0 \\
0   & T_{\fp}({w}_1,\dots,{w}_{n})\\
\ch{^{$\fp$} T}({u}_1,\dots, {u}_{n}) & 0 \\
0& 0
\end{bsmallmatrix}.
\]
If $\lambda  = 0$ then $m, n$ are even, $Z, H_mV H_n$ are lower triangular alternating Toeplitz and $H_m U, WH_n$ are lower triangular Toeplitz. Thus, $Y$ has one of the following two forms, depending on $m\ge n$ or $m < n$:
\[
\begin{bsmallmatrix}
0& 0 \\
\ch{_{$\fp$} S}(z_1,\dots, z_{n})   & T_{\fp}(w_1,\dots, w_{n})\\
\ch{^{$\fp$} T}(u_1,\dots, u_{n}) & S^{\fp} (v_1,\dots, v_{n})  \\
0& 0
\end{bsmallmatrix}
~\quad~\text{or}~\quad~
\begin{bsmallmatrix}
\ch{_{$\fp$} S}(z_1,\dots, z_{m}) & 0 & 0   & T_{\fp}(w_1,\dots, w_{m})\\
\ch{^{$\fp$} T}(u_1,\dots, u_{m}) & 0 & 0 & S^{\fp} (v_1,\dots, v_{m})
\end{bsmallmatrix}.
\]
Thus $Y_{21}$ is 
\begin{align*}
&-\begin{bsmallmatrix}
\ch{_{$\fp$} S}(v_1,-v_2,\dots, v_{n-1},  -v_{n}) & 0 & 0   & T_{\fp}(w_1,\dots, w_{n})\\
\ch{^{$\fp$} T}(u_1,\dots, u_{n}) & 0 & 0 & S^{\fp} (z_1, -z_2, \dots, z_{n-1}, -z_{n}) 
\end{bsmallmatrix} 
~\quad~\text{or}~\quad~\\
&-\begin{bsmallmatrix}
0& 0 \\
\ch{_{$\fp$} S}(v_1,-v_2,\dots,v_{m-1}, -v_{m})   & T_{\fp}(w_1,\dots, w_{m})\\
\ch{^{$\fp$} T}(u_1,\dots, u_{m}) & S^{\fp} (z_1, -z_2, \dots, z_{m-1}, -z_{m})  \\
0& 0
\end{bsmallmatrix}.
\end{align*}
In particular, $Y_{12} = Y_{21}$ implies 
\[
Y_{12} = Y_{21} = 
\begin{bsmallmatrix}
\ch{_{$\fp$} S}(z_1,\dots, z_{n})   & 0 \\
0 & -S^{\fp} (z_1,-z_2, \dots,z_{n-1}, -z_{n}) 
\end{bsmallmatrix}.
\] 
\item $(X_1, X_2) = \left( J_m\left( \begin{bmatrix}
0 & b \\
-b & 0
\end{bmatrix}
\right), J_n\left( \begin{bmatrix}
0 & b \\
-b & 0
\end{bmatrix}
\right) \right)$, $b > 0$. Thus we have \[
(B_1, B_2) = (\kappa F_2^{m-1}\otimes F_m,\kappa F_2^{n-1}\otimes F_n)
\]
and \eqref{lem:Yij:eq4} becomes 
\[
J_m\left( \begin{bmatrix}
0 & b \\
-b & 0
\end{bmatrix}
\right) Y_{12} + Y_{12} J_n\left( \begin{bmatrix}
0 & b \\
-b & 0
\end{bmatrix}
\right) = 0.
\]
According to Lemma~\ref{lem:lower triangular} \eqref{lem:lower triangular:item3}, $Y_{12}$ is a block lower triangular matrix where each block is $2\times 2$. Hence $Y$ has one of the following two forms, depending on $m\ge n$ or $m < n$:
\[
\begin{bsmallmatrix}
0\\
\ch{_{$\fp$}S}(Z_1,\dots, Z_{n})
\end{bsmallmatrix}~\quad~\text{or}~\quad~ 
\begin{bsmallmatrix}
\ch{_{$\fp$}S}(Z_1,\dots, Z_{m}) & 0
\end{bsmallmatrix}
\]
and $Y_{21}$ is 
\begin{align*}
&\begin{bsmallmatrix}
(-1)^{m} \ch{_{$\fp$}S}((F_2^{m-1} Z_1 F_2^{n-1})^\tp,\dots, (F_2^{m-1} Z_{n} F_2^{n-1})^\tp) & 0
\end{bsmallmatrix}
~\quad~\text{or}~\quad~  \\
&\begin{bsmallmatrix}
0 \\ 
(-1)^{m} \ch{_{$\fp$}S}((F_2^{m-1} Z_1 F_2^{n-1})^\tp,\dots, (F_2^{m} Z_{m-1} F_2^{n-1})^\tp) 
\end{bsmallmatrix}
\end{align*}
If $Y_{12} = Y_{21}$ then $Y_{12} = Y_{21} = \ch{_{$\fp$}S}(Z_1,\dots, Z_{n})$ where $Z_j = (-1)^{m} (F_2^{m-1} Z_j F_2^{n-1})^\tp$ for each $1 \le j \le n$.
\item Let
\begin{align*}
X_1 &=  
\diag \left( 
J_m\left( 
\begin{bmatrix}
a & b \\
-b & a
\end{bmatrix}
\right), -J_m\left( \begin{bmatrix}
a & b \\
-b & a
\end{bmatrix}
\right)^\tp
\right), \\
X_2 &= 
\diag \left( 
J_n\left( 
\begin{bmatrix}
a & b \\
-b & a
\end{bmatrix}
\right), -J_n\left( \begin{bmatrix}
a & b \\
-b & a
\end{bmatrix}
\right)^\tp
\right),
\end{align*}
 where $a,b > 0$. We have $
(B_1, B_2) = (I_{2m} \otimes H_2,I_{2n} \otimes H_2)$. We partition $Y_{12}$ as $Y_{12} = \begin{bmatrix}
Z & W \\
U & V
\end{bmatrix}\in \mathbb{R}^{4m \times 4n}$ where $Z,W,U, V\in \mathbb{R}^{2m\times 2n}$ so that \eqref{lem:Yij:eq1} becomes
\begin{align*}
J_m\left( 
\begin{bmatrix}
a & b \\
-b & a
\end{bmatrix}
\right) Z + Z J_n\left( 
\begin{bmatrix}
a & b \\
-b & a
\end{bmatrix}
\right) &= 0, \\ 
J_m\left( 
\begin{bmatrix}
a & b \\
-b & a
\end{bmatrix}
\right) W - W J_n\left( 
\begin{bmatrix}
a & b \\
-b & a
\end{bmatrix}
\right)^\tp &= 0,  \\
- J_m\left( 
\begin{bmatrix}
a & b \\
-b & a
\end{bmatrix}
\right)^\tp U + U J_n\left( 
\begin{bmatrix}
a & b \\
-b & a
\end{bmatrix}
\right) &= 0,\\ 
J_m\left( 
\begin{bmatrix}
a & b \\
-b & a
\end{bmatrix}
\right)^\tp V + V J_n\left( 
\begin{bmatrix}
a & b \\
-b & a
\end{bmatrix}
\right)^\tp &= 0.
\end{align*}
Since $a, b >0$, we conclude that $Z = V = 0$ by Lemma~\ref{lem:block}. Moreover, according to Lemma~\ref{lem:lower triangular} \eqref{lem:lower triangular:item5}, $W(I_2 \otimes H_n)$ and $(I_2 \otimes H_m)U$ are block lower triangular matrices where each block is $2\times 2$. This implies that $Y_{12}$ has one of the following two forms, depending on $m\ge n$ or $m < n$:
\[
\begin{bsmallmatrix}
0& 0 \\
0   & T_{\fp}(W_1,\dots, W_{n})\\
\ch{^{$\fp$} T}(U_1,\dots, U_{n}) & 0 \\
0& 0
\end{bsmallmatrix}
~\quad~\text{or}~\quad~
\begin{bsmallmatrix}
0& 0 & 0   & T_{\fp}(W_1,\dots, W_{m})\\
\ch{^{$\fp$} T}(U_1,\dots, U_{m}) & 0 & 0 & 0
\end{bsmallmatrix}
\]
and $Y_{21}$ is 
\[
-\begin{bsmallmatrix}
0& 0 & 0   & T_{\fp}(W_1^\tp,\dots, W_{n}^\tp)\\
\ch{^{$\fp$} T}(U_1^\tp,\dots, U_{n}^\tp) & 0 & 0 & 0
\end{bsmallmatrix}
~\quad~\text{or}~\quad~
-\begin{bsmallmatrix}
0& 0 \\
0   & T_{\fp}(W_1^\tp,\dots, W_{m}^\tp)\\
\ch{^{$\fp$} T}(U_1^\tp,\dots, U_{m}^\tp) & 0 \\
0& 0
\end{bsmallmatrix}.
\]
In particular, $Y_{12} = Y_{21}$ implies 
\[
Y_{12} = Y_{21} = 
\begin{bsmallmatrix}
0   & T_{\fp}(W_1,\dots, W_{n})\\
\ch{^{$\fp$} T}(U_1,\dots, U_{n}) & 0 
\end{bsmallmatrix},
\]
where $W_j^\tp = - W_j$ and $U_j^\tp = U_j$ for each $1 \le j \le n$.
\end{enumerate}  
\item $\varepsilon = -1$ and $\sigma$ is the transpose. By the same observation as in No.~\ref{lem:Yij:No4}, we only need to consider five sub-cases.
\begin{enumerate}[(a)]
\item $(X_1,X_2) = (J_{2m}(0),J_{2n}(0))$. We have $(B_1,B_2) = (\kappa F_{2m},\kappa F_{2n})$ and \eqref{lem:Yij:eq4} becomes $J_{2m}(0) Y_{12} + Y_{12} J_{2n}(0) = 0$. By Lemma~\ref{lem:lower triangular} \eqref{lem:lower triangular:item1}, $Y_{12}$ has one of the following two forms depending on $m \ge n$ or $m < n$:
\[
\begin{bsmallmatrix}
0\\
\ch{_{$\fp$}S}(z_1,\dots, z_{2n})
\end{bsmallmatrix}~\quad~\text{or}~\quad~ 
\begin{bsmallmatrix}
\ch{_{$\fp$}S}(z_1,\dots, z_{2m}) & 0
\end{bsmallmatrix}
\]
and $Y_{21}$ is 
\[
-\begin{bsmallmatrix}
\ch{_{$\fp$}S}(z_1,\dots, z_{2n}) & 0
\end{bsmallmatrix}
~\quad~\text{or}~\quad~ 
-\begin{bsmallmatrix}
0\\
\ch{_{$\fp$}S}(z_1,\dots, z_{2m})
\end{bsmallmatrix}.
\]
If $Y_{12} = Y_{21}$ then $Y_{12} = Y_{21} = 0$. 
\item$(X_1,X_2) = (J_{2m}(0),\diag(J_{2n+1}(0), -J_{2n+1}(0)^\tp ) )$. We have 
\[
(B_1,B_2) = (\kappa  F_{2m}, I_{2n+1}\otimes F_2),\quad Y_{12} = \begin{bmatrix}
Z & W
\end{bmatrix}
\] 
where $Z, W\in \mathbb{R}^{2m \times (2n+1)}$ satisfy
\[
J_{2m}(0) Z + Z J_{2n+1}(0) = 0,\quad J_{2m}(0) W -  W J_{2n+1}(0)^\tp = 0.
\]
Therefore, $Z$ (resp. $WH_{2n}$) is lower triangular alternating Toeplitz (resp. lower triangular Toeplitz) by Lemma~\ref{lem:lower triangular} \eqref{lem:lower triangular:item1}. This implies that $Y_{12}$ has one of the following two forms, depending on $m\ge n$ or $m < n$:
\[
\begin{bsmallmatrix}
0 & 0 \\
\ch{_{$\fp$} S}(z_1,\dots, z_{2n+1}) & T_{\fp}(w_1,\dots, w_{2n+1})
\end{bsmallmatrix}~\quad~\text{or}~\quad~ 
\begin{bsmallmatrix}
\ch{_{$\fp$} S} (z_1,\dots, z_{2m}) & 0 & 0 & T_{\fp}(w_1,\dots, w_{2m})
\end{bsmallmatrix}.
\]
Thus $Y_{21}$ is 
\[
\kappa \begin{bsmallmatrix}
\ch{_{$\fp$} S}(w_1,\dots, w_{2n+1}) & 0 \\ 
-\ch{^{$\fp$} T} (z_1, -z_2\dots, -z_{2n}, z_{2n+1}) & 0
\end{bsmallmatrix}
~\quad~\text{or}~\quad~ 
\kappa \begin{bsmallmatrix}
0 \\
\ch{_{$\fp$} S} (w_1,\dots, w_{2m}) \\ -\ch{^{$\fp$}T} (z_1, - z_{2},\dots, z_{2m-1}, -z_{2m}) \\
0
\end{bsmallmatrix}.
\]
\item $(X_1,X_2) =\left( \diag(J_{m}(\lambda), -J_{m}(\lambda)^\tp), \diag(J_{n}(\lambda), -J_{n}(\lambda)^\tp ) \right)$, $\lambda \ge 0$. In this case we have $(B_1,B_2) = ( I_m \otimes F_2, I_n \otimes F_2)$. We partition $Y\in \mathbb{R}^{2m \times 2n}$ as $Y = \begin{bmatrix}
Z & W \\
U & V
\end{bmatrix}$ where $Z, W, U, V\in \mathbb{R}^{m\times n}$, then \eqref{lem:Yij:eq4} becomes
\begin{align*}
J_{m}(\lambda) Z + Z J_{n}({\lambda}) &= 0,\quad \ J_{m}(\lambda) W - W J_{n}(\lambda)^\tp = 0, \\
J_{m}({\lambda})^\tp U - U J_{n}({\lambda}) &= 0,\quad J_{m}({\lambda})^\tp V + V J_{n}(\lambda)^\tp = 0.
\end{align*} 
If $\lambda > 0$ then $Z = V = 0$ and $H_m U$, $WH_n$ are lower triangular Toeplitz matrices. This implies that $Y_{12}$ has one of the following two forms, depending on $m\ge n$ or $m < n$:
\[
\begin{bsmallmatrix}
0& 0 \\
0   & T_{\fp}(w_1,\dots, w_{n})\\
\ch{^{$\fp$} T}(u_1,\dots, u_{n}) & 0 \\
0& 0
\end{bsmallmatrix}
~\quad~\text{or}~\quad~
\begin{bsmallmatrix}
0& 0 & 0   & T_{\fp}(w_1,\dots, w_{m})\\
\ch{^{$\fp$} T}(u_1,\dots, u_{m}) & 0 & 0 & 0
\end{bsmallmatrix}.
\]
Thus $Y_{21}$ is 
\[
-\begin{bsmallmatrix}
0& 0 & 0   & T_{\fp}(w_1,\dots, w_{n})\\
\ch{^{$\fp$} T}(u_1,\dots, u_{n}) & 0 & 0 & 0
\end{bsmallmatrix}
~\quad~\text{or}~\quad~
-\begin{bsmallmatrix}
0& 0 \\
0   & T_{\fp}(w_1,\dots, w_{m})\\
\ch{^{$\fp$} T}(u_1,\dots, u_{m}) & 0 \\
0& 0
\end{bsmallmatrix}.
\]
Hence $Y_{12} = Y_{21}$ implies $Y_{12} = Y_{21} = 0$. If $\lambda = 0$ then $m, n$ are odd and $Z, H_mV H_n$ are lower triangular alternating Toeplitz and $H_m U, WH_n$ are lower triangular Toeplitz. Thus, $Y_{12}$ has one of the following two forms, depending on $m\ge n$ or $m < n$:
\[
\begin{bsmallmatrix}
0& 0 \\
\ch{_{$\fp$} S}(z_1,\dots, z_{n})   & T_{\fp}(w_1,\dots, w_{n})\\
\ch{^{$\fp$} T}(u_1,\dots, u_{n}) & S^{\fp} (v_1,\dots, v_{n})  \\
0& 0
\end{bsmallmatrix}
~\quad~\text{or}~\quad~
\begin{bsmallmatrix}
\ch{_{$\fp$} S}(z_1,\dots, z_{m}) & 0 & 0   & T_{\fp}(w_1,\dots, w_{m})\\
\ch{^{$\fp$} T}(u_1,\dots, u_{m}) & 0 & 0 & S^{\fp} (v_1,\dots, v_{m})
\end{bsmallmatrix}.
\]
Thus $Y_{21}$ is 
\begin{align*}
&\begin{bsmallmatrix}
\ch{_{$\fp$} S}(v_1,-v_2,\dots, -v_{n-1}, v_{n}) & 0 & 0   & -T_{\fp}(w_1,\dots, w_{n})\\
-\ch{^{$\fp$} T}(u_1,\dots, u_{n}) & 0 & 0 & S^{\fp} (z_1,-z_2, \dots,-z_{n-1},  z_{n})
\end{bsmallmatrix}~\quad~\text{or}~\quad~ \\
&\begin{bsmallmatrix}
0& 0 \\
\ch{_{$\fp$} S}(v_1, -v_2,\dots,-v_{m-1}, v_{m})   & -T_{\fp}(w_1,\dots, w_{m})\\
-\ch{^{$\fp$} T}(u_1,\dots, u_{m}) & -S^{\fp} (z_1,-z_2,\dots,-z_{m-1}, z_{m})  \\
0& 0
\end{bsmallmatrix}.
\end{align*}
Moreover, $Y_{12} = Y_{21}$ implies 
\[
Y_{12} = Y_{21} = \begin{bsmallmatrix}
\ch{_{$\fp$} S}(z_1,\dots, z_{n})   & 0\\
0 & S^{\fp} (z_1,-z_2, \dots,-z_{n-1}, z_{n}) 
\end{bsmallmatrix}.
\]
\item $(X_1, X_2) = \left( J_m\left( \begin{bsmallmatrix}
0 & b \\
-b & 0
\end{bsmallmatrix}
\right), J_n\left( \begin{bsmallmatrix}
0 & b \\
-b & 0
\end{bsmallmatrix}
\right) \right)$, $b > 0$. Thus we have \[
(B_1, B_2) = (\kappa F_2^{m}\otimes F_m,\kappa F_2^n \otimes F_n)
\]
and \eqref{lem:Yij:eq4} becomes 
\[
J_m\left( \begin{bsmallmatrix}
0 & b \\
-b & 0
\end{bsmallmatrix}
\right) Y_{12} + Y_{12} J_n\left( \begin{bsmallmatrix}
0 & b \\
-b & 0
\end{bsmallmatrix}
\right) = 0.
\]
According to Lemma~\ref{lem:lower triangular} \eqref{lem:lower triangular:item3}, $Y_{12}$ is a block lower triangular alternating Toeplitz matrix where each block is $2\times 2$. Hence $Y_{12}$ has one of the following two forms, depending on $m\ge n$ or $m < n$:
\[
\begin{bsmallmatrix}
0\\
\ch{_{$\fp$}S}(Z_1,\dots, Z_{n})
\end{bsmallmatrix}~\quad~\text{or}~\quad~ 
\begin{bsmallmatrix}
\ch{_{$\fp$}S}(Z_1,\dots, Z_{m}) & 0
\end{bsmallmatrix}
\]
and $Y_{21}$ is 
\[
\begin{bsmallmatrix}
(-1)^{m-1} \ch{_{$\fp$}S}((F_2^{m}Z_1F_2^n)^\tp,\dots, (F_2^{m}Z_{n} F_2^n)^\tp) & 0
\end{bsmallmatrix}~\quad~\text{or}~\quad~ 
\begin{bsmallmatrix}
0 \\ 
(-1)^{m-1} \ch{_{$\fp$}S}((F_2^{m}Z_1F_2^n)^\tp,\dots, (F_2^{m}Z_{m}F_2^n)^\tp) 
\end{bsmallmatrix}
\]
If $Y_{12} = Y_{21}$ then $Y_{12} = Y_{21} = \ch{_{$\fp$}S}(Z_1,\dots, Z_{n})$ where $Z_j = (-1)^{m-1} (F_2^{m} Z_j F_2^n)^\tp$ for each $1 \le j \le n$.
\item Let
\begin{align*}
X_1 &=  
\diag \left( 
J_m\left( 
\begin{bsmallmatrix}
a & b \\
-b & a
\end{bsmallmatrix}
\right), -J_m\left( \begin{bsmallmatrix}
a & b \\
-b & a
\end{bsmallmatrix}
\right)^\tp
\right), \\
X_2 &= 
\diag \left( 
J_n\left( 
\begin{bsmallmatrix}
a & b \\
-b & a
\end{bsmallmatrix}
\right), -J_n\left( \begin{bsmallmatrix}
a & b \\
-b & a
\end{bsmallmatrix}
\right)^\tp
\right),
\end{align*}
 where $a,b > 0$. We have $
(B_1, B_2) = (I_{2m} \otimes F_2,I_{2n} \otimes F_2)$. We partition $Y_{12}$ as $Y_{12} = \begin{bsmallmatrix}
Z & W \\
U & V
\end{bsmallmatrix}\in \mathbb{R}^{4m \times 4n}$ where $Z,W,U, V\in \mathbb{R}^{2m\times 2n}$ so that \eqref{lem:Yij:eq1} becomes
\begin{align*}
J_m\left( 
\begin{bsmallmatrix}
a & b \\
-b & a
\end{bsmallmatrix}
\right) Z + Z J_n\left( 
\begin{bsmallmatrix}
a & b \\
-b & a
\end{bsmallmatrix}
\right) &= 0, \\ 
J_m\left( 
\begin{bsmallmatrix}
a & b \\
-b & a
\end{bsmallmatrix}
\right) W - W J_n\left( 
\begin{bsmallmatrix}
a & b \\
-b & a
\end{bsmallmatrix}
\right)^\tp &= 0,  \\
- J_m\left( 
\begin{bsmallmatrix}
a & b \\
-b & a
\end{bsmallmatrix}
\right)^\tp U + U J_n\left( 
\begin{bsmallmatrix}
a & b \\
-b & a
\end{bsmallmatrix}
\right) &= 0,\\ 
J_m\left( 
\begin{bsmallmatrix}
a & b \\
-b & a
\end{bsmallmatrix}
\right)^\tp V + V J_n\left( 
\begin{bsmallmatrix}
a & b \\
-b & a
\end{bsmallmatrix}
\right)^\tp &= 0.
\end{align*}
We conclude that $Z = V = 0$ by Lemma~\ref{lem:block}. Moreover, according to Lemma~\ref{lem:lower triangular} \eqref{lem:lower triangular:item5}, $W(I_2 \otimes H_n)$ and $(I_2 \otimes H_m)U$ are block lower triangular matrices where each block is $2\times 2$. This implies that $Y_{12}$ has one of the following two forms, depending on $m\ge n$ or $m < n$:
\[
\begin{bsmallmatrix}
0& 0 \\
0   & T_{\fp}(W_1,\dots, W_{n})\\
\ch{^{$\fp$} T}(U_1,\dots, U_{n}) & 0 \\
0& 0
\end{bsmallmatrix}
~\quad~\text{or}~\quad~
\begin{bsmallmatrix}
0& 0 & 0   & T_{\fp}(W_1,\dots, W_{m})\\
\ch{^{$\fp$} T}(U_1,\dots, U_{m}) & 0 & 0 & 0
\end{bsmallmatrix}.
\]
Thus $Y_{21}$ is 
\[
-\begin{bsmallmatrix}
0& 0 & 0   & T_{\fp}(W_1^\tp,\dots, W_{n}^\tp)\\
\ch{^{$\fp$} T}(U_1^\tp,\dots, U_{n}^\tp) & 0 & 0 & 0
\end{bsmallmatrix}
~\quad~\text{or}~\quad~
-\begin{bsmallmatrix}
0& 0 \\
0   & T_{\fp}(W_1^\tp,\dots, W_{m}^\tp)\\
\ch{^{$\fp$} T}(U_1^\tp,\dots, U_{m}^\tp) & 0 \\
0& 0
\end{bsmallmatrix}.
\]
If $Y_{12} = Y_{21}$, then 
\[
Y_{12} = Y_{21} = \begin{bsmallmatrix}
0   & T_{\fp}(W_1,\dots, W_{n})\\
\ch{^{$\fp$} T}(U_1,\dots, U_{n}) & 0
\end{bsmallmatrix},
\]
where $U_1,\dots, U_n ,W_1,\dots, W_n$ are skew symmetric .
\end{enumerate}
\item\label{lem:Yij:No6} $\varepsilon = 1$ and $\sigma$ is the conjugate transpose. We observe that for the three types of normal forms, we have 
\begin{align*}
0&\not\in \sigma \left( J_m\left(\begin{bsmallmatrix}
0 & 0 \\
0 & 0
\end{bsmallmatrix}\right) \right) 
+ \sigma\left( J_n\left(\begin{bsmallmatrix}
0 & b \\
-b & 0
\end{bsmallmatrix}\right)
\right),\\
0&\not\in 
\sigma \left( J_m\left(\begin{bsmallmatrix}
0 & 0 \\
0 & 0
\end{bsmallmatrix}\right) \right)  +\sigma\left( \diag
\left(
J_n\left(\begin{bsmallmatrix}
\lambda & 0 \\
0 & \overline{\lambda}
\end{bsmallmatrix}\right),
-J_n\left(\begin{bsmallmatrix}
\lambda & 0 \\
0 & \overline{\lambda}
\end{bsmallmatrix}\right)^\ast
\right)
\right), \\
0&\not\in 
\sigma \left( J_m\left(\begin{bsmallmatrix}
0 & b \\
-b & 0
\end{bsmallmatrix}\right) \right)  +\sigma\left( \diag
\left(
J_n\left(\begin{bsmallmatrix}
\lambda & 0 \\
0 & \overline{\lambda}
\end{bsmallmatrix}\right),
-J_n\left(\begin{bsmallmatrix}
\lambda & 0 \\
0 & \overline{\lambda}
\end{bsmallmatrix}\right)^\ast
\right)
\right), \\
\end{align*}
where $b >0$ and $\lambda \in \mathbb{C}, \operatorname{Re}(\lambda) > 0, \operatorname{Im}(\lambda) \ge 0$. Hence we only need to consider three sub-cases.
\begin{enumerate}[(a)]
\item $(X_1,X_2) = \left( J_m\left(\begin{bsmallmatrix}
0 & 0 \\
0 & 0
\end{bsmallmatrix}\right), J_n\left(\begin{bsmallmatrix}
0 & 0 \\
0 & 0
\end{bsmallmatrix}\right) \right)$. We have 
\[
(B_1,B_2) = (\kappa^{m} F_2^{m-1}\otimes  F_{m}, 
\kappa^{n} F_2^{n-1}\otimes  F_{n} )
\] 
and by Lemma~\ref{lem:lower triangular} \eqref{lem:lower triangular:item3}, $Y_{12}$ is block lower triangular whose blocks are $2\times 2$. Hence $Y_{12}$ has one of the following two forms, depending on $m\ge n$ or $m < n$:
\[
\begin{bsmallmatrix}
0\\
\ch{_{$\fp$}S}(Z_1,\dots, Z_{n})
\end{bsmallmatrix}~\quad~\text{or}~\quad~ 
\begin{bsmallmatrix}
\ch{_{$\fp$}S}(Z_1,\dots, Z_{m}) & 0
\end{bsmallmatrix},
\]
and $Y_{21}$ is 
\begin{align*}
&(-1)^m \kappa^{m+n} \begin{bsmallmatrix}
\ch{_{$\fp$}S}( (F_2^{m-1} Z_1 F_2^{n-1})^\ast,\dots, (F_2^{m-1} Z_n F_2^{n-1})^\ast) & 0
\end{bsmallmatrix}~\quad~\text{or}~\quad~  \\
&(-1)^m \kappa^{m+n} \begin{bsmallmatrix}
0\\
\ch{_{$\fp$}S}( (F_2^{m-1} Z_1 F_2^{n-1})^\ast,\dots, (F_2^{m-1} Z_m F_2^{n-1})^\ast)
\end{bsmallmatrix}.
\end{align*}
If $Y_{12} = Y_{21}$ then $Y_{12} = Y_{21} = \ch{_{$\fp$}S}(Z_1,\dots, Z_{n})
$ where $(-1)^n (F_2^{n-1} Z_j F_2^{n-1})^\ast= Z_j$ for each $1 \le j \le n$.
\item $(X_1,X_2) = \left( J_m\left(\begin{bsmallmatrix}
0 & b \\
-b & 0
\end{bsmallmatrix}\right), J_n\left(\begin{bsmallmatrix}
0 & b \\
-b & 0
\end{bsmallmatrix}\right) \right)$, $b > 0$. We have 
\[
(B_1,B_2) = (\kappa F_2^{m-1}\otimes  F_{m}, 
\kappa F_2^{n-1}\otimes  F_{n} ).
\] 
By Lemma~\ref{lem:lower triangular} \eqref{lem:lower triangular:item3}, $Y_{12}$ has one of the following two forms, depending on $m\ge n$ or $m < n$:
\[
\begin{bsmallmatrix}
0\\
\ch{_{$\fp$}S}(Z_1,\dots, Z_{n})
\end{bsmallmatrix}~\quad~\text{or}~\quad~ 
\begin{bsmallmatrix}
\ch{_{$\fp$}S}(Z_1,\dots, Z_{m}) & 0
\end{bsmallmatrix},
\]
and $Y_{21}$ is 
\begin{align*}
&(-1)^m \begin{bsmallmatrix}
\ch{_{$\fp$}S}( (F_2^{m-1} Z_1 F_2^{n-1})^\ast,\dots, (F_2^{m-1} Z_n F_2^{n-1})^\ast) & 0
\end{bsmallmatrix}~\quad~\text{or}~\quad~  \\
&(-1)^m \begin{bsmallmatrix}
0\\
\ch{_{$\fp$}S}( (F_2^{m-1} Z_1 F_2^{n-1})^\ast,\dots, (F_2^{m-1} Z_m F_2^{n-1})^\ast)
\end{bsmallmatrix}.
\end{align*}
If $Y_{12} = Y_{21}$ then $Y_{12} = Y_{21} = \ch{_{$\fp$}S}(Z_1,\dots, Z_{n})
$ where $(-1)^n (F_2^{n-1} Z_j F_2^{n-1})^\ast= Z_j$ for each $1 \le j \le n$.
\item Let 
\begin{align*}
X_1  &=\left( 
\diag
\left(
J_m\left(
\begin{bsmallmatrix}
\lambda & 0 \\
0 & \overline{\lambda}
\end{bsmallmatrix}\right),
-J_m\left(\begin{bsmallmatrix}
\lambda & 0 \\
0 & \overline{\lambda}
\end{bsmallmatrix}\right)^\ast
\right)
\right), \\
X_2  &=\left( 
\diag
\left(
J_n \left(
\begin{bsmallmatrix}
\lambda & 0 \\
0 & \overline{\lambda}
\end{bsmallmatrix}\right),
-J_n\left(\begin{bsmallmatrix}
\lambda & 0 \\
0 & \overline{\lambda}
\end{bsmallmatrix}\right)^\ast
\right)
\right),
\end{align*}
where $\lambda \in \mathbb{C}, \operatorname{Re}(\lambda) > 0, \operatorname{Im}(\lambda) \ge 0$. Then we have $(B_1,B_2) = ( I_{2m} \otimes H_2,I_{2n} \otimes H_2)$. We partition $Y\in \mathbb{H}^{4m \times 4n}$ as $Y = \begin{bsmallmatrix}
Z & W \\
U & V
\end{bsmallmatrix}$ where $Z, W, U, V\in \mathbb{H}^{2m\times 2n}$. Then \eqref{lem:Yij:eq1} becomes
\begin{align*}
J_m\left(
\begin{bsmallmatrix}
\lambda & 0 \\
0 & \overline{\lambda}
\end{bsmallmatrix}\right) Z + Z J_n\left(
\begin{bsmallmatrix}
\lambda & 0 \\
0 & \overline{\lambda}
\end{bsmallmatrix}\right) &= 0, \\ 
J_m\left(
\begin{bsmallmatrix}
\lambda & 0 \\
0 & \overline{\lambda}
\end{bsmallmatrix}\right) W - W J_n\left(
\begin{bsmallmatrix}
\lambda & 0 \\
0 & \overline{\lambda}
\end{bsmallmatrix}\right)^\ast &= 0,  \\
- J_m\left(
\begin{bsmallmatrix}
\lambda & 0 \\
0 & \overline{\lambda}
\end{bsmallmatrix}\right)^\ast U + U J_n \left(
\begin{bsmallmatrix}
\lambda & 0 \\
0 & \overline{\lambda}
\end{bsmallmatrix}\right) &= 0,\\ 
J_m\left(
\begin{bsmallmatrix}
\lambda & 0 \\
0 & \overline{\lambda}
\end{bsmallmatrix}\right)^\ast V + V J_n\left(
\begin{bsmallmatrix}
\lambda & 0 \\
0 & \overline{\lambda}
\end{bsmallmatrix}\right)^\ast &= 0.
\end{align*}
We conclude that $Z = V = 0$ by Lemma~\ref{lem:block} since $\lambda \ne 0$. Moreover, according to Lemma~\ref{lem:lower triangular} \eqref{lem:lower triangular:item6}, $W(I_2 \otimes H_n)$ and $(I_2 \otimes H_m)U$ are block lower triangular matrices where each block is $2\times 2$. This implies that $Y_{12}$ has one of the following two forms, depending on $m\ge n$ or $m < n$:
\[
\begin{bsmallmatrix}
0& 0 \\
0   & T_{\fp}(W_1,\dots, W_{n})\\
\ch{^{$\fp$} T}(U_1,\dots, U_{n}) & 0 \\
0& 0
\end{bsmallmatrix}
~\quad~\text{or}~\quad~
\begin{bsmallmatrix}
0& 0 & 0   & T_{\fp}(W_1,\dots, W_{m})\\
\ch{^{$\fp$} T}(U_1,\dots, U_{m}) & 0 & 0 & 0
\end{bsmallmatrix}.
\]
Thus $Y_{21}$ is 
\[
-\begin{bsmallmatrix}
0& 0 & 0   & T_{\fp}(W_1^\ast,\dots, W_{n}^\ast)\\
\ch{^{$\fp$} T}(U_1^\ast,\dots, U_{n}^\ast) & 0 & 0 & 0
\end{bsmallmatrix}
~\quad~\text{or}~\quad~
-\begin{bsmallmatrix}
0& 0 \\
0   & T_{\fp}(W_1^\ast,\dots, W_{m}^\ast)\\
\ch{^{$\fp$} T}(U_1^\ast,\dots, U_{m}^\ast) & 0 \\
0& 0
\end{bsmallmatrix}.
\]
If $Y_{12} = Y_{21}$ then $Y_{12} = Y_{21} = \begin{bsmallmatrix}
0   & T_{\fp}(W_1,\dots, W_{n})\\
\ch{^{$\fp$} T}(U_1,\dots, U_{n}) & 0 \\
\end{bsmallmatrix}$ where $-U_j^\ast = U_j $ and $-W_j^\ast = W_j$ for each $1\le j \le n$.
\end{enumerate}  
\item \label{lem:Yij:No7} $\varepsilon = - 1$ and $\sigma$ is the conjugate transpose. By the same argument in No.~\ref{lem:Yij:No6}, it suffices to consider three sub-cases.
\begin{enumerate}[(a)]
\item $(X_1,X_2) = \left( J_m\left(\begin{bsmallmatrix}
0 & 0 \\
0 & 0
\end{bsmallmatrix}\right), J_n\left(\begin{bsmallmatrix}
0 & 0 \\
0 & 0
\end{bsmallmatrix}\right) \right)$. We have 
\[
(B_1,B_2) = (\kappa^{m-1} F_2^{m}\otimes  F_{m}, 
\kappa^{n-1} F_2^{n}\otimes  F_{n} )
\] 
and by Lemma~\ref{lem:lower triangular} \eqref{lem:lower triangular:item3}, $Y_{12}$ is block lower triangular whose blocks are $2\times 2$. Hence $Y_{12}$ has one of the following two forms, depending on $m\ge n$ or $m < n$:
\[
\begin{bsmallmatrix}
0\\
\ch{_{$\fp$}S}(Z_1,\dots, Z_{n})
\end{bsmallmatrix}~\quad~\text{or}~\quad~ 
\begin{bsmallmatrix}
\ch{_{$\fp$}S}(Z_1,\dots, Z_{m}) & 0
\end{bsmallmatrix},
\]
and $Y_{21}$ is 
\begin{align*}
&\kappa^{m+n} (-1)^{m-1} \begin{bsmallmatrix}
\ch{_{$\fp$}S}( (F_2^{m} Z_1 F_2^{n})^\ast,\dots, (F_2^{m} Z_n F_2^{n})^\ast) & 0
\end{bsmallmatrix}~\quad~\text{or}~\quad~  \\
&\kappa^{m+n} (-1)^{m-1} \begin{bsmallmatrix}
0\\
\ch{_{$\fp$}S}( (F_2^{m} Z_1 F_2^{n})^\ast,\dots, (F_2^{m} Z_m F_2^{n})^\ast)
\end{bsmallmatrix}.
\end{align*}
If $Y_{12} = Y_{21}$ then $Y_{12} = Y_{21} = \ch{_{$\fp$}S}(Z_1,\dots, Z_{n})
$ where $\kappa^{m+n} (-1)^{m-1} (F_2^{m} Z_j F_2^{n})^\ast= Z_j$ for each $1 \le j \le n$.
\item $(X_1,X_2) = \left( J_m\left(\begin{bsmallmatrix}
0 & b \\
-b & 0
\end{bsmallmatrix}\right), J_n\left(\begin{bsmallmatrix}
0 & b \\
-b & 0
\end{bsmallmatrix}\right) \right)$, $b > 0$. We have 
\[
(B_1,B_2) = (\kappa F_2^{m}\otimes  F_{m}, 
\kappa F_2^{n}\otimes  F_{n} ).
\] 
By Lemma~\ref{lem:lower triangular} \eqref{lem:lower triangular:item3}, $Y_{12}$ has one of the following two forms, depending on $m\ge n$ or $m < n$:
\[
\begin{bsmallmatrix}
0\\
\ch{_{$\fp$}S}(Z_1,\dots, Z_{n})
\end{bsmallmatrix}~\quad~\text{or}~\quad~ 
\begin{bsmallmatrix}
\ch{_{$\fp$}S}(Z_1,\dots, Z_{m}) & 0
\end{bsmallmatrix},
\]
and $Y_{21}$ is 
\begin{align*}
&(-1)^{m-1} \begin{bsmallmatrix}
\ch{_{$\fp$}S}( (F_2^{m} Z_1 F_2^{n})^\ast,\dots, (F_2^{m} Z_n F_2^{n})^\ast) & 0
\end{bsmallmatrix}~\quad~\text{or}~\quad~  \\
&(-1)^{m-1} \begin{bsmallmatrix}
0\\
\ch{_{$\fp$}S}( (F_2^{m} Z_1 F_2^{n})^\ast,\dots, (F_2^{m} Z_m F_2^{n})^\ast)
\end{bsmallmatrix}.
\end{align*}
If $Y_{12} = Y_{21}$ then $Y_{12} = Y_{21} = \ch{_{$\fp$}S}(Z_1,\dots, Z_{n})
$ where $(-1)^{m-1} (F_2^{m} Z_j F_2^{n})^\ast= Z_j$ for each $1 \le j \le n$.
\item Let 
\begin{align*}
X_1  &=\left( 
\diag
\left(
J_m\left(
\begin{bsmallmatrix}
\lambda & 0 \\
0 & \overline{\lambda}
\end{bsmallmatrix}\right),
-J_m\left(\begin{bsmallmatrix}
\lambda & 0 \\
0 & \overline{\lambda}
\end{bsmallmatrix}\right)^\ast
\right)
\right), \\
X_2  &=\left( 
\diag
\left(
J_n \left(
\begin{bsmallmatrix}
\lambda & 0 \\
0 & \overline{\lambda}
\end{bsmallmatrix}\right),
-J_n\left(\begin{bsmallmatrix}
\lambda & 0 \\
0 & \overline{\lambda}
\end{bsmallmatrix}\right)^\ast
\right)
\right),
\end{align*}
where $\lambda \in \mathbb{C}, \operatorname{Re}(\lambda) > 0, \operatorname{Im}(\lambda) \ge 0$. Then we have $(B_1,B_2) = ( I_{2m} \otimes F_2,I_{2n} \otimes F_2)$. We partition $Y\in \mathbb{H}^{4m \times 4n}$ as $Y = \begin{bsmallmatrix}
Z & W \\
U & V
\end{bsmallmatrix}$ where $Z, W, U, V\in \mathbb{H}^{2m\times 2n}$. Then \eqref{lem:Yij:eq1} becomes
\begin{align*}
J_m\left(
\begin{bsmallmatrix}
\lambda & 0 \\
0 & \overline{\lambda}
\end{bsmallmatrix}\right) Z + Z J_n\left(
\begin{bsmallmatrix}
\lambda & 0 \\
0 & \overline{\lambda}
\end{bsmallmatrix}\right) &= 0, \\ 
J_m\left(
\begin{bsmallmatrix}
\lambda & 0 \\
0 & \overline{\lambda}
\end{bsmallmatrix}\right) W - W J_n\left(
\begin{bsmallmatrix}
\lambda & 0 \\
0 & \overline{\lambda}
\end{bsmallmatrix}\right)^\ast &= 0,  \\
- J_m\left(
\begin{bsmallmatrix}
\lambda & 0 \\
0 & \overline{\lambda}
\end{bsmallmatrix}\right)^\ast U + U J_n \left(
\begin{bsmallmatrix}
\lambda & 0 \\
0 & \overline{\lambda}
\end{bsmallmatrix}\right) &= 0,\\ 
J_m\left(
\begin{bsmallmatrix}
\lambda & 0 \\
0 & \overline{\lambda}
\end{bsmallmatrix}\right)^\ast V + V J_n\left(
\begin{bsmallmatrix}
\lambda & 0 \\
0 & \overline{\lambda}
\end{bsmallmatrix}\right)^\ast &= 0.
\end{align*}
We conclude that $Z = V = 0$ by Lemma~\ref{lem:block} since $\lambda \ne 0$. Moreover, according to Lemma~\ref{lem:lower triangular} \eqref{lem:lower triangular:item6}, $W(I_2 \otimes H_n)$ and $(I_2 \otimes H_m)U$ are block lower triangular matrices where each block is $2\times 2$. This implies that $Y_{12}$ has one of the following two forms, depending on $m\ge n$ or $m < n$:
\[
\begin{bsmallmatrix}
0& 0 \\
0   & T_{\fp}(W_1,\dots, W_{n})\\
\ch{^{$\fp$} T}(U_1,\dots, U_{n}) & 0 \\
0& 0
\end{bsmallmatrix}
~\quad~\text{or}~\quad~
\begin{bsmallmatrix}
0& 0 & 0   & T_{\fp}(W_1,\dots, W_{m})\\
\ch{^{$\fp$} T}(U_1,\dots, U_{m}) & 0 & 0 & 0
\end{bsmallmatrix}.
\]
Thus $Y_{21}$ is 
\[
-\begin{bsmallmatrix}
0& 0 & 0   & T_{\fp}(W_1^\ast,\dots, W_{n}^\ast)\\
\ch{^{$\fp$} T}(U_1^\ast,\dots, U_{n}^\ast) & 0 & 0 & 0
\end{bsmallmatrix}
~\quad~\text{or}~\quad~
-\begin{bsmallmatrix}
0& 0 \\
0   & T_{\fp}(W_1^\ast,\dots, W_{m}^\ast)\\
\ch{^{$\fp$} T}(U_1^\ast,\dots, U_{m}^\ast) & 0 \\
0& 0
\end{bsmallmatrix}.
\]
If $Y_{12} = Y_{21}$ then $Y_{12} = Y_{21} = \begin{bsmallmatrix}
0   & T_{\fp}(W_1,\dots, W_{n})\\
\ch{^{$\fp$} T}(U_1,\dots, U_{n}) & 0 \\
\end{bsmallmatrix}$ where $U_j = -U_j^\ast$ and $W_j = -W_j^\ast$ for each $1\le j \le n$.
\end{enumerate}  
\end{enumerate}
\end{proof}
\section{Proof of Theorem~\ref{thm:classification O_n2}}\label{append:proof of O_n2}
\begin{proof}
By the same argument as in the proof of Theorem~\ref{thm:classification O_n1}, we may write
\begin{equation}\label{thm:classification O_n2:eq1}
\alpha(t) =R\left( \frac{t^2  I_{n+2} + t \diag(X_1,\dots, X_{s}) + Y }{t^2 + 1} \right)R^{-1}, \quad I_{n,2}  = R \diag(B_1,\dots, B_{s}) R^{\tp},
\end{equation}
where $(X_1,\dots, X_{s})$, $(B_1,\dots, B_{s})$, $R\in \GL_{n+2}(\mathbb{R})$ and $Y = (Y_{pq})_{p,q=1}^{s}$ are those in Table~\ref{Tab:indecomposable} and Table~\ref{Tab:CandidatesYij}~No.~4, respectively. If we denote by $(p_j,q_j)$ the signature of $B_j$ for each $1 \le j \le s$, then $(n,2) = (\sum_{j=1}^s p_j, \sum_{j=1}^s q_j)$ and one of the following two cases must hold:
\begin{enumerate}[(a)]
\item\label{thm:classification O_n2:item-a} $(p_{s-1}, q_{s-1}), (p_s, q_s) \in  \{ 
(0,1), (1,1), (2,1), (3,1) \}$ and $(p_j,q_j)  \in  \{ 
(1,0), (2,0)\}$, $1 \le j \le s-2$.
\item\label{thm:classification O_n2:item-b} $(p_s, q_s) \in  \{ 
(0,2), (1,2), (2,2), (3,2), (4,2) \}$ and $(p_j,q_j)  \in  \{ 
(1,0), (2,0)\}$, $1 \le j \le s-1$.
\end{enumerate}
For simplicity, we suppose that for \eqref{thm:classification O_n2:item-a},
\[
(p_1,q_1) = \cdots = (p_m,q_m) = (1,0),\quad (p_{m+1},q_{m+1}) = \cdots = (p_{s-2},q_{s-2}) = (2,0),
\] 
while for \eqref{thm:classification O_n2:item-b},
\[
(p_1,q_1) = \cdots = (p_m,q_m) = (1,0),\quad (p_{m+1},q_{m+1}) = \cdots = (p_{s-1},q_{s-1}) = (2,0).
\] 
Moreover, we observe that in case~\eqref{thm:classification O_n2:item-a}, if $0 \not\in \rho(X_{s-1}) + \rho(X_s)$ then clearly $\alpha$ is obtained by the natural inclusion $\O_{m,1} \times \O_{n-m,1} \subseteq \O_{n,2}$. Thus we may assume $0 \in \rho(X_{s-1}) + \rho(X_s)$ in \eqref{thm:classification O_n2:item-a}. 

Our subsequent discussion is split into ten sub-cases. The first five are obtained from \eqref{thm:classification O_n2:item-a}:
\begin{enumerate}[(\ref{thm:classification O_n2:item-a}1)]
\item $(p_{s-1},q_{s-1}) = (p_{s},q_{s}) = (0,1)$, $(X_{s-1},X_s) = (0,0)$, $(B_{s-1},B_s) = (-1,-1)$ and $(\kappa_{s-1},\kappa_s) = (-1,-1)$: the matrix $R$ in \eqref{thm:classification O_n2:eq1} is in $\O_{n,2}$. By Theorem~\ref{thm:structure}, it suffices to consider $Y' \coloneqq (Y_{pq})$ for $p,q\in \{1,\dots, m, s-1, s\}$, which can be written as $Y' = \begin{bsmallmatrix}
I_{m} + \Lambda & Z \\
Z^\tp & I_2 + x F_2
\end{bsmallmatrix}$ for some $Z\in \mathbb{R}^{m \times 2}$, $\Lambda \in \mathfrak{o}_{m}(\mathbb{R})$ and $x\in \mathbb{R}$. According to \eqref{thm:structure:eq1}, we have 
\[
\begin{bsmallmatrix}
I_m + \Lambda & Z \\
Z^\tp & I_2 + x F_2
\end{bsmallmatrix} 
\begin{bsmallmatrix}
I_m  & 0 \\
0 & -I_2
\end{bsmallmatrix}
\begin{bsmallmatrix}
I_m + \Lambda & Z \\
Z^\tp & I_2 + x F_2
\end{bsmallmatrix} ^\tp 
= \begin{bsmallmatrix}
I_m  & 0 \\
0 & -I_2
\end{bsmallmatrix}.
\]
This implies 
\[
\Lambda^2 = -ZZ^\tp,\quad Z^\tp Z = x^2 I_2,\quad \Lambda Z + x Z F_2 = 0.
\]
Observing that $\rank (\Lambda)\le 2$, we may write $\Lambda^2 =-\lambda^2 Q \diag (I_2,0) Q^\tp$ for some $Q\in \O_m(\mathbb{R})$ and $\lambda \ge 0$. Thus, $\lambda^2 = x^2$ and $Q^\tp Z = \begin{bsmallmatrix}
Z_1 \\ 0  
\end{bsmallmatrix}$, where $Z_1\in \mathbb{R}^{2\times 2}$ and $Z_1 Z_1^\tp = \lambda^2 I_2$. Thus we obtain 
\[
(\Lambda, Z) = \begin{cases}
\left( 
Q \diag(\lambda F_2, 0) Q^\tp
,Q \begin{bsmallmatrix}
a & b \\ 
b & -a \\
0 & 0 \\
\vdots & \vdots \\
0 & 0
\end{bsmallmatrix} \right),&~\text{if}~ x = \lambda \vspace*{4pt} \\ 
\left( 
Q \diag(\lambda F_2, 0) Q^\tp,
Q \begin{bsmallmatrix}
a & b \\ 
-b & a \\
0 & 0 \\
\vdots & \vdots \\
0 & 0
\end{bsmallmatrix}\right), &~\text{if}~x = -\lambda
\end{cases}.
\]
Here $a,b\in \mathbb{R}$ satisfy $a^2 + b^2 = \lambda^2$.
\item $(p_{s-1},q_{s-1}) = (0,1)$, $(p_s,q_s) = (2,1)$, $(X_{s-1},X_s) = (0, J_3(0)), (B_{s-1},B_s) = (-1, -F_3)$ and $(\kappa_{s-1},\kappa_s) = (-1,1)$: the matrix $R$ in \eqref{thm:classification O_n2:eq1} satisfies $R \diag(I_{n-2},Q_{1,3}) \in \O_{n,2}(\mathbb{R})$. Let $Y' \coloneqq (Y_{pq})$ for $p,q\in \{1,\dots, m, s-1, s\}$. According to Theorem~\ref{thm:structure}, we write 
\[
Y' = \begin{bsmallmatrix}
I_{m} + \Lambda & z & w & 0 & 0 \\
z^\tp & 1 & y & 0 & 0 \\
0 & 0 & 1 & 0 & 0 \\
0 & 0 & 0 & 1 & 0 \\
w^\tp & -y  & \frac{1}{2} & 0 & 1 
\end{bsmallmatrix},\quad \Lambda \in \mathfrak{o}_m(\mathbb{R}), z,w \in \mathbb{R}^m, y\in \mathbb{R}.
\]
By \eqref{thm:structure:eq1} we have $Y' \diag(I_m, - 1, -F_3) {Y'}^\tp = \diag(I_m, - 1, -F_3)$, which implies 
\[
I_m - \Lambda^2 - z z^\tp = I_m,\quad  z^\tp z -1 = -1, \quad y = 0, \quad w^\tp w  = 1.
\]
Thus we obtain $\Lambda = 0, z = 0, y = 0$ and $w\in \mathbb{S}^{m-1}$.
\item $(p_{s-1},q_{s-1}) = (p_s,q_s) = (1,1)$, $(X_{s-1},X_s) = (\diag(\lambda,-\lambda)$, $\diag(\mu,-\mu)), \lambda,\mu >0$, $B_{s-1} = B_s = H_2$: equation~\eqref{thm:structure:eq1} implies $Y' \diag(H_2,H_2) {Y'}^\tp = \diag(H_2,H_2)$ where 
\[
Y' = \begin{bsmallmatrix}
1 + \frac{\lambda^2}{2} & 0  & 0  & w \\
0 & 1 + \frac{\lambda^2}{2} &  z & 0  \\
0 & -w  & 1 + \frac{\mu^2}{2} & 0 \\
-z & 0 & 0 &  1 + \frac{\mu^2}{2}
\end{bsmallmatrix},\quad z = w = 0 \text{ if } \lambda = \mu.
\]
A direct calculation leads to a contradiction that $\lambda = 0$.
\item $(p_{s-1},q_{s-1}) = (p_s,q_s) = (2,1)$, $(X_{s-1},X_s) = (J_3(0), J_3(0))$, $B_{s-1} = B_s = -F_3$ and $\kappa_{s-1} = \kappa_s = 1$: the matrix $R$ in \eqref{thm:classification O_n2:eq1} satisfies $R \diag(I_{n-4}, Q_{3,3}) \in \O_{n,2}(\mathbb{R})$. Let $Y' = (Y_{pq})$ for $p,q\in \{1,\dots, m,s-1,s\}$. By Theorem~\ref{thm:structure} we may write 
\[
Y' = \begin{bsmallmatrix}
I_m + \Lambda & x & 0 & 0 & y & 0 & 0 \\
0 &  &  &  & &  & \\ 
0 & & I_3 + \frac{1}{2} J_3(0)^2 &  &  & \ch{_{$\fp$} $S(z)$} &\\
x^\tp & & &  &  &  &  \\
0 & & &  &  &  & \\
0  & &  -\ch{_{$\fp$} $S(z)$}  & &  & I_3 + \frac{1}{2} J_3(0)^2  &\\
y^\tp & & &  &  &  &
\end{bsmallmatrix},\quad \Lambda\in \mathfrak{o}_m(\mathbb{R}), x,y\in \mathbb{R}^m, z = (z_1,z_2,z_3)\in \mathbb{R}^3.
\]
Now \eqref{thm:structure:eq1} indicates that $\Lambda =0, z_3 = 0$, $x^\tp y = 0$, $z_2 \in [-1,1]$ and $x^\tp x =  y^\tp y = 1 - z_2^2$.
\item $(p_{s-1},q_{s-1}) = (p_s,q_s) = (3,1)$: according to Table~\ref{Tab:indecomposable}~No.~4, we must have $3 - 1 = \kappa (1 - (-1)^2) = 0$ which is impossible. 
\end{enumerate}
Next we deal with the other five sub-cases from \eqref{thm:classification O_n2:item-b}.
\begin{enumerate}[(\ref{thm:classification O_n2:item-b}1)]
\item $(p_s,q_s) = (0,2)$, $X_s = \begin{bsmallmatrix}
0 & b \\
-b & 0
\end{bsmallmatrix}, b > 0$, $\kappa_s = -1$, $B_s = -I_2$: the matrix $R$ in \eqref{thm:classification O_n2:eq1} lies in $\O_{n,2}(\mathbb{R})$. Without loss of generality, we assume $X_p = X_s$ for all $p\in \{m+1,\dots, s\}$. Let $Y' = (Y_{pq})$ where $p, q \in \{m+1,\dots, s\}$. According to Theorem~\ref{thm:structure}, we can write 
\[
Y' = 
\begin{bsmallmatrix}
(1 - b^2/2) I_{2(s - m -1)} + \Lambda & Z \\
Z^\tp & (1 - b^2/2) I_2
\end{bsmallmatrix},\quad \Lambda\in \mathfrak{o}_{2(s - m -1)}(\mathbb{R}), 
Z = \begin{bsmallmatrix}
Z_1 \\
\vdots \\
Z_{s - m -1}
\end{bsmallmatrix}
 \in \mathbb{R}^{2(s-m-1) \times 2},
\]
where each $Z_p$ is of the form $Z_p =\begin{bsmallmatrix}
x_p & y_p \\
y_p & -x_p
\end{bsmallmatrix}$ for some $x_p,y_p \in \mathbb{R}$. By \eqref{thm:structure:eq1}, we obtain 
\[
b^2 (b^2/4 - 1) I_{2(s-m-1)} - \Lambda^2 =  Z Z^\tp,\quad \Lambda Z = 0,\quad Z^\tp Z =b^2 (b^2/4 - 1) I_2.
\]
In particular, $b \ge 2$ and $\rank (b^2 (b^2/4 - 1) I_{2(s-m-1)} - \Lambda^2) \le 2$. We recall that there exist some $Q\in \O_{2(s-m-1)}(\mathbb{R})$, and $\lambda_1 \ge \cdots \ge \lambda_{s-m-1} \ge 0$ such that $\Lambda = Q^\tp \diag(\lambda_1 F_2, \dots, \lambda_{s-m-1} F_2) Q$. Hence we have $b^2 (b^2/4 - 1) + \lambda_p^2 = 0$ whenever $p \ge 2$. If $s-m-1 \ge 2$ then
\[
\lambda_2 =\cdots = \lambda_{s-m-1} = b \left( \frac{b^2}{4} - 1\right) = 0.
\]
Since $Z^\tp Z = b^2 (b^2/4 - 1) I_2 = 0$, we conclude that $Z = 0$, $b = 2$ and $\lambda_1 = 0$
 which implies $\Lambda = 0$. If $s - m - 1 = 1$ then $Z = \begin{bsmallmatrix}
x & y \\
y & -x
\end{bsmallmatrix}$ and we have $Z Z^\tp = (\lambda_1^2 + b^2 (b^2/4 - 1) ) I_2$. It is clear that we again have $\lambda_1 = 0$, $\Lambda = 0$ and $x^2 + y^2 = b^2 (b^2/4 - 1)$.
\item $(p_s,q_s) = (1,2)$, $X_s = J_3(0)$, $\kappa_s = -1$, $B_s = F_3$: let $Y' = (Y_{pq})$ for $p,q\in \{1,\dots,m, s\}$. By Theorem~\ref{thm:structure}, $Y'$ has the form: 
\[
Y' = \begin{bsmallmatrix}
I_m + \Lambda & x & 0 & 0 \\
0 & &  &\\
0 &  & I_3 + J_3(0)^2 &\\
-x^\tp &  & & 
\end{bsmallmatrix},\quad \Lambda \in \mathfrak{o}_m(\mathbb{R}),\quad x\in \mathbb{R}^m.
\]
Then \eqref{thm:structure:eq1} leads to a contradiction:
\[
\begin{bsmallmatrix}
0 & 0 & 0 \\
0 & 0  & 0 \\
0 & 0 & 1 + x^\tp x
\end{bsmallmatrix} = \begin{bsmallmatrix}
0 & 0 & 1 \\
0 & -1  & 0 \\
1 & 0 & 0
\end{bsmallmatrix}.
\]
\item $(p_s,q_s) = (2,2)$: by Table~\ref{Tab:indecomposable}~No.~4, there are three possibilities for $(X_s,B_s)$:  
\small{
\[
(\diag(J_2(\lambda), - J_2(\lambda)^\tp), I_2 \otimes H_2),\quad (J_2\left( \begin{bsmallmatrix}
0 & b \\
-b & 0
\end{bsmallmatrix}
 \right), \kappa_s F_2 \otimes F_2),\quad 
\left( \diag \left( \begin{bsmallmatrix}
a & b \\
-b & a
\end{bsmallmatrix}, \begin{bsmallmatrix}
a & b \\
-b & a
\end{bsmallmatrix} \right), I_2 \otimes H_2 \right),
\]}\normalsize
where $\lambda \ge 0, a,b > 0, \kappa_s = \pm 1$. We claim that the latter two are impossible. Indeed, if $(X_s, B_s) = (J_2\left( \begin{bsmallmatrix}
0 & b \\
-b & 0
\end{bsmallmatrix}
 \right), \kappa_s F_2 \otimes F_2)$, then we let $Y' = (Y_{pq})$ for $p,q\in \{m+1,\dots, s\}$ and Theorem~\ref{thm:structure} implies 
\[
Y' = \begin{bsmallmatrix}
(1-b^2/2)I_{2(s-m-1)} + \Lambda & Z & 0 \\
0 & ((1-b^2/2)) I_2 &  0 \\
-\kappa_s (Z F_2)^\tp & b F_2 & (1-b^2/2) I_2 
\end{bsmallmatrix} ,
\]
where $\Lambda\in \mathfrak{o}_{2(s-m-1)}(\mathbb{R})$, $Z =\begin{bsmallmatrix}
Z_1 \\
\vdots \\
Z_{s-m-1}
\end{bsmallmatrix} \in \mathbb{R}^{2(s-m-1) \times 2}$ and each $Z_p = \begin{bsmallmatrix}
x_p & y_p  \\
y_p & -x_p
\end{bsmallmatrix}$. Then \eqref{thm:structure:eq1} implies $\kappa_s (1 - b^2/2)^2 F_2 = \kappa_s F_2$ which forces $b = 0$. If $(X_s,B_s) = \left( \diag \left( \begin{bsmallmatrix}
a & b \\
-b & a
\end{bsmallmatrix}, 
-\begin{bsmallmatrix}
a & b \\
-b & a
\end{bsmallmatrix}^\tp \right), I_2 \otimes H_2 \right)$, Theorem~\ref{thm:structure} implies $(I_4 + 1/2 X_s^2) (I_2 \otimes H_2) (I_4 + 1/2 X_s^2)^\tp = I_2 \otimes H_2$, from which we obtain 
\[
(1 + (a^2-b^2)/2)^2 - a^2 b^2 = 1,\quad (1 + (a^2 - b^2)/2) a b = 0.
\]
This forces $a = 0$ contradicting to the assumption $a > 0$. 

Thus, it is sufficient to consider $(X_s,B_s) = (\diag(J_2(\lambda), - J_2(\lambda)^\tp), I_2 \otimes H_2)$ where $\lambda \ge 0$. Theorem~\ref{thm:structure} again implies that $\lambda > 0$ is not possible. Therefore, we let $\lambda = 0$ and $Y' = (Y_{p,q})$ for $p,q\in \{1,\dots, m,s\}$. Moreover, the matrix $R$ in \eqref{thm:classification O_n2:eq1} satisfies $R \diag (I_{n-2}, Q_4) \in \O_{n,2}(\mathbb{R})$. Theorem~\ref{thm:structure} ensures that we can write 
\[
Y' = \begin{bsmallmatrix}
I_m + \Lambda & z & 0 & 0 & w \\
0 & 1 + b & 0   & 0 & 0 \\
-w^\tp & a  & 1 - b & 0 & 0 \\
-z^\tp & 0 & 0  & 1-b  & -a  \\
0 & 0 & 0 & 0 & 1+b
\end{bsmallmatrix},\quad a,b\in \mathbb{R},\quad z,w\in \mathbb{R}^m.
\]
By \eqref{thm:structure:eq1} we obtain $\Lambda = 0$, $w = z = 0$ and $b = 0$.
\item $(p_s,q_s) = (3,2)$, $X_s = J_5(0)$, $\kappa_s = 1$, $B_s = F_5$: let $Y' = (Y_{p,q})$ for $p,q\in \{1,\dots, m, s\}$. By Theorem~\ref{thm:structure}, we write 
\[
Y' = \begin{bsmallmatrix}
I_m + \Lambda & z & 0 & 0 & 0 & 0 \\
0 & 1 & 0 & 0 & 0 & 0\\
0 & \frac{1}{2} & 1 & 0 & 0 & 0 \\
0 & 0 &  \frac{1}{2} & 1 & 0 & 0 \\
0 & 0 & 0 &  \frac{1}{2} & 1 & 0 \\
-z^\tp & 0 & 0 & 0 &  \frac{1}{2} & 1
\end{bsmallmatrix},\quad \Lambda \in \mathfrak{o}_m(\mathbb{R}), \quad z \in \mathbb{R}^m.
\]
We obtain a contradictory relation $z^\tp z + 1/4 = 0$ from \eqref{thm:structure:eq1}, thus $(p_s,q_s) = (3,2)$ is not possible.
\item $(p_s,q_s) = (4,2)$, $X_s = J_3\left( \begin{bsmallmatrix} 0 & b \\
-b & 0
\end{bsmallmatrix} \right)$, $b > 0$, $\kappa_s = 1$, $B_s = -I_2 \otimes F_3$: let $Y' = (Y_{p,q})$ for $p,q\in \{m+1,\dots, s\}$. Theorem~\ref{thm:structure} implies 
\[
Y' = \begin{bsmallmatrix}
\left( 1 - \frac{b^2}{2} \right) I_{2(s-m-1)} + \Lambda & Z & 0 & 0 \\
0 & \left( 1 - \frac{b^2}{2} \right) I_{2}  & 0 & 0  \\
0 & b F_2  & \left( 1 - \frac{b^2}{2} \right) I_{2} & 0   \\
Z^\tp  & \frac{1}{2} I_2  & b F_2 & \left( 1 - \frac{b^2}{2} \right) I_{2}
\end{bsmallmatrix},
\]
where $\Lambda\in \mathfrak{o}_{s - m - 1}(\mathbb{R})$, $Z = \begin{bsmallmatrix} 
Z_1 \\
\vdots \\
Z_{s-m-1}
\end{bsmallmatrix} \in \mathbb{R}^{2(s-m-1) \times 2}$ and each $Z_p = \begin{bsmallmatrix} 
x_p & y_p \\
y_p & - x_p
\end{bsmallmatrix}$. By \eqref{thm:structure:eq1}, we obtain $-(1-b^2/2)^2 I_2 = -I_2$ which indicates $b = 0$.  
\end{enumerate}
\end{proof}
\bibliographystyle{abbrv}
\bibliography{new}

\end{document}